\documentclass[11pt]{amsart}
\usepackage{mathtools, amssymb,amsfonts,amsthm,amscd,amsmath,amsgen,upref,enumerate,nicefrac,pdfpages}

\usepackage[margin=1in]{geometry}

\usepackage[colorinlistoftodos,prependcaption,color=yellow,textsize=tiny]{todonotes}

\theoremstyle{plain}
\newtheorem{cor}{Corollary}
\newtheorem{lem}[cor]{Lemma}
\newtheorem{prop}[cor]{Proposition}

\newtheorem{thm}[cor]{Theorem}
\newtheorem*{thm*}{Theorem}
\theoremstyle{definition}
\newtheorem{definition}[cor]{Definition}
\newtheorem{remark}[cor]{Remark}

\newtheorem{assumption}[cor]{Assumption}

\numberwithin{cor}{section}
\numberwithin{equation}{section}

\DeclareMathOperator{\C}{C}

\DeclareMathOperator{\Supp}{Supp}

\newcommand{\8}{\infty}

\newcommand{\E}{\mathbb{E}}

\newcommand{\esssup}{\text{ess sup}}

\renewcommand{\d}{\delta}

\renewcommand{\and}{\quad\textrm{ and }\quad}

\renewcommand{\P}{\mathbb{P}}

\renewcommand{\o}{\omega}
\renewcommand{\O}{\Omega}

\newcommand{\sgn}{\text{sgn}}

\newcommand{\Ent}{\text{Ent}}

\newcommand{\F}{\mathcal F}

\newcommand{\mcS}{\mathcal{S}}

\newcommand{\R}{\mathbb{R}}
\newcommand{\TT}{\mathbb{T}}

\newcommand{\N}{\mathbb{N}}
\newcommand{\Z}{\mathbb{Z}}

\newcommand{\norm}[1]{\left\| #1 \right\|}

\newcommand{\D}{\Delta}

\newcommand{\s}{\sigma}

\newcommand{\ve}{\varepsilon}

\newcommand{\abs}[1]{\left|#1\right|}
\providecommand{\ud}[1]{\, \mathrm{d} #1}
\providecommand{\dx}{\ud{x}}
\providecommand{\dy}{\ud{y}}
\providecommand{\dxi}{\ud \xi}
\providecommand{\deta}{\ud{\eta}}
\providecommand{\dr}{\ud{r}}

\providecommand{\dxp}{\ud{x'}}
\providecommand{\dxip}{\ud{\xi'}}
\providecommand{\dyp}{\ud{y'}}
\providecommand{\ds}{\ud{s}}
\providecommand{\dt}{\ud{t}}
\providecommand{\dz}{\ud{z}}
\providecommand{\dd}{\ud}

\def\XXint#1#2#3{{\setbox0=\hbox{$#1{#2#3}{\int}$ }
\vcenter{\hbox{$#2#3$ }}\kern-.6\wd0}}

\usepackage{graphicx}
\usepackage{color}

\author{Benjamin Fehrman}
\address{Mathematical Institute, University of Oxford, OX2 6GG Oxford, United Kingdom}
\email{Benjamin.Fehrman@maths.ox.ac.uk}

\author{Benjamin Gess}
\address{Max Planck Institute for Mathematics in the Sciences, 04103 Leipzig, Germany \newline \indent Fakult\"at f\"ur Mathematik, Universit\"at Bielefeld, 33615 Bielefeld, Germany}
\email{Benjamin.Gess@mis.mpg.de}


\date{\today}

\makeatother

\begin{document}

\title{Non-equilibrium large deviations and parabolic-hyperbolic PDE with irregular drift}
\begin{abstract}
Large deviations of conservative interacting particle systems, such as the zero range process, about their hydrodynamic limit and their respective rate functions lead to the analysis of the skeleton equation; a degenerate parabolic-hyperbolic PDE with irregular drift.  We develop a robust well-posedness theory for such PDEs in energy-critical spaces based on concepts of renormalized solutions and the equation's kinetic form.  We establish these properties by proving that renormalized solutions are equivalent to classical weak solutions, extending concepts of [DiPerna, Lions; Ann.\ Math., 1989], [Ambrosio; Invent.\ Math., 2004] to the nonlinear setting.

The relevance of the results toward large deviations in interacting particle systems is demonstrated by applications to the identification of l.s.c.\ envelopes of restricted rate functions, to zero noise large deviations for conservative (singular) SPDE, and to the $\Gamma$-convergence of rate functions. The first of these solves a long-standing open problem in the large deviations for zero range processes. The second makes rigorous an informal link between the non-equilibrium statistical mechanics approaches of macroscopic fluctuation theory and fluctuating hydrodynamics.
\end{abstract}

\maketitle

\markright{LARGE DEVIATIONS FOR CONSERVATIVE SPDE}

\section{Introduction}

Large deviations of conservative interacting particle systems, such as the zero range process, about their hydrodynamic limit are described by their rate function
\begin{equation}\label{eq:intro-rate-function}
I(\rho) =\frac{1}{2}\inf\left\{\norm{g}^2_{L^2(\TT^d\times[0,T];\R^d)}\colon \partial_{t}\rho=\Delta\Phi(\rho)-\nabla\cdot (\Phi^{\frac{1}{2}}(\rho)g)\right\} \end{equation}
where  $\Phi$ is a monotone nonlinearity. This motivates the analysis of the corresponding skeleton equation, a parabolic-hyperbolic PDE
\begin{equation}
\begin{aligned}\partial_t\rho & =\D\Phi(\rho)-\nabla\cdot(\Phi^\frac{1}{2}(\rho)g) &  & \textrm{in}\;\;\mathbb{T}^d\times(0,T),\end{aligned}
\label{eq:intro-skeleton}
\end{equation}
with irregular drift $g\in L^2(\TT^d\times (0,T);\R^d)$, posed on the $d$-dimensional torus $\mathbb{T}^d$. A primary result of this work is the well-posedness and stability of \eqref{eq:intro-skeleton}.

These results make important contributions to the understanding of large deviations in non-equilibrium conservative fluctuating systems, including applications to the large deviations of the zero range process (Theorem \ref{thm:envelope} below), to large deviations of conservative SPDE (Theorem \ref{prop_collapse_3} below), which rigorously establishes the link between the non-equilibrium statistical mechanics theories of fluctuating hydrodynamics and macroscopic fluctuation theory, and to the $\Gamma$-convergence of the large deviations rate functions (Theorem \ref{thm:gamma-convergence} below).

The results obtained here extend to the nonlinear setting concepts of renormalized solutions, introduced by DiPerna and Lions \cite{DPL89}, Ambrosio \cite{A04-1}, Le Bris and Lions \cite{LBL08}, and combine these with the concept of kinetic solutions introduced by Lions, Perthame, and Tadmor \cite{LPT94} and Chen and Perthame \cite{CP03}.

To appreciate the difficulty in proving well-posedness and stability for \eqref{eq:intro-skeleton} let us consider the model case given by the porous media equation $\Phi(\rho)=\rho^m$, for some $m\in[1,\infty)$. Scaling arguments (see Section~\ref{energy_critical} below) demonstrate that \eqref{eq:intro-skeleton} is critical for controls $g\in L^q([0,T];L^p(\TT^d;\R^d))$ for $q=p=2$ and initial condition in $L^1(\TT^d)$, and is supercritical for initial condition in $L^r(\TT^d)$, for every $r\in(1,\infty)$. Hence, even in the case of independent particles $\ensuremath{\Phi(\rho)=\rho}$, the skeleton equation \eqref{eq:intro-skeleton} with $g\in L^2(\TT^d\times[0,T];\R^d)$ is not of semilinear nature, since also on small scales the diffusive operator does not dominate the convective, irregular term. Consequently, the well-posedness and stability of solutions to \eqref{eq:intro-skeleton} are challenging problems.

In this paper, we obtain a new a-priori estimate, which leads to the existence of weak solutions (see Definition~\ref{classical_weak} below) to \eqref{eq:intro-skeleton} and optimal regularity estimates. In contrast to linear Fokker-Planck equations (see, for example, Le Bris and Lions \cite{LBL04,LBL08}), the supercriticality of \eqref{eq:intro-skeleton} in $L^r(\TT^d)$ for $r\in(1,\infty)$ makes it impossible to exploit $L^r(\TT^d)$-based estimates. Our estimates are based on entropy-entropy dissipation, which require the nonnegativity of the initial data and yield that, for some $c\in(0,\infty)$,
\begin{equation}\label{intro:apriori} \sup_{t\in[0,T]}\int_{\TT^d}\Psi_\Phi(\rho)\dx+\int_0^T\int_{\TT^d}\abs{\nabla  \Phi^\frac{1}{2}(\rho)}^2\dx\dt
 \leq \int_{\TT^d}\Psi_\Phi(\rho_0)+c\int_0^T\int_{\TT^d}\abs{g(x,t)}^2\dx\dt,
\end{equation}
with $\Psi_\Phi'(\xi)=\log(\Phi(\xi))$.  We therefore define the energy space for the initial data to be
\begin{equation}\label{intro_energy_space}\Ent_\Phi(\TT^d)=\left\{\rho_0\in L^1(\TT^d)\colon \rho_0\geq 0\;\;\textrm{and}\;\;\int_{\TT^d}\Psi_\Phi(\rho_0(x))\dx<\infty\right\}.\end{equation}
The proof of uniqueness is significantly more complicated due to the lower integrability of the solution provided by \eqref{intro:apriori}. For this reason, previous techniques such as Otto \cite{O96} do not apply.  Indeed, the arguments of \cite{O96} were based on $L^r(\TT^d)$ estimates, for $r\in(1,\infty)$. 

We instead introduce the concept of a renormalized kinetic solution (see Definition~\ref{skel_sol_def} below) to recover the uniqueness, for which it is necessary to deal with two new difficulties.  First, it is not obvious that weak solutions are renormalized kinetic solutions, a difficulty that should be expected from the linear case (see, for example, DiPerna and Lions \cite{DPL89}, Ambrosio \cite{A04-1}).  In comparison to the linear case, additional commutator errors appear due to commuting convolutions and nonlinearities in \eqref{eq:intro-skeleton}, and the commutator estimates have to be based on the regularity $\Phi^\frac{1}{2}(\rho)$ implied by \eqref{intro:apriori} rather than on regularity of $\rho$ itself.  Second, even on the renormalized level, standard techniques for uniqueness (see, for example, \cite{ChenPerthame}) do not apply, since the estimate \eqref{intro:apriori} does not imply the decay of entropy and parabolic defect measures at infinity.

The three main theorems below are obtained under general assumptions that apply, in particular, to the porous media nonlinearity $\Phi(\xi)=\xi^m$ for every $m\in[1,\infty)$.  The first main result is the well-posedness and stability of the skeleton equation.

\begin{thm*}[Theorem \ref{thm_unique}, Theorem \ref{equiv}, Proposition \ref{gen_exist}, Proposition \ref{weak_strong} below] Let $\Phi\in\C([0,\infty))\cap\C^1_{\textrm{loc}}((0,\infty))$ satisfy Assumptions~\ref{as_unique},  \ref{as_equiv}, and \ref{as_compact} below. Then,
\begin{enumerate}[(1)]
\item For every $g\in L^{2}(\TT^d\times [0,T];\R^d)$ and $\rho_{0} \in\Ent_\Phi(\mathbb{T}^d)$ there exists a unique weak solution to \eqref{eq:intro-skeleton}. Moreover, if $\rho^{1},\rho^{2}$ are weak solutions of \eqref{eq:intro-skeleton} with initial data $\rho_{0}^{1},\rho_{0}^{2}\in\Ent_\Phi(\TT^d)$ and the same control $g\in L^2(\TT^d\times(0,T);\R^d)$,
\[
\norm{\rho^{1}-\rho^{2}}_{L^{\infty}([0,T];L^{1}(\mathbb{T}^{d}))}\leq\norm{\rho_{0}^{1}-\rho_{0}^{2}}_{L^{1}(\mathbb{T}^{d})}.
\]
\item Let $\rho^n_0,\rho_0\in\Ent_\Phi(\TT^d)$ and $g_n,g\in L^{2}([0,T];L^2(\TT^{d};\R^d))$ for every $n\in\N$ with $\rho^n_0\rightharpoonup \rho_0$ and $g_{n} \rightharpoonup g$ weakly in $L^1(\TT^d)$ and $L^{2}([0,T];L^2(\TT^{d};\R^d)$ respectively.  Let $\rho_{n},\rho \in L^{1}([0,T];L^{1}(\TT^{d}))$ solve \eqref{eq:intro-skeleton} with controls $g_{n},g$ and initial datum $\rho^n_0,\rho_0$ respectively. Then, 
\[
\rho_{n}\to\rho\;\;\textrm{strongly in}\;\;L^{1}([0,T];L^{1}(\TT^d)).
\]
\end{enumerate}
\end{thm*}

Having established the well-posedness and stability of the skeleton equation, we next describe the consequences of our results for the non-equilibrium large deviations of interacting particle systems.

The hydrodynamic limit of the empirical density field $\mu^n$ of the zero range process on the torus with mean local jump rate $\Phi$ is the solution to the nonlinear diffusion equation
\[\partial_t\rho = \Delta\Phi(\rho)\;\;\textrm{on}\;\;\TT^d\times(0,T),\]
see, for example, Ferrari, Presutti, and Vares \cite{FPV83}, Kipnis and Landim \cite{KL99}. The particle system will exhibit large fluctuations about this limit which, though infrequent, can have catastrophic effects, such as earthquakes or mechanical failure (see, for example,  \cite{ERVE07,GVE18,VEW12}). It is thus important to understand and simulate such rare events. On an exponential scale, the (im)probability of a large fluctuation is described in terms of a large deviations principle with rate function $I$. 

The general approach to large deviations in interacting particle systems, introduced in the seminal works \cite{DV89,KOV89}, relies on two ingredients to be verified in a case by case manner: First, the so-called superexponential estimates, second, the identification of the lower semicontinuous (l.s.c.) envelope of the rate function restricted to smooth and strictly positive fluctuations, see \cite[Chapter 10]{KL99}. This second ingredient is required, since, in a first step, the large deviations lower bound can be established only for nice enough fluctuations.  More precisely, one first restricts to the space  $\mcS$ of positive fluctuations $\rho$ so that 
\[\partial_{t}\rho=\D\Phi(\rho)-\nabla\cdot(\Phi(\rho)\nabla H),\]
for some strictly positive $\rho_0\in C^\infty(\TT^d)$ and $H\in C^{3,1}(\TT^d\times[0,T])$. The extension to a full large deviations result requires the identification of the l.s.c.\ envelope of the rate function \eqref{eq:intro-rate-function} restricted to $\mcS$. For the symmetric simple exclusion process, this envelope can be identified, due to the convexity of the rate function \cite[Chapter 10, Lemma 5.5]{KL99}; a property not available for the zero range process. For this reason, for the zero range process, so far, only a restricted large deviations estimate is known, in the sense that the lower bound is known only for nice enough fluctuations, see  \cite[Theorem 1, and the discussion on p.\ 66]{BKL95}. The identification of the l.s.c.\ envelope has remained an open problem for over twenty years since then, and is resolved in the present work on the torus.

\begin{thm*} [Theorem~\ref{thm:envelope} below]  Let $\Phi\in\C([0,\infty))\cap\C^1_{\textrm{loc}}((0,\infty))$ satisfy Assumptions~\ref{as_unique},  \ref{as_equiv}, and \ref{as_compact} below.  Then, the rate function \eqref{eq:intro-rate-function} is equal to the lower semicontinuous envelope of its restriction to the space $\mcS$.

\end{thm*}

Since the assumptions on $\Phi$ in this theorem include the case of degenerate $\Phi'$, e.g.\ $\Phi(\rho)=\rho^m$ for every $m\in[1,\infty)$, its applicability goes much beyond the fluctuations of the zero range process. We next demonstrate this by its application to the large deviations for degenerate stochastic PDEs with conservative noise, and their relation the fluctuations of interacting particle systems.

Macroscopic fluctuation theory (MFT) introduces a general framework for non-equilibrium diffusive systems (see, for example, Bertini, De Sole, Gabrielli, Jona Lasinio, and Landim \cite{BDSGJLL15}, Derrida \cite{D07}), thus extending Onsager-like near to equilibrium theories for non-equilibrium thermodynamics. MFT, which is based on a constitutive formula for large fluctuations around thermodynamic variables, like density and current, can be justified by fluctuating hydrodynamics (see, for example, Hohenberg, and Halperin \cite{HH77}, Landau, and Lifshitz \cite{LL87}, Spohn \cite{S91}, Bouchet, Gaw\c edzki, and Nardini \cite{BGN16}). The latter postulates a conservative, singular stochastic PDE describing fluctuations in systems out of equilibrium. The fundamental ansatz of MFT can then be obtained, informally, from fluctuating hydrodynamics as the zero noise large deviations principle for this stochastic PDE. In addition to this conceptual relevance, the relation between zero noise large deviations for conservative stochastic PDE and MFT may serve as the basis for the development of importance sampling techniques to numerically simulate rare events in systems far from equilibrium (see, for example, E, Ren, and Vanden-Eijnden \cite{ERVE07}, Grafke and Vanden-Eijnden \cite{GVE18}, Vanden-Eijnden and Weare \cite{VEW12}).  We provide a rigorous link between the zero-noise large deviations of conservative stochastic PDE and the large deviations rate functions appearing in interacting particle systems, including zero range processes, the latter being an example for MFT.

The rate function \eqref{eq:intro-rate-function} is informally connected to the small-noise large deviations ($\ve\to0,K\to\infty$) of the stochastic PDE
\begin{equation}
\partial_{t}\rho=\Delta\Phi(\rho)-\sqrt{\ve}\nabla\cdot(\Phi^{\frac{1}{2}}(\rho)\xi^K)\quad\textrm{in}\;\;\TT^d\times(0,\infty),\label{intro_eq-1}
\end{equation} 
where $\xi^K$ is an $\R^d$-valued  spatially correlated noise, converging to space-time white noise for $K\to \infty$, see, for example, Dirr, Stamatakis, and Zimmer \cite{DSZ16} and Giacomin, Lebowitz, and Presutti \cite[Section 4]{GLP99}. In particular, in the ultra-violet limit $K\to\infty$, this observation links the case of independently diffusing particles to the Dean-Kawasaki equation
\begin{equation}
\partial_{t}\rho=\Delta\rho-\sqrt{\ve}\nabla\cdot (\rho^{\frac{1}{2}}\xi)\quad\textrm{in}\;\;\TT^d\times(0,\infty),\label{eq:introDK}
\end{equation}
which has been considered, for example, by Dean \cite{D96}, Kawasaki \cite{K94}, Donev, Fai, and Vanden-Eijnden \cite{DFVE14}, Donev, and Vanden-Eijnden \cite{DVE14}, Konarovskyi and von Renesse \cite{KR15,KR17}, and Lehmann, Konarovskyi, and von Renesse \cite{LKR19}.  

This relationship, however, had to remain informal because \eqref{intro_eq-1} has only recently been shown to be well-posedness by the authors in \cite{FehGes21}. In addition, the ultra-violet limit $K\to\infty$ of \eqref{intro_eq-1} is supercritical in the language of singular stochastic PDE and regularity structures (see, for example, Hairer \cite{H13-3}, Gubinelli, Imkeller, and Perkowski \cite{GIP12}), and therefore falls outside the scope of the theory. For a discussion of this issue and the appearance of ultraviolet divergences in fluctuating hydrodynamics see \cite[p.\ 595]{BDSGJLL15}. Furthermore, it has been observed in \cite{LKR19} that a renormalization is necessary in order to obtain function-valued solutions of \eqref{eq:introDK}.  These renormalization terms destroy the relationship between the stochastic PDE and the zero range process, since they appear in the rate function for \eqref{eq:introDK} and thereby lead to incorrect predictions of rare events for the particle process. We refer to Hairer and Weber \cite{HW15} for a discussion of related aspects in the context of the (singular) stochastic quantization equation.

To address the issue of renormalization, the main result of this work identifies a joint scaling regime $\ve<<\nicefrac{1}{K(\ve)}$ such that possible renormalization constants vanish, and such that the solutions $\rho^{\ve,K(\ve)}$ constructed in \cite{FehGes21} correctly simulate the rare events in the particle system by satisfying a large deviations principle with rate function \eqref{eq:intro-rate-function}.

\begin{thm*}[Theorem~\ref{prop_collapse_3} below]  Let $\Phi\in\C([0,\infty))\cap\C^1_{\textrm{loc}}((0,\infty))$ satisfy Assumptions~\ref{as_unique},  \ref{as_equiv}, and \ref{as_compact} below.  Then, in the $\ve\rightarrow 0$ scaling regime
\[\ve K(\ve)^{d+2}\rightarrow 0\;\;\textrm{and}\;\;K(\ve)\rightarrow\infty,\]
the $\{\rho^{\ve,K(\ve)}\}_{\ve\in(0,1)}$ solving \eqref{intro_eq-1} satisfy a large deviations principle with rate function \eqref{eq:intro-rate-function} uniformly with respect to weakly $L^1(\TT^d)$-compact subsets of $\Ent_\Phi(\TT^d)$.
\end{thm*}

A detailed comparison to known results on large deviations for conservative stochastic PDE is given in Section \ref{sec:literature} below.

A simplified version of these techniques prove that, for every fixed $K\in\N$, the solutions of \eqref{intro_eq-1} satisfy a small-noise large deviations principle with rate function
\begin{equation}
I^K(\rho)=\frac{1}{2}\inf\left\{\norm{g}^2_{L^2(\TT^d\times[0,T];\R^d)}\colon \partial_{t}\rho=\Delta\Phi(\rho)-\nabla\cdot (\Phi^{\frac{1}{2}}(\rho)P_Kg)\right\},\label{finite_rate_function-1-1}
\end{equation}
for $P_Kg\in L^2([0,T];L^2(\TT^d;\R^d))$ the Fourier projection of $g\in L^2([0,T];L^2(\TT^d;\R^d))$ onto the Fourier modes of the noise $\xi^K$.  The final result of this work is the $\Gamma$-convergence of these rate functions, as $K\rightarrow\infty$, to the rate function \eqref{eq:intro-rate-function}.

\begin{thm*} [Theorem~\ref{thm:gamma-convergence} below]  Let $\Phi\in\C([0,\infty))\cap\C^1_{\textrm{loc}}((0,\infty))$ satisfy Assumptions~\ref{as_unique},  \ref{as_equiv}, and \ref{as_compact} below.  Then, as $K\rightarrow\infty$,
\[
I^K\overset{\Gamma}{\longrightarrow}I,
\]
in the sense of $\Gamma$-convergence.  \end{thm*}

\subsection{Comments on the literature}\label{sec:literature}\ \\ 

\textit{{Fluctuations of the zero range process}} have been analyzed, for example, by Benois, Kipnis, and Landim in \cite{BKL95}. We also refer to Kipnis and Landim \cite{KL99}, Evans and Hanney \cite{EH05} and the references therein for a detailed account of the theory. Equilibrium large deviations for the  zero-range process with long jumps and reservoirs have been analyzed by Bernardin, Gon\c{c}alves, Oviedo-Jim\'enez, and Scotta \cite{BGOS22}. For large deviations results in mean-field interacting particle systems we refer to the works by Barr\'e, Bernardin, Ch\'etrite, Chopra, and Mariani \cite{Bar2019}, Dawson and G\"artner \cite{DG89,DG87}, and Gvalani and Schlichting \cite{GS20}.

\textit{{Large deviations for conservative stochastic PDE}} have previously been considered in the works \cite{M10} by Mariani and \cite{BBMN10} by 
Bellettini, Bertini, Mariani, and Novaga. We here compare the methods and results when applied in the framework of conservative SPDE
\[
\partial_{t}\rho=\Delta\Phi(\rho)-\sqrt{\ve}\nabla\cdot(\s(\rho)\xi^{K}).
\]
We emphasize, however, that the scope of the works \cite{M10,BBMN10} goes beyond this by including the case of asymptotically vanishing dissipation, which heuristically corresponds to asymmetric simple exclusion processes. In this sense, the present work extends the first order large deviations obtained in \cite{M10,BBMN10} in two directions:  First, the rate functions are identified on the space of function-valued solutions. Second, this is achieved under significantly more general conditions on the coefficients, which include nonlinear dissipation $\Phi$ and degenerate noise $\s$.  Since we allow for unbounded initial data and $\s$, the methods developed here handle possible concentration and vacuum effects. Heuristically, this corresponds to potentially degenerate zero range processes, as compared to regularized exclusion processes in \cite{M10,BBMN10}.

In \cite{M10,BBMN10} the rate function could be identified only after passing to a measure-valued formulation for which the skeleton equation is linear, and the corresponding rate function is convex.  This significantly simplifies the identification, for example, since the linear structure is nicely compatible with mollification---a fact which is not true in the nonlinear setting, and the control of the resulting nonlinear commutators is a key technical step in the present work.  The linearization also comes at the cost of working on much larger spaces, and on the level of function-valued large deviations no explicit representation of the rate function has been obtained, compare \cite[Corollary 1]{M10}.  This issue is fully resolved in the present work. We present an approach that entirely avoids passing to measure-valued solutions, and therefore establishes an LDP with an explicit rate function for the function-valued solutions that is consistent with the LDP for the particle system.

The second main generalization of \cite{M10,BBMN10} lies in the assumptions on the coefficients. In \cite{M10} it is assumed that $\sigma$ is $\C^2$-smooth with $\sigma(0)=\sigma(1)=0$ and that $0\leq \rho_{0}\leq 1$.  As a consequence, solutions $\rho$ also take values in the interval $[0,1]$, which rules out concentration and vacuum phenomena. As demonstrated in \cite{DFG20}, this significantly simplifies the analysis of the large deviations, since the solutions are bounded and satisfy standard $L^{p}$-based regularity estimates. In addition, simple exclusion processes and zero range processes heuristically correspond to singular diffusion coefficients $\s(\rho)=\sqrt{\rho(1-\rho)}$ and $\s(\rho)=\sqrt{\rho}$ respectively, thus violating the $C^2$ assumption in \cite{M10,BBMN10}. The results of the present work include also these singular cases.

\textit{{Stochastic PDE with conservative noise}} have been considered in the framework of stochastic scalar conservation laws by Lions, Perthame, and Souganidis \cite{LPS13-2,LPS13,LPS14}, Friz and Gess \cite{FG16}, Gess and Souganidis \cite{GS14,GS17}.  Most recently, based on the equation's kinetic formulation,  the authors \cite{FehGes21} treated a general class of parabolic-hyperbolic stochastic PDE with nonlinear conservative and multiplicative noise, including the Dean-Kawasaki and nonlinear Dawson-Watanabe equations.  Earlier works studying stochastic PDEs with conservative noise include Gess and Souganidis \cite{GS16-2}, Fehrman and Gess \cite{FG17} and Dareiotis and Gess \cite{DG18}. Approaches to their numerical treatment have been developed by Gess, Perthame, and Souganidis \cite{GPS16}, and Hoel, Karlsen, Risebro, and Storrosten \cite{HKRS17,HKRS18}. Related results on the level of stochastic Hamilton-Jacobi equations, based on entirely different methods, have been developed by Lions and Souganidis \cite{LS98-2,LS98,LS00,LS00-2} (see Souganidis \cite{S18} for a recent account on the theory), with extensions by Friz, Gassiat, Lions, and Souganidis \cite{FGLS16}.  A detailed analysis of fine properties of solutions has been given by Gassiat and Gess \cite{GG19}, Gassiat, Gess, Lions, and Souganidis \cite{GGLS18}, and Lions and Souganidis \cite{LS19}.

\textit{{Large deviation estimates for singular stochastic PDE}} have been derived by Faris and Jona-Lasinio \cite{FJL82}, Jona-Lasinio and Mitter \cite{JLM90}, Cerrai and Freidlin \cite{CF11}, and Hairer and Weber \cite{HW15} in the context of stochastic Allen-Cahn equations. In particular, we emphasize that in \cite{HW15} it is observed that renormalization constants may enter the rate function in the setting of singular stochastic PDE. The results treated in these works are quite different from the present paper, since, due to the additive noise structure of the stochastic Allen-Cahn equation, the treatment of the corresponding skeleton equation does not pose a major difficulty.   

\textit{{The weak convergence approach to large deviation principles}} for stochastic PDE goes back to Budhiraja, Dupuis and Maroulas \cite{BudDupMar2008}, and has been used to derive large deviation estimates for singular stochastic PDE by Cerrai and Debussche \cite{CD19}. Further applications to stochastic PDE with multiplicative noise, depending only on the values of the solution, have been given by Brze\'zniak, Goldys, Jegaraj \cite{BGJ17} and Dong, Wu, Zhang, and Zhang \cite{DWZZ20}.

\textit{{The Dean-Kawasaki equation}} has been introduced in  Dean \cite{D96}, Kawasaki \cite{K94}, and analyzed in Donev, Fai, and Vanden-Eijnden \cite{DFVE14}, Donev and Vanden-Eijnden \cite{DVE14}, Konarovskyi and von Renesse \cite{KR15,KR17}, Lehmann, Konarovskyi, and von Renesse \cite{LKR19} and in the references therein. For an informal treatment of the link of macroscopic fluctuation theory (MFT) to fluctuating hydrodynamics in the context of the Dean-Kawasaki equation see Bouchet, Gaw\c edzki and Nardini \cite{BGN16}. The general MFT has received considerable attention, with a comprehensive overview given by Bertini, De Sole, Gabrielli, Jona-Lasinio, and Landim in \cite{BDSGJLL15}.

\textit{{Linear Fokker-Planck equations with irregular coefficients}}, their well-posedness and stability of solutions have received considerable attention. We refer to Bogachev, Krylov, R\"ockner, and Shaposhnikov \cite{BKRS15} for a detailed account of the literature and to Le Bris and Lions \cite{LBL04,LBL08}, Ambrosio \cite{A04-1}, Boccardo, Orsina, and Porretta \cite{BOP03}, and Porretta \cite{P15} for approaches relying on renormalization.

\subsection{Overview}  Extended lecture notes have been posted online by the authors \cite{FehGes19}.  So, while we emphasize that this version of the paper contains full details, extended remarks and more detailed computations can be found online.  In particular, we prove in \cite[Appendix~A]{FehGes19} that every assumption is satisfied by the model example $\Phi(\xi)=\xi^m$, for every $m\in[1,\infty)$, and therefore that every result in this paper applies to the case of the Laplacian and to every porous media nonlinearity.

We introduce the assumptions on the nonlinearity $\Phi$ at the beginning of each section.  In Section~\ref{skeleton}, we present an informal analysis of the skeleton equation broken down as follows.  In Section~\ref{energy_critical}, we argue by scaling that the skeleton equation \eqref{eq:intro-skeleton} is energy critical for $L^1(\TT^d)$ and is energy supercritical for $L^p(\TT^d)$, if $p\in(1,\infty)$.  We obtain formal a-priori estimates for the solution of \eqref{eq:intro-skeleton} in Section~\ref{sec_a_priori} and thereby identify the correct energy space \eqref{intro_energy_space} for the initial data.  Based on these estimates, in Section~\ref{renormalized_sol} we define a \emph{renormalized kinetic solution} (see Definition~\ref{skel_sol_def}).

In Section~\ref{sec_unique}, we prove the uniqueness of renormalized kinetic solutions (see Theorem~\ref{thm_unique}).  In Section~\ref{sec_equiv}, we prove in Theorem~\ref{equiv} the equivalence of renormalized kinetic solutions and weak solutions (see Definition~\ref{classical_weak}) to the skeleton equation.  In Section~\ref{sec_existence}, we prove the existence of renormalized kinetic solutions (see Proposition~\ref{gen_exist}), and obtain in Proposition~\ref{weak_strong} the strong continuity of the solutions with respect to weak convergence of the controls.  In Section~\ref{sec_uniform_ldp}, we prove the uniform large deviations principle, which relies on the weak approach to large deviations developed in \cite{BudDupMar2008}.  In Section~\ref{sec_gamma}, in Theorem~\ref{thm:gamma-convergence}, we prove the $\Gamma$-convergence of the large deviations rate functions \eqref{finite_rate_function-1-1} to the rate function \eqref{eq:intro-rate-function}.  In Section~\ref{sec_lsc_envelope}, we characterize the l.s.c.\ envelope of the restricted rate function.

\section{The skeleton equation}\label{skeleton}

The equation defining the large deviations rate function of the zero range process \eqref{eq:intro-rate-function} is the so-called \emph{skeleton equation}.  In Sections~\ref{sec_unique}, \ref{sec_equiv}, and \ref{sec_existence} we will prove the existence and uniqueness of renormalized kinetic solutions (see Definition~\ref{skel_sol_def} below), for $\rho_0\in\Ent_\Phi(\TT^d)$ and for $g\in L^2(\TT^d\times[0,T];\R^d)$, to the equation
\begin{equation}\label{skel_eq}\partial_t \rho = \Delta\Phi(\rho)-\nabla\cdot\left(\Phi^\frac{1}{2}(\rho)g\right)\;\; \textrm{in}\;\;\TT^d\times(0,T)\;\;\textrm{with}\;\;\rho(\cdot,0)=\rho_0.
\end{equation}
We first argue formally in Section~\ref{energy_critical} below that equation \eqref{skel_eq} is energy critical for $L^1(\TT^d)$ and energy supercritical for $L^p(\TT^d)$, if $p\in(1,\infty)$.   This argument suggests that no standard $L^p$-theory can be applied, and indeed in Section~\ref{sec_a_priori} below we derive an energy estimate for solutions with initial data in $\Ent_\Phi(\TT^d)$ defined in \eqref{intro_energy_space}.  This estimate will be the basis for Definition~\ref{skel_sol_def} below, where we present the definition of a renormalized kinetic solution.

We observe in particular that the formal estimates obtained in Section~\ref{sec_a_priori} are significantly weaker than are required to apply standard techniques based on the entropy or kinetic formulation of the equation (see \cite{ChenPerthame}).  This can be seen on the level of the parabolic defect measure (see \eqref{parabolic_defect_measure} below), which is neither globally integrable nor decaying at infinity.  The proof of uniqueness therefore requires new techniques to control errors at infinity, and the proof of existence is based on a compactness argument that requires optimal estimates for the solution.

\subsection{Energy criticality}\label{energy_critical}  In this section, we will argue formally that the skeleton equation \eqref{skel_eq} is energy critical for initial data $\rho_0\in L^1(\TT^d)$ and controls $g\in L^2(\TT^d\times[0,T];\R^d)$.  We will consider the case of the porous media equation $\Phi(\xi)=\xi^m$, for some $m\in[1,\infty)$, set on the whole space $\R^d\times[0,\infty)$.  Precisely, for $m,p,q\in[1,\infty)$, $T\in(0,\infty)$, $g\in L^{p}([0,T];L^{q}(\R^d;\R^d))$, and a nonnegative $\rho_0$, suppose that $\rho$ solves
\[\partial_t \rho = \Delta\left(\rho^m\right)-\nabla\cdot\left(\rho^\frac{m}{2}g\right)\;\; \textrm{in}\;\;\R^d\times(0,T)\;\;\textrm{with}\;\; \rho(\cdot,0)=\rho_0.\]
We will ``zoom in'' in the sense that, for positive real numbers $\lambda,\eta,\tau\rightarrow 0$ we consider the rescaling
\[\tilde{\rho}(x,t)=\lambda \rho(\eta x,\tau t).\]
It follows that $\tilde{\rho}$ solves the equation
\[\partial_t \tilde{\rho}  = \left(\frac{\tau}{\eta^2\lambda^{m-1}}\right)\Delta\left(\tilde{\rho}^m\right)-\nabla\cdot\left(\tilde{\rho}^\frac{m}{2}\tilde{g}\right)\;\; \textrm{in}\;\;\R^d\times(0,T)\;\;\textrm{with}\;\;\rho(\cdot,0) =\lambda \rho_0(\eta\cdot),\]
for $\tilde{g}$ defined by
\begin{equation}\label{critical_1}\tilde{g}(x,t)=\left(\frac{\tau}{\eta\lambda^{\frac{m}{2}-1}}\right)g(\eta x, \tau t).\end{equation}
We are interested in understanding the effect of this scaling on the balance between the parabolic and hyperbolic terms.  We preserve the diffusion by fixing
\begin{equation}\label{critical_2}\left(\frac{\tau}{\eta^2\lambda^{m-1}}\right)=1,\end{equation}
and for $r\in[1,\infty)$ we preserve the $L^r(\R^d)$-norm of the initial data by fixing
\begin{equation}\label{critical_3}\lambda=\eta^\frac{d}{r}.\end{equation}
It follows from \eqref{critical_1}, \eqref{critical_2}, and \eqref{critical_3} that
\begin{align*}
\norm{\tilde{g}}_{L^p([0,T];L^q(\R^d;\R^d))} & =\left(\frac{\tau}{\eta\lambda^{\frac{m}{2}-1}}\right)\left(\int_0^T\left(\int_{\R^d}\abs{g(\eta x, \tau t)}^q\dx\right)^\frac{p}{q}\dt\right)^\frac{1}{p}
\\ & = \eta^{1-\frac{d}{p}+\frac{2}{q}+\frac{d}{r}\left(\frac{m}{2}-\frac{m}{q}+\frac{1}{q}\right)}\norm{g}_{L^p([0,T];L^q(\R^d;\R^d))}.
\end{align*}
To ensure that this norm does not diverge as $\eta\rightarrow 0$, we require that
\[1+\frac{d}{r}\left(\frac{m}{2}+\frac{1}{q}\right)\geq \frac{2}{q}+\frac{d}{p}+\frac{dm}{rq}.\]
If $p=q=2$, we conclude that $\nicefrac{d}{2r}\geq\nicefrac{d}{2}$ and therefore that $r=1$.  Conversely, since the lefthand side of this equality is largest for $r=1$, and since the case $r=1$ yields the inequality
\[1+d\left(\frac{m}{2}+\frac{1}{q}\right)\geq\frac{2}{q}+d\left(\frac{1}{p}+\frac{m}{q}\right),\]
we conclude that $p=q=2$ is critical for $L^1(\TT^d)$ and supercritical for $L^r(\R^d)$, for every $r\in(1,\infty)$.

\subsection{A-priori estimates}\label{sec_a_priori}  In this section, we will motivate the definition of a renormalized kinetic solution of the skeleton equation (see Definition~\ref{skel_sol_def} below).  This definition is the foundation of the existence and uniqueness theory to follow.  We will first derive a formal energy estimate for the solution, and thereby identify the correct energy space for the initial data.

We will restrict attention to nonnegative initial data, which is a necessary assumption for the following estimates to be true (see Remark~\ref{rem_nonnegative} below).  Let $T\in(0,\infty)$, let $\rho_0\colon\TT^d\rightarrow[0,\infty)$ be nonnegative, let $g\in L^2(\TT^d\times[0,T];\R^d)$, and let $\rho$ denote the solution
\[\partial_t \rho = \Delta\Phi(\rho)-\nabla\cdot\left(\Phi^\frac{1}{2}(\rho)g\right)\;\; \textrm{in}\;\;\TT^d\times(0,T)\;\; \textrm{with}\;\;\rho(\cdot,0)=\rho_0.\]
Let $\psi\colon [0,\infty)\rightarrow\R$, and define $\Psi\colon[0,\infty)\rightarrow\R$ to be the antiderivative
\[\Psi(\xi)=\int_0^\xi\psi(\xi')\dxp.\]
We test the equation with the composition $\psi(\rho)$ to obtain
\[\partial_t\int_{\TT^d}\Psi(\rho)\dx+\int_{\TT^d}\Phi'(\rho)\psi'(\rho)\abs{\nabla  \rho}^2\dx=-\int_{\TT^d}\Phi^\frac{1}{2}(\rho)\psi'(\rho)g(x,t)\cdot\nabla \rho\dx.\]
The nonnegativity of $\rho$, H\"older's inequality, and Young's inequality imply that
\begin{align*} & \left.\int_{\TT^d}\Psi(\rho)\dx\right|_{t=0}^{t=T}+\int_0^T\int_{\TT^d}\Phi'(\rho)\psi'(\rho)\abs{\nabla  \rho}^2\dx\dt
 \\ & \leq \frac{1}{2}\int_0^T\int_{\TT^d}\abs{g(x,t)}^2\dx\dt+\frac{1}{2}\int_0^T\int_{\TT^d}\Phi(\rho)\psi'(\rho)^2\abs{\nabla  \rho}^2\dx\dt.
 \end{align*}
To close the estimate we require that
\[\Phi(\xi)\psi'(\xi)^2\leq \Phi'(\xi)\psi'(\xi)\;\;\textrm{and hence that}\;\;\psi'(\xi)\leq \frac{\Phi'(\xi)}{\Phi(\xi)}.\]
We therefore fix $\psi_\Phi(\xi)=\log(\Phi(\xi))$ and define
\[\Psi_\Phi(\xi)=\int_0^\xi\log(\Phi(\xi'))\dxp,\]
to conclude that
\begin{equation}\label{formal_est} \left.\int_{\TT^d}\Psi_\Phi(\rho)\dx\right|_{t=0}^{t=T}+\frac{1}{2}\int_0^T\int_{\TT^d}\frac{\Phi'(\rho)^2}{\Phi(\rho)}\abs{\nabla  \rho}^2\dx\dt \leq \frac{1}{2}\int_0^T\int_{\TT^d}\abs{g(x,t)}^2\dx\dt, \end{equation}
where the estimate follows from the identity
\[2\left(\nabla  \Phi^\frac{1}{2}(\rho)\right)=\frac{\Phi'(\rho)}{\Phi^\frac{1}{2}(\rho)}\nabla  \rho.\]

\begin{remark}\label{rem_nonnegative} (see \cite[Remark~2.1]{FehGes19} for full details) Estimate \eqref{formal_est} is in general false for signed initial data, which can be seen for the heat equation.  Indeed, since $\rho(x,t)=x$ solves the heat equation with linear initial data, and since $x^\frac{1}{2}\notin L^2([0,T];H^1_{\textrm{loc}}(\R))$, after localizing this argument to the torus we conclude that the estimate fails in general for signed initial data.\end{remark}

\begin{remark}(see \cite[Remark~2.2]{FehGes19} for full details) We observe that estimate \eqref{formal_est} is based on the physical entropy of the initial data in the case that $\Phi(\xi)=\xi^m$, for some $m\in(0,\infty)$.  In this case,
\[\Psi_\Phi(\xi)=m\int_0^\xi\log(\xi')\dxp=m\left(\xi\log(\xi)-\xi\right),\]
from which it follows from the preservation of the $L^1$-norm that the physical entropy is a nondecreasing function of time.  \end{remark}

\subsection{Renormalized kinetic solutions}\label{renormalized_sol}  In this section, we will define renormalized kinetic solutions of the skeleton equation \eqref{skel_eq}.  Based on estimate \eqref{formal_est}, we first rewrite the equation in the form
\begin{equation}
\label{re_skel_eq} \partial_t \rho = 2\nabla\cdot \left(\Phi^\frac{1}{2}(\rho)\nabla\Phi^\frac{1}{2}(\rho)\right)-\nabla\cdot\left(\Phi^\frac{1}{2}(\rho)g\right)\;\; \textrm{in}\;\;\TT^d\times(0,T)\;\;\textrm{with}\;\;\rho(\cdot,0) =\rho_0.
\end{equation}
On this level, the criticality of the equation can be seen by analyzing the integrability of the products
\begin{equation}\label{rks_1}\Phi^\frac{1}{2}(\rho)\nabla\Phi^\frac{1}{2}(\rho)\;\;\textrm{and}\;\;\Phi^\frac{1}{2}(\rho)g,\end{equation}
which are formally $L^1(\TT^d\times[0,T]);\R^d)$.  Indeed, even in one dimension, embedding theorems do not readily yield an improvement because they do not improve the integrability in time.

The borderline integrability of the products \eqref{rks_1} and the lack of regularity for the solution make classical techniques untenable and suggest the necessity of a generalized solution theory.  We therefore pass to the equation's kinetic formulation.  Let the kinetic function $\overline{\chi}\colon\R^2\rightarrow\R$ be defined by
\[\overline{\chi}(s,\xi)=\mathbf{1}_{\{0<\xi<s\}}-\mathbf{1}_{\{s<\xi<0\}}.\]
Proceeding formally, suppose that $\rho$ is a solution of \eqref{re_skel_eq}.  The kinetic function $\chi$ of $\rho$ is defined for every $(x,\xi,t)\in\TT^d\times\R\times[0,T]$ by
\[\chi(x,\xi,t)=\overline{\chi}(\rho(x,t),\xi)=\mathbf{1}_{\{0<\xi<\rho(x,t)\}}-\mathbf{1}_{\{\rho(x,t)<\xi<0\}}.\]
The identities
\[\nabla_x\chi(x,\xi,t)=\delta_0(\xi-\rho(x,t))\nabla \rho(x,t)\;\;\textrm{and}\;\;\partial_\xi\chi(x,\xi,t)=\delta_0(\xi)-\delta_0(\xi-\rho(x,t)),\]
show formally that the kinetic function $\chi$ of $\rho$ satisfies the equation
\begin{equation}\label{skel_kin_00}\partial_t \chi =\Phi'(\xi)\Delta_x\chi -\left(\partial_\xi \Phi^\frac{1}{2}(\xi)\right)g(x,t)\nabla_x\chi+ \left(\nabla_xg(x,t)\right)\Phi^\frac{1}{2}(\xi)\partial_\xi \chi+\partial_\xi p\;\; \textrm{in}\;\;\TT^d\times\mathbb{R}\times(0,T),\end{equation}
where $\chi(\cdot,0)= \overline{\chi}(\rho_0)$ and where $p$ is the parabolic defect measure on $\TT^d\times\mathbb{R}\times(0,T)$ defined by
\begin{equation}\label{parabolic_defect_measure}p=\delta_0(\xi-\rho) \frac{4\abs{\Phi(\xi)}}{\Phi'(\xi)}\abs{\nabla \Phi^\frac{1}{2}(\rho)}^2.\end{equation}
We rewrite \eqref{skel_kin_00} in the conservative form
\begin{equation}\label{skel_kin} \partial_t \chi =\Phi'(\xi)\Delta_x\chi -\partial_\xi\left(g(x,t)\Phi^\frac{1}{2}(\xi)\nabla_x\chi\right)+\nabla_x\cdot\left(g(x,t)\Phi^\frac{1}{2}(\xi)\partial_\xi \chi\right)+\partial_\xi p\;\;\textrm{in}\;\;\TT^d\times\mathbb{R}\times(0,T). \end{equation}
Based on the estimates leading to \eqref{formal_est}, we will obtain the well-posedness of \eqref{skel_eq} for initial data in the following space.
\begin{definition}\label{def_initial}  Let $\Phi\in\C([0,\infty))\cap\C^1_{\textrm{loc}}((0,\infty))$ be nonnegative.  Let $\Psi_\Phi\colon[0,\infty)\rightarrow[0,\infty)$ be defined by $\Psi_\Phi(\xi)=\int_0^\xi\log(\Phi(\xi'))\dxp$.
We define
\[\Ent_\Phi(\TT^d)=\left\{\rho_0\in L^1(\TT^d)\colon \rho_0\geq 0\;\;\textrm{a.e.}\;\;\textrm{and}\;\;\int_{\TT^d}\Psi_\Phi(\rho_0(x))\dx<\infty\right\}.\]
\end{definition}
The following definition of a renormalized kinetic solution is based on the kinetic equation \eqref{skel_kin}, Definition~\ref{def_initial}, and the formal estimates obtained in Section~\ref{sec_a_priori}.

\begin{definition}\label{skel_sol_def}  Let $T\in(0,\infty)$, let $\Phi\in \C[0,\infty)\cap \C^1_{\textrm{loc}}((0,\infty))$, and let $\rho_0\in \Ent_\Phi(\TT^d)$.  A nonnegative function $\rho\in L^\infty([0,T],L^1(\mathbb{T}^d))$ is a renormalized kinetic solution of \eqref{skel_kin} with initial data $\rho_0$ if $\rho$ satisfies the following two properties.

\begin{enumerate}[(a)]

\item We have that
\begin{equation}\label{optimal_regularity_sol}\Phi^\frac{1}{2}(\rho)\in L^2([0,T];H^1(\TT^d)).\end{equation}
\item There exists a subset $\mathcal{N}\subseteq(0,T]$ of Lebesgue measure zero such that for every $t\in[0,T]\setminus \mathcal{N}$, for every $\psi\in \C^\infty_c(\TT^d\times(0,\infty))\cap C(\TT^d\times[0,\infty))$ with $\psi(x,0)=0$, the kinetic function $\chi$ of $\rho$ satisfies that
\begin{equation}\label{skel_sol_def_1}\begin{aligned}
\int_{\TT^d}\int_\mathbb{R} \chi(x,\xi,t)\psi(x,\xi)\dx\dxi = & \int_0^t\int_\R\int_{\TT^d} \Phi'(\xi)\chi \Delta_x\psi \dx\dxi\dr - \int_0^t\int_\R\int_{\TT^d}p\partial_\xi \psi \dx\dxi\dr
\\ & + 2\int_0^t\int_{\TT^d} \frac{\Phi(\rho)}{\Phi'(\rho)}\nabla_x\Phi^\frac{1}{2}(\rho)\cdot g(x,t)(\partial_\xi\psi)(x,\rho)\dx\dr
\\ & + \int_0^t\int_{\TT^d} \Phi^\frac{1}{2}(\rho)g(x,t)\cdot (\nabla_x\psi)(x,\rho) \dx\dr
\\ & + \int_\R\int_{\TT^d}\overline{\chi}(\rho_0(x),\xi)\psi(x,\xi)\dx\dxi.
\end{aligned}\end{equation}
\end{enumerate}
\end{definition}

\begin{remark}  We observe that the equality in equation \eqref{skel_sol_def_1} is satisfied due to the optimal regularity \eqref{optimal_regularity_sol}, which requires the nonnegativity of the initial data (see Remark~\ref{rem_nonnegative}).  In general, we would only expect to obtain an inequality due to the presence of a nonnegative entropy defect measure (see \cite{ChenPerthame}).  \end{remark}

\begin{remark}  Definition~\ref{skel_sol_def} does not require any continuity from the solution in time, which is convenient when constructing and establishing the stability of solutions.  However, in Proposition~\ref{L1-continuity} below, we prove that every renormalized kinetic solution of \eqref{skel_kin} has a representative in $L^\infty([0,T];L^1(\TT^d))$ that is strongly $L^1(\TT^d)$-continuous in time.  \end{remark}

The following lemma proves that pathwise kinetic solutions in the sense of Definition~\ref{skel_sol_def} satisfy an integration by parts formula on the level of their kinetic functions.

\begin{lem}\label{ibp}  Let $\Phi\in\C([0,\infty))\cap\C^1_{\textrm{loc}}((0,\infty))$ be nondecreasing, let $\rho\colon\TT^d\rightarrow \R$ be a measurable function, and let $\chi=\overline{\chi}(\rho)\colon\TT^d\times\R\rightarrow\R$.  Assume that $\Phi^\frac{1}{2}(\rho)\in H^1(\mathbb{T}^d)$.  Then, for every $\psi\in\C^\infty_c(\mathbb{T}^d\times\mathbb{R})$,
\[\frac{1}{2}\int_\R\int_{\TT^d}\frac{\Phi'(\xi)}{\Phi^\frac{1}{2}(\xi)}\nabla_x\psi(x,\xi)\chi(x,\xi,r)\dx\dxi  =-\int_{\TT^d}\psi(x,\rho(x))\nabla_x\Phi^\frac{1}{2}(\rho(x))\dx.\]
\end{lem}

\begin{proof} (see \cite[Lemma~2.6]{FehGes19} for full details) The proof is a small modification of \cite[Lemma~3.6]{FG17} and follows from change of variables formula and the fact that $\Phi$ is nondecreasing.\end{proof}

\section{Uniqueness of renormalized kinetic solutions}\label{sec_unique}

In this section, we will prove the uniqueness of renormalized kinetic solutions in the sense of Definition~\ref{skel_sol_def} for nonlinearities $\Phi$ that satisfy Assumption~\ref{as_unique} below.  The proof of uniqueness is significantly complicated by the fact that the parabolic defect measure
\[p=\delta_0(\xi-\rho)\frac{4\Phi(\xi)}{\Phi'(\xi)}\abs{\nabla\Phi^\frac{1}{2}(\rho)}^2,\]
is neither globally integrable nor decaying at infinity with respect to the velocity variable $\xi\in\R$.   It is for this reason that we introduce a cutoff in velocity.  Lemma~\ref{lem_partition} below is used to control the error terms that arise when removing this cutoff function.  We prove the uniqueness of renormalized kinetic in Theorem~\ref{thm_unique} below.

\begin{assumption}\label{as_unique}  Let $\Phi\in\C([0,\infty))\cap\C^1_{\textrm{loc}}((0,\infty))$ satisfy $\Phi(0)=0$ and $\Phi'(\xi)>0$ for every $\xi\in(0,\infty)$.  Assume that
\[\Phi'\;\;\textrm{is locally $\nicefrac{1}{2}$-H\"older continuous on $(0,\infty)$,}\]
and that there exists $c\in(0,\infty)$ such that
\begin{equation}\label{skel_continuity_3}\sup_{0\leq \xi\leq M}\abs{\frac{\Phi(\xi)}{\Phi'(\xi)}}\leq cM.\end{equation}
\end{assumption}

\begin{lem}\label{lem_partition}  Let $(X,\mathcal{S},\mu)$ be a measure space, let $K\in\mathbb{N}$, let $\{f_k\colon X\rightarrow\R\}_{k\in\{1,2,\ldots,K\}}\subseteq L^1(X)$, and for every $k\in\{1,2,\ldots,K\}$ let $\{B_{n,k}\subseteq X\}_{n\in\mathbb{N}}\subseteq \mathcal{S}$ be disjoint subsets.  Then,
\[\liminf_{n\rightarrow\infty}\left(n\sum_{k=1}^K\int_{B_{n,k}}\abs{f_k}\dd\mu\right)=0.\]
\end{lem}

\begin{proof}   Proceeding by contradiction, suppose that there exists $\ve\in(0,1)$ such that 
\[\liminf_{n\rightarrow\infty}\left(n\sum_{k=1}^K\int_{B_{n,k}}\abs{f_k}\dd\mu\right)\geq\ve.\]
Then, there exists $N\in\mathbb{N}$ such that, for every $n\geq N$,
\[n\sum_{k=1}^K\int_{B_{n,k}}\abs{f_k}\dd\mu\geq \frac{\ve}{2}.\]
For every $k\in\{1,2,\ldots,K\}$ let $\mathcal{I}_{N,k}\subseteq[N,N+1,\ldots)$ be defined by
\[\mathcal{I}_{N,k}=\left\{n\in[N,N+1,\ldots)\colon \int_{B_{n,k}}\abs{f_k}\dd\mu\geq \frac{\ve}{2Kn}\right\}.\]
Since by definition $[N,N+1,\ldots)=\cup_{k=1}^K\mathcal{I}_{N,k}$, and since $\sum_{n=N}^\infty\frac{1}{n}=\infty$,
there exists $k_0\in\{1,2,\ldots,K\}$ such that $\sum_{n\in \mathcal{I}_{N,k_0}}\frac{1}{n}=\infty$.  This contradicts the assumption that $f_{k_0}\in L^1(X)$, since the assumption that the $\{B_{n,k_0}\}_{n\in\mathbb{N}}$ are disjoint and the definition of $\mathcal{I}_{n,k_0}$ imply that
\[\infty=\sum_{n\in\mathcal{I}_{N,k_0}}\frac{1}{n}\leq \frac{2K}{\ve}\sum_{n\in\mathcal{I}_{N,k_0}}\int_{B_{n,k_0}}\abs{f_{k_0}}\dd\mu\leq \frac{2K}{\ve}\int_X \abs{f_{k_0}}<\infty,\]
which completes the proof.  \end{proof}

\begin{thm}\label{thm_unique} Let $T\in(0,\infty)$, let $\Phi\in\C([0,\infty))\cap\C^1_{\textrm{loc}}((0,\infty))$ satisfy Assumption~\ref{as_unique}, let $g\in L^2(\TT^d\times [0,T];\R^d)$, and let $\rho_0^1,\rho_0^2\in \Ent_\Phi(\mathbb{T}^d)$.  Let $\rho^1,\rho^2\in L^\infty([0,T];L^1(\TT^d))$ be renormalized kinetic solutions (see Definition~\ref{skel_sol_def}) of equation \eqref{skel_eq} with control $g$ and with initial data $\rho_0^1,\rho_0^2$.  Then,
\[\norm{\rho^1-\rho^2}_{L^\infty([0,T];L^1(\mathbb{T}^d))}\leq \norm{\rho^1_0-\rho^2_0}_{L^1(\mathbb{T}^d)}.\]
\end{thm}

\begin{proof}  We will write $\chi^1,\chi^2\in L^\infty([0,T];L^1(\TT^d\times\mathbb{R}))$ for the corresponding kinetic functions, and we will write $p^1,p^2$ for the corresponding parabolic defect measures.  Finally, we will write $\mathcal{N}^1,\mathcal{N}^2\subseteq (0,T)$ for the zero sets appearing in Definition~\ref{skel_sol_def}, and we define $\mathcal{N}=\mathcal{N}^1\cup\mathcal{N}^2$.  To simplify the notation we define, for every $(x,\xi,t)\in\TT^d\times\mathbb{R}\times[0,T]$ and $i\in\{1,2\}$,
\[\chi^i_t(x,\xi)=\chi(x,\xi,t),\]
and we will make similar conventions for $p^i_t$ and all other time-dependent functions or measures appearing in the proof.

Let $\kappa^s\colon\TT^d\rightarrow[0,\infty)$ be a standard convolution kernel satisfying that,  as distributions $\TT^d$,
\[\int_{\TT^d}\kappa^s \dx=1\;\;\textrm{and}\;\;\lim_{\varepsilon\rightarrow 0}\varepsilon^{-d}\kappa^s(\nicefrac{\cdot}{\varepsilon})=\delta_0.\]
Let $\kappa^v\colon\mathbb{R}\rightarrow[0,\infty)$ be a standard convolution kernel satisfying that, as distributions on $\mathbb{R}$,
\[\int_\mathbb{R}\kappa^v\dx=1\;\;\textrm{and}\;\;\lim_{\varepsilon\rightarrow 0}\varepsilon^{-1}\kappa^v(\nicefrac{\cdot}{\varepsilon})\rightarrow\delta_0.\]
For every $\varepsilon,\delta\in(0,1)$, let $\kappa^{\ve,\d}\colon(\TT^d)^2\times\mathbb{R}^2\rightarrow[0,\infty)$ be defined by
\[\kappa^{\ve,\d}(x,y,\xi,\eta)=\left(\varepsilon^{-d}\kappa^s(\nicefrac{x-y}{\varepsilon})\right)\left(\delta^{-1}\kappa^v(\nicefrac{\xi-\eta}{\delta})\right).\]
Finally, for every $M\in(0,\infty)$, let $\zeta^M:\mathbb{R}\rightarrow[0,1]$ be continuous piecewise linear function satisfying
\[\zeta^M(\xi)=\left\{\begin{aligned}
& 0 && \textrm{if}\;\;\abs{\xi}\leq\frac{1}{M}\;\;\textrm{or}\;\;\abs{\xi}\geq M+1, && 1 && \textrm{if}\;\;\frac{2}{M}\leq\abs{\xi}\leq M,
\\ & M\left(\abs{\xi}-\frac{1}{M}\right) && \textrm{if}\;\;\frac{1}{M}\leq\abs{\xi}\leq \frac{2}{M}, && M+1-\abs{\xi} && \textrm{if}\;\;M\leq \abs{\xi}\leq M+1.\end{aligned}\right.\]
The convolution kernel $\kappa^{\ve,\d}$ will play the role of the test function in Definition~\ref{skel_sol_def}.  The cutoff $\zeta^M$ is necessary owing to the fact that the parabolic defect measure is not globally integrable in the velocity variable $\xi\in\R$.

For every $i\in\{1,2\}$ and $\varepsilon,\delta\in(0,1)$, let $\chi^{i,\varepsilon,\delta}\colon\TT^d\times\mathbb{R}\times[0,T]\rightarrow\mathbb{R}$ be defined by
\[\chi^{i,\varepsilon,\delta}_t(y,\eta)=\int_\mathbb{R}\int_{\TT^d}\chi^i_t(x,\xi)\kappa^{\ve,\d}(x,y,\xi,\eta)\dx\dxi.\]
For every $\delta\in(0,1)$, let $\sgn^\delta\colon\mathbb{R}\rightarrow\mathbb{R}$ be defined for every $\eta\in\mathbb{R}$, $y\in \TT^d$, and $\ve\in(0,1)$ by
\[\sgn^\delta(\eta)=\int_\mathbb{R}\sgn(\xi)\left(\delta^{-1}\kappa^v\left(\frac{\xi-\eta}{\delta}\right)\right)\dxi=\int_\mathbb{R}\int_{\TT^d}\sgn(\xi)\kappa^{\ve,\d}(x,y,\xi,\eta)\dx\dxi.\]
Let $t\in(0,T]\setminus \mathcal{N}$ and $M\in(0,\infty)$.  Properties of the kinetic function prove that
\begin{equation}\label{su_0}\begin{aligned}
& \int_\mathbb{R}\int_{\TT^d}\abs{\chi^1_t-\chi^2_t}^2\zeta^M(\eta)\dy\deta \\ & =\lim_{\varepsilon,\delta\rightarrow 0}\int_\mathbb{R}\int_{\TT^d}\abs{\chi^{1,\varepsilon,\delta}_t-\chi^{2,\varepsilon,\delta}_t}^2\zeta^M(\eta)\dy\deta
\\ & = \lim_{\varepsilon,\delta\rightarrow 0}\int_\mathbb{R}\int_{\TT^d}\left(\chi^{1,\ve,\delta}_t\sgn^\delta(\eta)+\chi^{2,\varepsilon,\delta}_t\sgn^\delta(\eta)-2\chi^{1,\ve,\delta}_t\chi^{2,\ve,\delta}_t\right)\zeta^M(\eta)\dy\deta.
\end{aligned}\end{equation}
Let $\varepsilon,\delta\in(0,1)$.  Definition~\ref{skel_sol_def} implies for every $i\in\{1,2\}$ that, as distributions on $\TT^d\times\R\times(0,T)$,
\[\begin{aligned}
\partial_t\chi^{i,\ve,\delta}_t(y,\eta) = & \int_\R\int_{\TT^d}\Phi'(\xi)\chi^i_t\Delta_x\kappa^{\ve,\d}(x,y,\xi,\eta)\dx\dxi \\ & -\int_\R\int_{\TT^d}p^i_t\partial_\xi\kappa^{\ve,\d}(x,y,\xi,\eta)\dx\dxi
\\ & +2 \int_{\TT^d}\frac{\Phi(\rho^i)}{\Phi'(\rho^i)}g(x,t)\cdot \nabla_x\Phi^\frac{1}{2}(\rho^i)(\partial_\xi \kappa^{\ve,\d})(x,y,\rho^i,\eta)\dx
\\ & + \int_{\TT^d}\Phi^\frac{1}{2}(\rho^i)g(x,t)\cdot(\nabla_x\kappa^{\ve,\d})(x,y,\rho^i,\eta)\dx.
\end{aligned}\]
For every $i\in\{1,2\}$, let $\overline{\kappa}^{\ve,\d}_i:(\TT^d)^2\times\mathbb{R}\times[0,T]\rightarrow[0,\infty)$ be defined by
\[\overline{\kappa}^{\ve,\d}_{i,t}(x,y,\eta)=\kappa^{\ve,\d}(x,y,\rho^i(x,t),\eta).\]
It then follow from Lemma~\ref{ibp} and the definition of the convolution kernel that, as distributions on $\TT^d\times\R\times(0,T)$,
\begin{equation}\label{su_7} \begin{aligned}
\partial_t\chi_t^{i,\ve,\delta}(y,\eta)= & \nabla_y\cdot \left(2\int_{\TT^d}\Phi^\frac{1}{2}(\rho^i)\nabla_x \Phi^\frac{1}{2}(u_i)\overline{\kappa}^{\ve,\d}_{i,t}\dx\right) + \partial_\eta\int_\R\int_{\TT^d}p^i_t\kappa^{\ve,\d}\dx\dxi
\\ & -\partial_\eta\left(2\int_\R\int_{\TT^d}g(x,t)\frac{\Phi(\rho^i)}{\Phi'(\rho^i)}\cdot \nabla_x\Phi^\frac{1}{2}(\rho^i)\overline{\kappa}^{\ve,\d}_{i,t}\dx\right)- \nabla_y\cdot  \int_{\TT^d}g(x,t)\Phi^\frac{1}{2}(\rho^i)\overline{\kappa}^{\ve,\d}_{i,t}\dx.
\end{aligned}\end{equation}
We define
\begin{equation}\label{su_0007}\mathbb{I}^{\ve,\d,M}_t=\int_\mathbb{R}\int_{\TT^d}\left(\chi^{1,\ve,\delta}_t\sgn^\delta(\eta)+\chi^{2,\varepsilon,\delta}_t\sgn^\delta(\eta)-2\chi^{1,\ve,\delta}_t\chi^{2,\ve,\delta}_t\right)\zeta^M(\eta)\dy\deta.\end{equation}
We will analyze the terms involving the $\sgn$ function and the mixed term separately.  For every $i\in\{1,2\}$, let
\begin{equation}\label{su_07}\mathbb{I}^{\ve,\d,M}_{t,i,\sgn}=\int_\mathbb{R}\int_{\TT^d}\chi^{i,\ve,\delta}_t\sgn^\delta(\eta)\zeta^M(\eta)\dy\deta\;\;\textrm{and}\;\;\mathbb{I}^{\ve,\d,M}_{t,\textrm{mix}}=\int_\R\int_{\TT^d}\chi^{1,\ve,\delta}_t\chi^{2,\ve,\delta}_t\zeta^M(\eta)\dy\deta.\end{equation}

\textbf{The sign terms.}  We will first analyze the $\sgn $ terms \eqref{su_07}, and we will first consider the case $i=1$.  We will write $(x,y,\xi,\eta)$ for the variables defining the convolution of $\chi^{1,\ve,\d}_t$, and we will write $(x',y,\xi',\eta)$ for the variables defining the convolution of $\sgn$.  Let $\kappa^{\ve,\d}_1:(\TT^d)^2\times\mathbb{R}^2\rightarrow[0,\8)$ be defined for every $(x,y,\xi,\eta)\in (\TT^d)^2\times\mathbb{R}^2$ by
\[\kappa^{\ve,\d}_1(x,y,\xi,\eta)=\kappa^{\ve,\d}(x,y,\xi,\eta)\]
and let $\kappa^{\ve,\d}_2:(\TT^d)^2\times\mathbb{R}^2\rightarrow[0,\infty)$ be defined analogously for every $(x',y,\xi',\eta)\in (\TT^d)^2\times\R^2$.
It follows from \eqref{su_7} that, as distributions on $(0,T)$, for every $t\in(0,T)$,
\[\begin{aligned}
\partial_t \mathbb{I}^{\ve,\d,M}_{t,1,\sgn}=& -\int_{\R^2}\int_{(\TT^d)^2}p^1_t\kappa^{\ve,\d}_1\partial_\eta\left(\sgn^\d(\eta)\zeta^{M}(\eta)\right)\dx\dxi\dy\deta
\\ &  + 2\int_{\R}\int_{(\TT^d)^2}g(x,t)\frac{\Phi(\rho^1)}{\Phi'(\rho^1)}\cdot \nabla_x\Phi^\frac{1}{2}(\rho^1)\overline{\kappa}^{\ve,\d}_{1,t}\partial_\eta\left(\sgn^\d(\eta)\zeta^M(\eta)\right)\dx\dy\deta.
\end{aligned}\]
Properties of the convolution kernel and the distributional equality $\partial_\xi\sgn=2\delta_0$ prove that
\begin{equation}\label{su_11}\begin{aligned}
\partial_t \mathbb{I}^{\ve,\d,M}_{t,1,\sgn}=& -2\int_{\R^2}\int_{(\TT^d)^3}p^1_t\kappa^{\ve,\d}_1\kappa^{\ve,\d}(x',y,0,\eta)\zeta^M - \int_{\R^2}\int_{(\TT^d)^2}p^1_t\kappa^{\ve,\d}_1\sgn^\delta(\eta)\partial_\eta\zeta^M
\\ & +4\int_\R\int_{(\TT^d)^3} g(x,t)\frac{\Phi(\rho^1)}{\Phi'(\rho^1)}\cdot \nabla_x\Phi^\frac{1}{2}(\rho^1)\overline{\kappa}^{\ve,\d}_{1,t}\kappa^{\ve,\d}(x',y,0,\eta)\zeta^M
\\ & + 2\int_\R\int_{(\TT^d)^2}g(x,t)\frac{\Phi(\rho^1)}{\Phi'(\rho^1)}\cdot \nabla_x\Phi^\frac{1}{2}(\rho^1)\overline{\kappa}^{\ve,\d}_{1,t}\sgn^\d(\eta)\partial_\eta\zeta^M.
\end{aligned}\end{equation}
The case $i=2$ is identical, but to clarify some important cancellations below we will use the variables $(x',y,\xi',\eta)\in(\TT^d)^2\times\R^2$ for the convolution defining $\chi^{2,\ve,\d}$ and the variables $(x,y,\xi,\eta)\in(\TT^d)^2\times\R^2$ for the variables defining $\sgn^\d$.  This completes the initial analysis of the $\sgn$ terms.

\textbf{The mixed term.}  We will now analyze the mixed term \eqref{su_07}.  We write $(x,y,\xi,\eta)$ for the variables defining $\chi^{1,\ve,\d}_t$, and we write $(x',y,\xi',\eta)$ for the variables defining $\chi^{2,\ve,\d}_t$.  Let $\kappa^{\ve,\d}_1,\kappa^{\ve,\d}_{2}:(\TT^d)^2\times\mathbb{R}^2\rightarrow[0,\8)$ be defined as above.
As distributions on $(0,T)$,
\begin{equation}\label{su_011}\begin{aligned}
\partial_t\mathbb{I}^{\ve,\d,M}_{t,\textrm{mix}} = & \int_\R\int_{\TT^d}\partial_t\chi^{1,\ve,\d}_t\chi^{2,\ve,\d}_t\zeta^M(\eta)\dy\deta+\int_\R\int_{\TT^d}\chi^{1,\ve,\d}_t\partial_t\chi^{2,\ve,\d}_t\zeta^M(\eta)\dy\deta
\\ = & \partial_t\mathbb{I}^{\ve,\d,M}_{t,1,\textrm{mix}}+\partial_t\mathbb{I}^{\ve,\d,M}_{t,2,\textrm{mix}}.
\end{aligned}\end{equation}
It follows from \eqref{su_7} that
\[\begin{aligned}
\partial_t\mathbb{I}^{\ve,\d,M}_{t,1,\textrm{mix}} = &  -2\int_\R\int_{(\TT^d)^2}\Phi^{\frac{1}{2}}(\rho^1)\nabla_x\Phi^\frac{1}{2}(\rho^1)\overline{\kappa}^{\ve,\d}_{1,t}\cdot \nabla_y\chi^{2,\ve,\d}_t\zeta^M\dx\dy\deta
\\ & -\int_{\R^2}\int_{(\TT^d)^2}p^1_t\kappa^{\ve,\d}_1\partial_\eta(\chi^{2,\ve,\d}_t\zeta^M)\dx\dxi\dy\deta
\\ & +2\int_\R\int_{(\TT^d)^2}g(x,t)\frac{\Phi(\rho^1)}{\Phi'(\rho^1)}\cdot \nabla_x\Phi^\frac{1}{2}(\rho^1)\overline{\kappa}^{\ve,\d}_{1,t}\partial_\eta(\chi^{2,\ve,\d}_t\zeta^M)\dx\dy\deta
\\ & +\int_\R\int_{(\TT^d)^2}g(x,t)\Phi^\frac{1}{2}(\rho^1)\overline{\kappa}^{\ve,\d}_{1,t}\cdot \nabla_y\chi^{2,\ve,\d}_t\zeta^M\dx \dy\deta.
\end{aligned}\]
The definition of the convolution kernel implies that, omitting the integration variables,
\begin{equation}\label{su_12}\begin{aligned}
\partial_t\mathbb{I}^{\ve,\d,M}_{t,1,\textrm{mix}} = &  2\int_{\R^2}\int_{(\TT^d)^3}\Phi^{\frac{1}{2}}(\rho^1)\nabla_x\Phi^\frac{1}{2}(\rho^1)\overline{\kappa}^{\ve,\d}_{1,t}\chi^2_t\cdot \big(\nabla_{x'}\kappa^{\ve,\d}_2\big)\zeta^M
\\ & +\int_{\R^3}\int_{(\TT^d)^3}p^1_t\kappa^{\ve,\d}_1\chi^2\big(\partial_{\xi'}\kappa^{\ve,\d}_2\big)\zeta^M-\int_{\R^2}\int_{(\TT^d)^2}p^1_t\kappa^{\ve,\d}_1\chi^{2,\ve,\d}_t\partial_\eta\zeta^M
\\ & -2\int_{\R^2}\int_{(\TT^d)^3}g(x,t)\frac{\Phi(\rho^1)}{\Phi'(\rho^1)}\cdot \nabla_x\Phi^\frac{1}{2}(\rho^1)\overline{\kappa}^{\ve,\d}_{1,t}\chi^2_t\big(\partial_{\xi'}\kappa^{\ve,\d}_2\big)\zeta^M
\\ & +2\int_{\R}\int_{(\TT^d)^2}g(x,t)\frac{\Phi(\rho^1)}{\Phi'(\rho^1)}\cdot \nabla_x\Phi^\frac{1}{2}(\rho^1)\overline{\kappa}^{\ve,\d}_{1,t}\chi^{2,\ve,\d}_t\partial_\eta\zeta^M
\\ & - \int_{\R^2}\int_{(\TT^d)^3}g(x,t)\Phi^\frac{1}{2}(\rho^1)\overline{\kappa}^{\ve,\d}_{1,t}\chi^2_t\cdot \big(\nabla_{x'}\kappa^{\ve,\d}_2\big)\zeta^M.
\end{aligned}\end{equation}
The distributional equality $\partial_{\xi'}\chi^2=\delta_0(\xi')-\delta_0(\xi'-\rho^2)$ and Lemma~\ref{ibp} prove that, after integrating by parts in \eqref{su_12},
\begin{equation}\label{su_18}\begin{aligned}
\partial_t\mathbb{I}^{\ve,\d,M}_{t,1,\textrm{mix}} = &  -4\int_{\R}\int_{(\TT^d)^3}\frac{\Phi^{\frac{1}{2}}(\rho^1)\Phi^\frac{1}{2}(\rho^2)}{\Phi'(\rho^2)}\nabla_x\Phi^\frac{1}{2}(\rho^1)\cdot \nabla_{x'}\Phi^\frac{1}{2}(\rho^2)\overline{\kappa}^{\ve,\d}_{1,t}\overline{\kappa}^{\ve,\d}_{2,t}\zeta^M
\\ &  +\int_{\R^2}\int_{(\TT^d)^3}p^1_t\kappa^{\ve,\d}_1\left(\overline{\kappa}^{\ve,\d}_{2,t}-\kappa^{\ve,\d}(x',y,0,\eta)\right)\zeta^M
\\ & -2\int_{\R}\int_{(\TT^d)^3}g(x,t)\frac{\Phi(\rho^1)}{\Phi'(\rho^1)}\cdot \nabla_x\Phi^\frac{1}{2}(\rho^1)\overline{\kappa}^{\ve,\d}_{1,t}\left(\overline{\kappa}^{\ve,\d}_{2,t}-\kappa^{\ve,\d}(x',y,0,\eta)\right)\zeta^M
\\ & +2\int_{\R}\int_{(\TT^d)^3}g(x,t)\frac{\Phi^\frac{1}{2}(\rho^1)\Phi^\frac{1}{2}(\rho^2)}{\Phi'(\rho^2)}\cdot \nabla_{x'}\Phi^\frac{1}{2}(\rho^2)\overline{\kappa}^{\ve,\d}_{1,t}\overline{\kappa}^{\ve,\d}_{2,t}\zeta^M
\\ & +2\int_{\R}\int_{(\TT^d)^2}g(x,t)\frac{\Phi(\rho^1)}{\Phi'(\rho^1)}\cdot \nabla_x\Phi^\frac{1}{2}(\rho^1)\overline{\kappa}^{\ve,\d}_{1,t}\chi^{2,\ve,\d}_t\partial_\eta\zeta^M-\int_{\R^2}\int_{(\TT^d)^2}p^1_t\kappa^{\ve,\d}_1\chi^{2,\ve,\d}_t\partial_\eta\zeta^M.
\end{aligned}\end{equation}
We obtain an identical formula for $\mathbb{I}^{\ve,\d,M}_{t,2,\textrm{mix}}$, after swapping the roles of $i\in\{1,2\}$, $x,x'\in\TT^d$, and $\xi,\xi'\in\R$.  Therefore, after returning to \eqref{su_011},  this completes the initial analysis of the mixed term.

\textbf{The full derivative.}  We will decompose the full derivative of \eqref{su_0007} defined by \eqref{su_07},  \eqref{su_11}, \eqref{su_011}, and \eqref{su_18} into the four terms
\begin{equation}\label{su_24}\partial_t\mathbb{I}^{\ve,\d,M}_t = \partial_t\mathbb{I}^{\ve,\d,M}_{t,\textrm{par}}+\partial_t\mathbb{I}^{\ve,\d,M}_{t,\textrm{hyp}}+ \partial_t\mathbb{I}^{\ve,\d,M}_{t,\textrm{con}}+\partial_t\mathbb{I}^{\ve,\d,M}_{t,\textrm{vel}},\end{equation}
defined by the parabolic term
\begin{equation}\label{su_21}\begin{aligned}
\partial_t\mathbb{I}^{\ve,\d,M}_{t,\textrm{par}} & =  8\int_{\R}\int_{(\TT^d)^3}\Phi^{\frac{1}{2}}(\rho^1)\Phi^\frac{1}{2}(\rho^2)\left(\frac{1}{\Phi'(\rho^1)}+\frac{1}{\Phi'(\rho^2)}\right)\nabla_x\Phi^\frac{1}{2}(\rho^1)\cdot \nabla_{x'}\Phi^\frac{1}{2}(\rho^2)\overline{\kappa}^{\ve,\d}_{1,t}\overline{\kappa}^{\ve,\d}_{2,t}\zeta^M
\\ &\quad  -2\int_{\R^2}\int_{(\TT^d)^3}\left(p^1_t\kappa^{\ve,\d}_1\overline{\kappa}^{\ve,\d}_{2,t}+p^2_t\kappa^{\ve,\d}_2\overline{\kappa}^{\ve,\d}_{1,t}\right)\zeta^M,
\end{aligned}\end{equation}
the hyperbolic term
\begin{equation}\label{su_22}\begin{aligned}
\partial_t\mathbb{I}^{\ve,\d,M}_{t,\textrm{hyp}} & =  4\int_\R\int_{(\TT^d)^3}g(x,t)\frac{\Phi^\frac{1}{2}(\rho^1)}{\Phi'(\rho^1)}\left(\Phi^\frac{1}{2}(\rho^1)-\Phi^\frac{1}{2}(\rho^2)\right)\cdot \nabla_x\Phi^\frac{1}{2}(\rho^1)\overline{\kappa}^{\ve,\d}_{1,t}\overline{\kappa}^{\ve,\d}_{2,t}\zeta^M
\\ & \quad +4\int_\R\int_{(\TT^d)^3}g(x',t)\frac{\Phi^\frac{1}{2}(\rho^2)}{\Phi'(\rho^2)}\left(\Phi^\frac{1}{2}(\rho^2)-\Phi^\frac{1}{2}(\rho^1)\right)\cdot \nabla_{x'}\Phi^\frac{1}{2}(\rho^2)\overline{\kappa}^{\ve,\d}_{1,t}\overline{\kappa}^{\ve,\d}_{2,t}\zeta^M,
\end{aligned}\end{equation}
the term involving the control
\begin{equation}\label{su_022}\begin{aligned}
\partial_t\mathbb{I}^{\ve,\d,M}_{t,\textrm{con}} & =  4\int_{\R}\int_{(\TT^d)^3}\left(g(x,t)-g(x',t)\right)\frac{\Phi^\frac{1}{2}(\rho^1)\Phi^\frac{1}{2}(\rho^2)}{\Phi'(\rho^1)}\cdot \nabla_x\Phi^\frac{1}{2}(\rho^1)\overline{\kappa}^{\ve,\d}_{1,t}\overline{\kappa}^{\ve,\d}_{2,t}\zeta^M
\\ & \quad +4\int_{\R}\int_{(\TT^d)^3}\left(g(x',t)-g(x,t)\right)\frac{\Phi^\frac{1}{2}(\rho^1)\Phi^\frac{1}{2}(\rho^2)}{\Phi'(\rho^2)}\cdot \nabla_{x'}\Phi^\frac{1}{2}(\rho^2)\overline{\kappa}^{\ve,\d}_{1,t}\overline{\kappa}^{\ve,\d}_{2,t}\zeta^M,
\end{aligned}\end{equation}
and the term defined by the cutoff
\begin{equation}\label{su_23} \begin{aligned}
\partial_t\mathbb{I}^{\ve,\d,M}_{t,\textrm{vel}} & =  2\int_{\R}\int_{(\TT^d)^2}g(x,t)\frac{\Phi(\rho^1)}{\Phi'(\rho^1)}\cdot \nabla_x\Phi^\frac{1}{2}(\rho^1)\overline{\kappa}^{\ve,\d}_{1,t}\left(\sgn^\delta-2\chi^{2,\ve,\d}_t\right)\partial_\eta\zeta^M
\\ & \quad + 2\int_{\R}\int_{(\TT^d)^2}g(x',t)\frac{\Phi(\rho^2)}{\Phi'(\rho^2)}\cdot \nabla_{x'}\Phi^\frac{1}{2}(\rho^2)\overline{\kappa}^{\ve,\d}_{2,t}\left(\sgn^\delta-2\chi^{1,\ve,\d}_t\right)\partial_\eta\zeta^M
\\ & \quad + \int_{\R^2}\int_{(\TT^d)^2}p^1_t\kappa^{\ve,\d}_1\left(2\chi^{2,\ve,\d}_t-\sgn^\delta\right)\partial_\eta\zeta^M
\\ & \quad + \int_{\R^2}\int_{(\TT^d)^2}p^2_t\kappa^{\ve,\d}_2\left(2\chi^{1,\ve,\d}_t-\sgn^\delta\right)\partial_\eta\zeta^M.
\end{aligned}\end{equation}
The four terms on the righthand side of \eqref{su_24} will be handled separately.

\textbf{The parabolic terms.}  After adding and subtracting $2(\Phi'(\rho^1)\Phi'(\rho^2))^{-\nicefrac{1}{2}}$ and using the identity
\[\frac{1}{\Phi'(\rho^1)}+\frac{1}{\Phi'(\rho^2)}-\frac{2}{\Phi'(\rho^1)^\frac{1}{2}\Phi'(\rho^2)^\frac{1}{2}}=\frac{\left(\Phi'(\rho^1)^\frac{1}{2}-\Phi'(\rho^2)^\frac{1}{2}\right)^2}{\Phi'(\rho^1)\Phi'(\rho^2)},\]
the parabolic term defined in \eqref{su_21} satisfies
\begin{equation}\label{su_25}\begin{aligned}
\partial_t\mathbb{I}^{\ve,\d,M}_{t,\textrm{par}} & = 8\int_{\R}\int_{(\TT^d)^3}\frac{\Phi^{\frac{1}{2}}(\rho^1)\Phi^\frac{1}{2}(\rho^2)}{\Phi'(\rho^1)\Phi'(\rho^2)}\left(\Phi'(\rho^1)^\frac{1}{2}-\Phi'(\rho^2)^\frac{1}{2}\right)^2\nabla_x\Phi^\frac{1}{2}(\rho^1)\cdot \nabla_{x'}\Phi^\frac{1}{2}(\rho^2)\overline{\kappa}^{\ve,\d}_{1,t}\overline{\kappa}^{\ve,\d}_{2,t}\zeta^M
\\ & \quad + 16\int_\R\int_{(\TT^d)^3}\frac{\Phi^{\frac{1}{2}}(\rho^1)\Phi^\frac{1}{2}(\rho^2)}{\Phi'(\rho^1)^\frac{1}{2}\Phi'(\rho^2)^\frac{1}{2}}\nabla_x\Phi^\frac{1}{2}(\rho^1)\cdot \nabla_{x'}\Phi^\frac{1}{2}(\rho^2)\overline{\kappa}^{\ve,\d}_{1,t}\overline{\kappa}^{\ve,\d}_{2,t}\zeta^M
\\ & \quad -2\int_{\R^2}\int_{(\TT^d)^3}\left(p^1_t\kappa^{\ve,\d}_1\overline{\kappa}^{\ve,\d}_{2,t}+p^2_t\kappa^{\ve,\d}_2\overline{\kappa}^{\ve,\d}_{1,t}\right)\zeta^M.
\end{aligned}\end{equation}
The definition of the parabolic defect measures $\{p^i\}_{i\in\{1,2\}}$, H\"older's inequality, and Young's inequality prove that
\begin{equation}\label{su_26}\begin{aligned}
& 16\int_\R\int_{(\TT^d)^3}\frac{\Phi^{\frac{1}{2}}(\rho^1)\Phi^\frac{1}{2}(\rho^2)}{\Phi'(\rho^1)^\frac{1}{2}\Phi'(\rho^2)^\frac{1}{2}}\nabla_x\Phi^\frac{1}{2}(\rho^1)\cdot \nabla_{x'}\Phi^\frac{1}{2}(\rho^2)\overline{\kappa}^{\ve,\d}_{1,t}\overline{\kappa}^{\ve,\d}_{2,t}\zeta^M
\\ & \leq 2\int_{\R^2}\int_{(\TT^d)^3}\left(p^1_t\kappa^{\ve,\d}_1\overline{\kappa}^{\ve,\d}_{2,t}+p^2_t\kappa^{\ve,\d}_2\overline{\kappa}^{\ve,\d}_{1,t}\right)\zeta^M.
\end{aligned}\end{equation}
It follows from the definition of the convolution kernels, the local $\nicefrac{1}{2}$-H\"older continuity of $\Phi'$, the positivity of $\Phi'$ on compact subsets of $(0,\infty)$, \eqref{su_25}, and \eqref{su_26} that there exists $c\in(0,\infty)$ depending on $M\in(0,\infty)$ but independent of $\ve,\delta\in(0,1)$ such that
\[\partial_t\mathbb{I}^{\ve,\d,M}_{t,\textrm{par}}\leq c\delta \int_\R\int_{(\TT^d)^3}\mathbf{1}_{\{0<\abs{\rho^1(x,t)-\rho^2(x',t)}\leq c\delta\}}\abs{\nabla_x\Phi^\frac{1}{2}(\rho^1)}\abs{\nabla_x\Phi^\frac{1}{2}(\rho^2)}\overline{\kappa}^{\ve,\d}_{1,t}\overline{\kappa}^{\ve,\d}_{2,t}\zeta^M.\]
The definition of the convolution kernels, the definition of $\zeta^M$, H\"older's inequality, and Young's inequality prove that there exists $c\in(0,\infty)$ depending on $M\in(0,\infty)$ which satisfies for every $\delta\in(0,1)$ that
\[\limsup_{\ve\rightarrow 0}\left(\partial_t\mathbb{I}^{\ve,\d,M}_{t,\textrm{par}}\right)\leq c\int_{\left\{\substack{0<\abs{\rho^1-\rho^2}\leq c\delta\\ \abs{\rho_1},\abs{\rho_2}\leq M+1+\delta}\right\}}\abs{\nabla_x\Phi^\frac{1}{2}(\rho^1)}^2+\abs{\nabla_x\Phi^\frac{1}{2}(\rho^2)}^2\dx.\]
The fact that for each $i\in\{1,2\}$ we have $\Phi^\frac{1}{2}(\rho^i)\in L^2([0,T];H^1(\TT^d))$ and the dominated convergence theorem then imply for almost every $t\in(0,T]$ that
\begin{equation}\label{su_27}\limsup_{\d\rightarrow 0}\left(\limsup_{\ve\rightarrow 0}\left(\partial_t\mathbb{I}^{\ve,\d,M}_{t,\textrm{par}}\right)\right)\leq 0,\end{equation}
which completes the analysis of the parabolic terms.

\textbf{The hyperbolic terms.}  The definition of the hyperbolic terms \eqref{su_22}, the definition of the convolution kernels, the local Lipschitz continuity of $\Phi$, and the positivity of $\Phi'$ on every compact subset of $(0,\infty)$ prove that there exists $c\in(0,\infty)$ depending on $M\in(0,\infty)$ but independent of $\ve,\d\in(0,1)$ so that
\[\begin{aligned}
& \partial_t\mathbb{I}^{\ve,\d,M}_{t,\textrm{hyp}}
\\ & \leq c\delta\int_\R\int_{(\TT^d)^3}\mathbf{1}_{\{0<\abs{\rho^1(x,t)-\rho^2(x',t)}\leq c\delta\}}\left(\abs{g(x,t)}\abs{\nabla_x\Phi^\frac{1}{2}(\rho^1)}+\abs{g(x',t)}\abs{\nabla_{x'}\Phi^\frac{1}{2}(\rho^2)}\right)\overline{\kappa}^{\ve,\d}_{1,t}\overline{\kappa}^{\ve,\d}_{2,t}\zeta^M.
\end{aligned}\]
The definition of the convolution kernels, the definition of $\zeta^M$, H\"older's inequality, and Young's inequality prove that there exists $c\in(0,\infty)$ depending on $M\in(0,\infty)$ such that, for every $\delta\in(0,1)$,
\[\limsup_{\ve\rightarrow 0}\left(\partial_t\mathbb{I}^{\ve,\d,M}_{t,\textrm{hyp}}\right)\leq c\int_{\{0<\abs{\rho^1-\rho^2}\leq c\delta\}}\abs{g(x,t)}^2+\abs{\nabla_x\Phi^\frac{1}{2}(\rho^1)}^2+\abs{\nabla_x\Phi^\frac{1}{2}(\rho^2)}^2\dx.\]
Since $g\in L^2(\TT^d\times[0,T];\R^d)$ and since for each $i\in\{1,2\}$ we have $\Phi^\frac{1}{2}(\rho^i)\in L^2([0,T];H^1(\TT^d))$, the dominated convergence theorem proves that, for almost every $t\in(0,T]$,
\begin{equation}\label{su_28}\limsup_{\delta\rightarrow 0}\left(\limsup_{\ve\rightarrow 0}\left(\partial_t\mathbb{I}^{\ve,\d,M}_{t,\textrm{hyp}}\right)\right)\leq 0,\end{equation}
which completes the analysis of the hyperbolic terms.

\textbf{The control terms.}  The definition of the convolution kernel, the definition of $\zeta^M$, the continuity of $\Phi$, and the positivity of $\Phi'$ on $(0,\infty)$ prove that there exists $c\in(0,\infty)$ depending on $M\in(0,\infty)$ but independent of $\ve,\d\in(0,1)$ such that \eqref{su_022} satisfies
\[\partial_t\mathbb{I}^{\ve,\d,M}_{t,\textrm{con}} \leq c\int_{\R}\int_{(\TT^d)^3}\abs{g(x',t)-g(x,t)}\left(\abs{\nabla_x\Phi^\frac{1}{2}(\rho^1)}+\abs{\nabla_{x'}\Phi^\frac{1}{2}(\rho^2)}\right)\overline{\kappa}^{\ve,\d}_{1,t}\overline{\kappa}^{\ve,\d}_{2,t}\zeta^M. \]
Since $g\in L^2(\TT^d\times[0,T];\R^d)$ and since for each $i\in\{1,2\}$ we have that $\Phi^\frac{1}{2}(\rho^i)\in L^2([0,T];H^1(\TT^d))$, for almost every $t\in(0,T)$, for every $M\in(0,\infty)$ and $\delta\in(0,1)$,
\[\limsup_{\ve\rightarrow 0}\left( \partial_t\mathbb{I}^{\ve,\d,M}_{t,\textrm{con}} \right)\leq 0,\]
and, therefore, for almost every $t\in(0,T)$, for every $M\in(0,\infty)$ and $\delta\in(0,1)$,
\begin{equation}\label{su_29}\limsup_{\d\rightarrow 0}\left(\limsup_{\ve\rightarrow 0}\left( \partial_t\mathbb{I}^{\ve,\d,M}_{t,\textrm{con}}\right) \right)\leq 0,\end{equation}
which completes the analysis of the control terms.

\textbf{The cutoff in velocity.}  For every $\delta\in(0,1)$ let $\kappa^{v,\d}\colon\R\rightarrow\R$ be defined by
\[\kappa^{v,\d}(\xi)=\delta^{-1}\kappa^v\left(\frac{\xi}{\delta}\right).\]
The definition of the cutoff in velocity \eqref{su_23}, the definition of the convolution kernels, the inequality
\[\abs{2\chi^{i,\ve,\d}-\sgn^\delta}\leq 1,\]
and the definition of $\zeta^M$ prove that, for every $M\in(0,\infty)$,
\[\begin{aligned}
\limsup_{\ve \rightarrow 0}\left(\partial_t\mathbb{I}^{\ve,\d,M}_{t,\textrm{vel}}\right) & \leq 2\int_{\R}\int_{\TT^d}\abs{g(y,t)}\frac{\Phi(\rho^1)}{\Phi'(\rho^1)}\abs{\nabla_x\Phi^\frac{1}{2}(\rho^1)}\kappa^{v,\d}(\rho^1-\eta)\abs{\partial_\eta\zeta^M} 
\\ & \quad  + 2\int_{\R}\int_{\TT^d}\abs{g(y,t)}\frac{\Phi(\rho^2)}{\Phi'(\rho^2)}\abs{\nabla_{x'}\Phi^\frac{1}{2}(\rho^2)}\kappa^{v,\d}(\rho^2-\eta)\abs{\partial_\eta\zeta^M}
\\ & \quad + \int_{\R^2}\int_{\TT^d}p^1_t\kappa^{v,\d}(\rho^1-\eta)\abs{\partial_\eta\zeta^M} + \int_{\R^2}\int_{\TT^d}p^2_t\kappa^{v,\d}(\rho^2-\eta)\abs{\partial_\eta\zeta^M}.
\end{aligned} \]
The definition of the convolution kernel, the definition of $\zeta^M$, $\Phi(0)=0$, and \eqref{skel_continuity_3} imply that there exists $c\in(0,\infty)$ such that, for every $M\in(0,\infty)$,
\begin{equation}\label{su_30}\begin{aligned}
& \limsup_{\d\rightarrow 0}\left(\limsup_{\ve\rightarrow 0}\left(\mathbb{I}^{\ve,\d,M}_{t,\textrm{vel}}\right)\right) 
\\ & \leq c(M+1)\int_{\{M\leq \rho^1\leq M+1\}}\abs{g(y,t)}\abs{\nabla_x\Phi^\frac{1}{2}(\rho^1)}+c\int_{\{0<\rho^1\leq\frac{2}{M}\}}\abs{g(y,t)}\abs{\nabla_x\Phi^\frac{1}{2}(\rho^1)}
\\ & \quad + c(M+1)\int_{\{M\leq \rho^2\leq M+1\}}\abs{g(y,t)}\abs{\nabla_x\Phi^\frac{1}{2}(\rho^2)}+c\int_{\{0<\rho^2\leq\frac{2}{M}\}}\abs{g(y,t)}\abs{\nabla_x\Phi^\frac{1}{2}(\rho^2)}
\\ & \quad + \int_\R\int_{\{M\leq \rho^1\leq M+1\}}p^1_t\dy\deta + \int_\R\int_{\{M\leq \rho^2\leq M+1\}}p^2_t\dy\deta
\\ & \quad + M\int_\R\int_{\{0< \rho^1\leq \frac{2}{M}\}}p^1_t\dy\deta + M\int_\R\int_{\{0< \rho^2\leq \frac{2}{M}\}}p^2_t\dy\deta.
\end{aligned}\end{equation}
The definition of the parabolic defect measures, $\Phi(0)=0$, and \eqref{skel_continuity_3} prove that there exists $c\in(0,\infty)$ satisfying for every $i\in\{1,2\}$ that
\begin{equation}\label{su_030} \begin{aligned} & \int_\R\int_{\{M\leq \rho^i\leq M+1\}}p^i_t\dy\deta + M\int_\R\int_{\{0< \rho^i\leq \frac{2}{M}\}}p^i_t\dy\deta \\ &  \leq  c(M+1)\int_\R\int_{\{M\leq \rho^i\leq M+1\}}\abs{\nabla\Phi^\frac{1}{2}(\rho^i)}^2+c\int_{\{0<\rho^i<\frac{2}{M}\}}\abs{\nabla\Phi^\frac{1}{2}(\rho^i)}^2. \end{aligned}\end{equation}
Therefore, it follows from H\"older's inequality, Young's inequality, \eqref{su_30}, and \eqref{su_030} that there exists $c\in(0,\infty)$ such that, for every $M\in(0,\infty)$,
\begin{equation}\label{su_0300}\begin{aligned}
& \limsup_{\d\rightarrow 0}\left(\limsup_{\ve\rightarrow 0}\left(\mathbb{I}^{\ve,\d,M}_{t,\textrm{vel}}\right)\right) 
\\ & \leq c\sum_{i=1}^2\left((M+1)\int_{\{M\leq \rho^i\leq M+1\}}\abs{g(y,t)}^2+\int_{\{0<\rho^i\leq\frac{2}{M}\}}\abs{g(y,t)}^2\right)
\\ & \quad + c\sum_{i=1}^2\left((M+1)\int_{\{M\leq \rho^i\leq M+1\}}\abs{\nabla_x\Phi^\frac{1}{2}(\rho^i)}^2+\int_{\{0<\rho^i\leq\frac{2}{M}\}}\abs{\nabla_x\Phi^\frac{1}{2}(\rho^i)}^2\right),
\end{aligned}\end{equation}
which completes the analysis of the cutoff in velocity.

\textbf{The conclusion.}  Returning to \eqref{su_0}, it follows from \eqref{su_27}, \eqref{su_28}, \eqref{su_29}, and \eqref{su_0300} that there exists $c\in(0,\infty)$ which satisfies for every $M\in(0,\infty)$ that
\[\begin{aligned}
& \norm{\abs{\chi^1-\chi^2}^2\zeta^M}_{L^\infty([0,T];L^1(\TT^d\times\R))}  \leq \norm{\abs{\overline{\chi}(\rho^1_0)-\overline{\chi}(\rho^2_0)}\zeta^M}_{L^1(\TT^d\times\R)}
\\ & \quad +c\sum_{i=1}^2\left((M+1)\int_0^T\int_{\{M\leq \rho^i\leq M+1\}}\abs{g(y,t)}^2+\int_0^T\int_{\{0<\rho^i\leq\frac{2}{M}\}}\abs{g(y,t)}^2\right)
\\ & \quad + c\sum_{i=1}^2\left((M+1)\int_0^T\int_{\{M\leq \rho^i\leq M+1\}}\abs{\nabla_x\Phi^\frac{1}{2}(\rho^i)}^2+\int_0^T\int_{\{0<\rho^i\leq\frac{2}{M}\}}\abs{\nabla_x\Phi^\frac{1}{2}(\rho^i)}^2\right).
\end{aligned}\]
The monotone convergence theorem, the dominated convergence theorem, and properties of the kinetic function prove that
\begin{equation}\label{su_31}\begin{aligned}
& \norm{\rho^1-\rho^2}_{L^\infty([0,T];L^1(\TT^d))} \leq   \norm{\rho^1_0-\rho^2_0}_{L^1(\TT^d)}
\\ &\quad  + \liminf_{M\rightarrow\infty}\left(c\sum_{i=1}^2\left((M+1)\int_0^T\int_{\{M\leq \rho^i\leq M+1\}}\abs{g(y,t)}^2+\abs{\nabla_x\Phi^\frac{1}{2}(\rho^i)}^2\dy\dt\right)\right).
\end{aligned}\end{equation}
After defining $\{f_i\}_{i\in\{1,2\}}\subseteq L^1(\TT^d\times[0,T])$ by
\[\begin{aligned}
f_1 & =\abs{g(y,t)}^2+\abs{\nabla_x\Phi^\frac{1}{2}(\rho^1)}^2 \;\;\textrm{and}\;\; f_2 & =\abs{g(y,t)}^2+\abs{\nabla_x\Phi^\frac{1}{2}(\rho^2)}^2,
\end{aligned}\]
and after defining the measurable subsets $\{A_{n,i}\}_{i\in\{1,2,3,4\},n\in\mathbb{N}\cup\{0\}}\subseteq\mathcal{B}(\TT^d\times(0,T))$, for every $n\in\mathbb{N}$,
\[\begin{aligned}
A_{n,1} & =\{(x,t)\in\TT^d\times(0,T)\colon 2n\leq \rho^1(x,t)\leq 2n+1\},
\\ A_{n,2} & =\{(x,t)\in\TT^d\times(0,T)\colon 2n+1\leq \rho^1(x,t)\leq 2n+2\},
\\ A_{n,3} & =\{(x,t)\in\TT^d\times(0,T)\colon 2n\leq \rho^2(x,t)\leq 2n+1\},
\\ A_{n,2} & =\{(x,t)\in\TT^d\times(0,T)\colon 2n+1\leq \rho^2(x,t)\leq 2n+2\},
\end{aligned}\]
it follows from Lemma~\ref{lem_partition} that
\begin{equation}\label{su_0000031} \liminf_{M\rightarrow\infty}\left(c\sum_{i=1}^2\left((M+1)\int_0^T\int_{\{M\leq \rho^i\leq M+1\}}\abs{g(y,t)}^2+\abs{\nabla_x\Phi^\frac{1}{2}(\rho^i)}^2\dy\dt\right)\right)= 0.\end{equation}
Returning to \eqref{su_31}, we conclude from \eqref{su_0000031} that
\[\norm{\rho^1-\rho^2}_{L^\infty([0,T];L^1(\TT^d))}\leq  \norm{\rho^1_0-\rho^2_0}_{L^1(\TT^d)},\]
which completes the proof.  \end{proof}

\begin{prop}\label{L1-continuity}  Let $T\in(0,\infty)$, let $\Phi\in\C([0,\infty))\cap\C^1_{\textrm{loc}}((0,\infty))$ satisfy Assumption~\ref{as_unique}, let $g\in L^2(\TT^d\times [0,T];\R^d)$, let $\rho_0\in \Ent_\Phi(\TT^d)$, and let $\rho\in L^1([0,T];L^1(\TT^d))$ be a renormalized kinetic solution of \eqref{skel_eq} in the sense of Definition~\ref{skel_sol_def}.  Then $\rho$ has a representative in $\C([0,T];L^1(\TT^d))$.  \end{prop}

\begin{proof}  The proof is a deterministic and simplified version of the proof of the analogous fact in \cite[Theorem~5.29]{FehGes21}.  We therefore only outline the main details.  In \cite[Equation~(5.65)]{FehGes21}, we first identify weakly right- and left-continuous representatives $\chi^{\pm}$ of the kinetic function $\chi$ of $\rho$ using the equation and properties of the kinetic measure.  We then use the equation and the uniqueness proof of Theorem~\ref{thm_unique} in \cite[Equations~(5.66), (5.67), and (5.68)]{FehGes21} to prove that the $\chi^{\pm}$ are kinetic functions in the sense that $\chi^+(x,\xi,t)=\mathbf{1}_{\{0<\xi<\rho^+(x,t)\}}$ and $\chi^-(x,\xi,t)=\mathbf{1}_{\{0<\xi<\rho^-(x,t)\}}$ for representatives $\rho^{\pm}$ of $\rho$ in $L^1([0,T];L^1(\TT^d))$.  Finally, using the definition of the parabolic defect measure and the optimal regularity estimate \eqref{optimal_regularity_sol}, we prove in \cite[Equation~(5.70)]{FehGes21} that $\rho^+=\rho^-$ is strongly $L^1(\TT^d)$-continuous in time, which completes the proof.\end{proof}

\section{Equivalence of renormalized kinetic solutions and weak solutions}\label{sec_equiv}  In this section, we will prove in Theorem~\ref{equiv} that, for nonlinearities $\Phi$ satisfying the additional Assumption~\ref{as_equiv}, renormalized kinetic solutions (see Definition~\ref{skel_sol_def}) are equivalent to classical weak solutions (see Definition~\ref{classical_weak} below).  Assumption~\ref{as_equiv} is broken into two cases, and for the model example $\Phi(\xi)=\xi^m$ an explicit computation proves that case (i) applies to $m\in [1,2]$ and case (ii) applies to $m\in[2,\infty)$.  The equivalence of renormalized kinetic solutions and classical weak solutions is used in Proposition~\ref{weak_strong} below to prove that a weakly convergent sequence of controls $g_n$ induces a strongly convergence sequence of solutions $\rho_n$.  This fact is not obvious for renormalized kinetic solutions, since the second term on the righthand side of \eqref{skel_sol_def_1} will contain what is in general the product between a weakly convergent gradient and weakly convergent control.

\begin{assumption}\label{as_equiv} Let $\Phi\in\C([0,\infty))\cap\C^1_{\textrm{loc}}((0,\infty))$ satisfy Assumption~\ref{as_unique}.  Assume that $\Phi$ satisfies one of the following two conditions.
\begin{enumerate}[(i)]
\item We have that $\Phi^\frac{1}{2}\colon[0,\infty)\rightarrow[0,\infty)$ is concave, and there exists $c\in(0,\infty)$ and $p\in[2,\infty)$ such that, for every $\xi\in[0,\infty)$,
\begin{equation}\label{skel_continuity_5} \left(\frac{\Phi^\frac{1}{2}(\xi)}{\Phi'(\xi)}\right)^p\leq c\left(\xi+1\right).\end{equation}
\end{enumerate}
\begin{enumerate}[(ii)]
\item We have that $\Phi^\frac{1}{2}\colon[0,\infty)\rightarrow[0,\infty)$ is convex, there exists $c\in(0,\infty)$ such that
\begin{equation}\label{skel_continuity_6} \esssup_{\{\xi\geq 1\}}\abs{\frac{\Phi(\xi+1)}{\Phi(\xi)}}\leq c,\end{equation}
and for every $M\in(0,1)$ there exists $c\in(0,\infty)$ such that
\begin{equation}\label{skel_continuity_7} \esssup_{\{\xi\geq M\}}\abs{\frac{\Phi^\frac{1}{2}(\xi)}{\Phi'(\xi)}}\leq c.\end{equation}
\end{enumerate}
\end{assumption}

\begin{definition}\label{classical_weak}  Let $T\in(0,\infty)$, let $\Phi\in\C([0,\infty))\cap\C^1_{\textrm{loc}}((0,\infty))$, let $g\in L^2(\TT^d\times[0,T];\R^d)$, and let $\rho_0\in\Ent_\Phi(\TT^d)$.  A function $\rho\in L^\infty([0,T];L^1(\TT^d))$ is a \emph{weak solution} of \eqref{skel_eq} with initial data $\rho_0$ if $\rho$ satisfies the following two properties.
\begin{enumerate}[(i)]
\item We have $\Phi^\frac{1}{2}(\rho)\in L^2([0,T];H^1(\TT^d))$.
\item There exists a subset $\mathcal{N}\subseteq(0,T]$ of Lebesgue measure zero such that, for every $t\in[0,T]\setminus\mathcal{N}$, for every $\psi\in \C^\infty(\TT^d)$,
\begin{equation}\begin{aligned}\label{equiv_2}\int_{\TT^d}\rho(x,t)\psi(x)\dx & =-\int_0^t\int_{\TT^d}2\Phi^\frac{1}{2}(\rho)\nabla\Phi^\frac{1}{2}(\rho)\cdot\nabla \psi \dx\dt+\int_0^t\int_{\TT^d}g\Phi^\frac{1}{2}(\rho)\cdot \nabla \psi \dx\dt  \\ & \quad +\int_{\TT^d}\rho_0(x)\psi(x)\dx.\end{aligned}\end{equation}
\end{enumerate}
\end{definition}

\begin{thm}\label{equiv}  Let $T\in(0,\infty)$, let $\Phi\in\C([0,\infty))\cap\C^1_{\textrm{loc}}((0,\infty))$ satisfy Assumptions~\ref{as_unique} and \ref{as_equiv}, let $g\in L^2(\TT^d\times[0,T];\R^d)$, let $\rho_0\in \Ent_\Phi(\TT^d)$, and let $\rho\in L^\infty([0,T];L^1(\TT^d))$.  Then, $\rho$ is a renormalized kinetic solution (see Definition~\ref{skel_sol_def}) with control $g$ and initial data $\rho_0$ if and only if $\rho$ is a weak solution (see Definition~\ref{classical_weak}) with control $g$ and initial data $\rho_0$.
\end{thm}

\begin{proof}  The fact that Definition~\ref{skel_sol_def} implies Definition~\ref{classical_weak} is a straightforward consequence of the definitions.  We will therefore only prove that Definition~\ref{classical_weak} implies Definition~\ref{skel_sol_def}.  Assume that $\rho$ is a weak solution.  Let $\mathcal{N}\subset(0,T]$ denote the set of Lebesgue measure zero from \eqref{equiv_2}.  Let $\chi\colon\TT^d\times\R\times[0,T]\rightarrow\R$ denote the kinetic function of $\rho$ and let $p$ denote the parabolic defect measure.  For every $\ve\in(0,1)$ let $\kappa^\ve\colon\TT^d\rightarrow\R$ be a standard symmetric convolution kernel of scale $\ve\in(0,1)$ and let $\rho^\ve\colon\TT^d\times[0,T]\rightarrow\R$ be defined by $\rho^\ve=(\rho * \kappa ^\ve)$.  Let $\psi\in\C^\infty_c(\TT^d\times\R)$ and let $\Psi\colon\TT^d\times\R\rightarrow\R$ be defined by $\Psi(x,\xi)=\int_0^\xi\psi(x,\xi')\dxip$.  Then, for every $t\in[0,T]\setminus\mathcal{N}$,
\begin{equation}\label{equiv_04}\begin{aligned}  \left.\int_{\TT^d}\Psi(x,\rho^\ve(x,r))\dx\right|_{r=0}^{t} & = \int_0^t\int_{\TT^d}\partial_t\rho^\ve\psi(x,\rho^\ve)\dx\dr
\\ & = \int_0^t\int_{(\TT^d)^2}2\Phi^\frac{1}{2}(\rho(y,t))\nabla\Phi^\frac{1}{2}(\rho(y,t))\cdot \nabla_x\kappa^\ve(y-x)\psi(x,\rho^\ve(x,t))\dy\dx\dt
\\ & \quad - \int_0^t\int_{(\TT^d)^2}g(y,t)\Phi^\frac{1}{2}(\rho(y,t))\cdot \nabla_x\kappa^\ve(y-x)\psi(x,\rho^\ve(x,t))\dy\dx\dt.
\end{aligned}\end{equation}
The final two terms on the righthand side of \eqref{equiv_04} are treated identically.  Focusing on the final term of \eqref{equiv_04}, after integrating by parts in the $x$-variable,
\begin{equation}\label{equ_1}\begin{aligned}
& \int_0^t\int_{(\TT^d)^2}g(y,t)\Phi^\frac{1}{2}(\rho(y,t))\cdot \nabla_x\kappa^\ve(y-x)\psi(x,\rho^\ve(x,t))\dy\dx\dt
\\ & = -\int_0^t\int_{(\TT^d)^2}g(y,t)\Phi^\frac{1}{2}(\rho(y,t))\kappa^\ve(y-x)\cdot (\nabla_x\psi)(x,\rho^\ve(x,t))\dy\dx\dt
\\ & \quad -\int_0^t\int_{(\TT^d)^2}g(y,t)\Phi^\frac{1}{2}(\rho(y,t))\kappa^\ve(y-x)(\partial_\xi \psi)(x,\rho^\ve(x,t))\cdot \nabla_x \rho^\ve(x,t)\dy\dx\dt.
\end{aligned}\end{equation}
Since $\Phi^\frac{1}{2}(\rho(x,t))\in L^2([0,T];H^1(\TT^d))$ and $g(x,t)\in L^2(\TT^d\times[0,T];\R^d)$, it follows from $\psi\in\C^\infty_c(\TT^d\times\R)$, the triangle inequality, and the dominated convergence theorem that
\[\begin{aligned}
& \lim_{\ve\rightarrow 0}\left(\int_0^t\int_{(\TT^d)^2}g(y,t)\Phi^\frac{1}{2}(\rho(y,t))\kappa^\ve(y-x)\cdot (\nabla_x\psi)(x,\rho^\ve(x,t))\dy\dx\dt\right)
\\ & =\int_0^t\int_{\TT^d}g(x,t)\Phi^\frac{1}{2}(\rho(x,t))\cdot (\nabla_x\psi)(x,\rho(x,t))\dx\dt.
\end{aligned}\]
The convergence of the final term on the righthand side of \eqref{equ_1} is more difficult.  For this, we will effectively upgrade convergence in $L^1$ to convergence in $L^2$ by bounding the solution $\rho$.  Precisely, since $\psi\in\C^\infty_c(\TT^d\times\R)$ there exists $M\in(0,\infty)$ such that $\Supp(\psi)\subseteq \TT^d\times[-M,M]$, which implies that, for every $\ve\in(0,1)$, $0\leq \rho^\ve\leq M$ on $\Supp[\psi(\cdot,\rho^\ve(\cdot,\cdot))]$.  We will now transfer this bound for the convolution to a bound for the solution itself.  Let $A_0,A_1\subseteq \TT^d\times[0,T]$ be defined by
\begin{equation}\label{equ_005}A_1=\{(y,t)\in\TT^d\times[0,T]\colon \rho(y,t)\geq M+1\}\;\;\textrm{and}\;\;A_0=(\TT^d\times[0,T])\setminus A_1,\end{equation}
we will prove that, returning to the final term of \eqref{equiv_04},
\begin{equation}\label{equ_05}\limsup_{\ve\rightarrow 0}\abs{\int_0^t\int_{(\TT^d)^2}g(y,t)\Phi^\frac{1}{2}(\rho(y,t))\cdot \nabla_x\kappa^\ve(y-x)\psi(x,\rho^\ve(x,t))\mathbf{1}_{A_1}(y,t)\dy\dx\dt}=0.\end{equation}
It will then remain to characterize the limit, after integrating by parts in the $x$-variable,
\begin{equation}\label{equ_5}\begin{aligned}
& \lim_{\ve\rightarrow 0}\left(\int_0^t\int_{(\TT^d)^2}g(y,t)\Phi^\frac{1}{2}(\rho(y,t))\cdot \nabla_x\kappa^\ve(y-x)\psi(x,\rho^\ve(x,t))\mathbf{1}_{A_0}(y,t)\dy\dx\dt\right)
\\ &  =\lim_{\ve\rightarrow 0}\left(-\int_0^t\int_{(\TT^d)^2}g(y,t)\Phi^\frac{1}{2}(\rho(y,t))\kappa^\ve(y-x)\cdot (\nabla_x\psi)(x,\rho^\ve(x,t))\mathbf{1}_{A_0}(y,t)\dy\dx\dt\right.
\\ & \quad \quad \quad \quad \left.-\int_0^t\int_{(\TT^d)^2}g(y,t)\Phi^\frac{1}{2}(\rho(y,t))\kappa^\ve(y-x)(\partial_\xi \psi)(x,\rho^\ve(x,t))\cdot \nabla_x \rho^\ve(x,t)\mathbf{1}_{A_0}(y,t)\dy\dx\dt\right).
\end{aligned}\end{equation}
The point is that, on $A_0$ we have $0\leq \rho(y,t)\leq M+1$ and therefore, as $\ve\rightarrow 0$, strongly in $L^2(\TT^d\times[0,T];\R^d)$,
\[\left(\int_{\TT^d}g(y,t)\Phi^\frac{1}{2}(\rho(y,t))\mathbf{1}_{A_0}(y,t)\kappa^\ve(y-x)\dy\right)\rightarrow g(x,t)\Phi^\frac{1}{2}(\rho(x,t))\mathbf{1}_{A_0}(x,t).\]
The proof of \eqref{equ_05} and the characterization of the limit \eqref{equ_5} will be separated into the case that $\Phi^\frac{1}{2}$ satisfies \eqref{skel_continuity_5} and the case that $\Phi^\frac{1}{2}$ satisfies \eqref{skel_continuity_6} and \eqref{skel_continuity_7}.

\textbf{The case that $\Phi^\frac{1}{2}$ satisfies \eqref{skel_continuity_5}.}  Since $\psi\in\C^\infty_c(\TT^d\times\R)$ there exists $M\in(0,\infty)$ such that $\Supp(\psi)\subseteq \TT^d\times[-M,M]$.  In analogy with \eqref{equ_005}, for every $k\in\N$ we define $A_k\subseteq(\TT^d\times[0,T])$ by
\begin{equation}\label{equiv_49} A_k=\{(y,t)\in\TT^d\times[0,T]\colon \Phi^\frac{1}{2}(\rho(y,t))\geq \Phi^\frac{1}{2}(M)+k\},\end{equation}
and let $A_0=(\TT^d\times[0,T])\setminus A_1$.  Returning to \eqref{equiv_04}, we form the decomposition
\begin{equation}\label{equiv_004}\begin{aligned} & \left.\int_{\TT^d}\Psi(x,\rho^\ve(x,r))\dx\right|_{r=0}^{t}
\\ & = \int_0^t\int_{(\TT^d)^2}2\Phi^\frac{1}{2}(\rho(y,t))\nabla\Phi^\frac{1}{2}(\rho(y,t))\cdot \nabla_x\kappa^\ve(y-x)\psi(x,\rho^\ve(x,t))\mathbf{1}_{A_1}(y,t)\dy\dx\dt
\\ & \quad - \int_0^t\int_{(\TT^d)^2}g(y,t)\Phi^\frac{1}{2}(\rho(y,t))\cdot \nabla_x\kappa^\ve(y-x)\psi(x,\rho^\ve(x,t))\mathbf{1}_{A_1}(y,t)\dy\dx\dt
\\ & \quad + \int_0^t\int_{(\TT^d)^2}2\Phi^\frac{1}{2}(\rho(y,t))\nabla\Phi^\frac{1}{2}(\rho(y,t))\cdot \nabla_x\kappa^\ve(y-x)\psi(x,\rho^\ve(x,t))\mathbf{1}_{A_0}(y,t)\dy\dx\dt
\\ & \quad - \int_0^t\int_{(\TT^d)^2}g(y,t)\Phi^\frac{1}{2}(\rho(y,t))\cdot \nabla_x\kappa^\ve(y-x)\psi(x,\rho^\ve(x,t))\mathbf{1}_{A_0}(y,t)\dy\dx\dt.
\end{aligned}\end{equation}
We will first prove that the first two terms on the righthand side of \eqref{equiv_004} vanish as $\ve\rightarrow 0$, which follows by choosing $F=g$ and $F=\nabla \Phi^\frac{1}{2}(\rho)$ in Lemma~\ref{aux_lemma_1} below.  It will remain to treat the final terms on the righthand side of \eqref{equiv_004}.

\begin{lem}\label{aux_lemma_1} Let $\Phi\in\C([0,\infty))\cap\C^1_{\textrm{loc}}((0,\infty))$ satisfy Assumption~\ref{as_unique} and \eqref{skel_continuity_5}.  Then, for every $F\in L^2(\TT^d\times[0,T];\R^d)$ and $\psi\in\C^\infty_c(\TT^d\times\R)$,
\begin{equation}\label{equiv_48} \limsup_{\ve\rightarrow 0}\abs{\int_0^t\int_{(\TT^d)^2}F(y,t)\Phi^\frac{1}{2}(\rho(y,t))\cdot \nabla_x\kappa^\ve(y-x)\psi(x,\rho^\ve(x,t))\mathbf{1}_{A_1}(y,t)\dy\dx\dt}=0.\end{equation}
\end{lem}
\begin{proof}  Let $\ve\in(0,1)$.  It follows from H\"older's inequality and \eqref{equiv_49} that
\begin{equation}\label{equiv_50}\begin{aligned}
& \abs{\int_0^t\int_{(\TT^d)^2}F(y,t)\Phi^\frac{1}{2}(\rho(y,t))\cdot \nabla_x\kappa^\ve(y-x)\psi(x,\rho^\ve(x,t))\mathbf{1}_{A_1}(y,t)\dy\dx\dt}
\\ & \leq \left(\int_0^t\int_{(\TT^d)^2}\abs{F}^2(y,t)\abs{\ve\nabla_x\kappa^\ve(y-x)}\abs{\psi(x,\rho^\ve(x,t))}\mathbf{1}_{A_1}(y,t)\dy\dx\dt\right)^\frac{1}{2}
\\ & \quad \cdot \ve^{-1}\left(\int_0^t\int_{(\TT^d)^2}\Phi(\rho(y,t))\abs{\ve\nabla_x\kappa^\ve(y-x)}\abs{\psi(x,\rho^\ve(x,t))}\mathbf{1}_{A_1}(y,t)\dy\dx\dt\right)^\frac{1}{2}.
\end{aligned}\end{equation}
We will first prove that the final term on the righthand side of \eqref{equiv_50} remains bounded.

For every $k\in\N$ let $\mathbf{1}_{k,k+1}\colon\TT^d\times[0,T]\rightarrow\{0,1\}$ denote the indicator function of $A_k\setminus A_{k+1}$.  It follows from the change of variables formula, the fact that $\Phi$ is increasing, and the definition of the $\{A_k\}_{k\in\N}$ that
\begin{equation}\label{equiv_51}\begin{aligned}
& \int_0^t\int_{(\TT^d)^2}\Phi(\rho(y,t))\abs{\ve\nabla_x\kappa^\ve(y-x)}\abs{\psi(x,\rho^\ve(x,t))}\mathbf{1}_{A_1}(y,t)\dy\dx\dt
\\ & = \sum_{k=1}^\infty \int_0^t\int_{(\TT^d)^2}\Phi(\rho(y,t))\abs{\ve\nabla_x\kappa^\ve(x)}\abs{\psi(y+x,\rho^\ve(y+x,t))}\mathbf{1}_{k,k+1}(y,t)\dy\dx\dt
\\ & \leq  \sum_{k=1}^\infty \int_0^t\int_{(\TT^d)^2}\left(\Phi^\frac{1}{2}(M)+k+1\right)^2\abs{\ve\nabla_x\kappa^\ve(x)}\abs{\psi(y+x,\rho^\ve(y+x,t))}\mathbf{1}_{k,k+1}(y,t)\dy\dx\dt.
\end{aligned}\end{equation}
For every $x\in\TT^d$ and $k\in\N$ let $B^\ve_{k,x}\subseteq\left(\TT^d\times[0,T]\right)$ be defined by
\begin{equation}\label{equiv_052}B^\ve_{k,x}=\left\{(y,t)\in\TT^d\times[0,T]\colon \abs{\Phi^\frac{1}{2}(\rho(y,t))-\left(\Phi^\frac{1}{2}(\rho)\right)^\ve(y+x,t)}\mathbf{1}_{k,k+1}(y,t)\geq k\right\},\end{equation}
where $(\Phi^\frac{1}{2}(\rho))^\ve=(\Phi^\frac{1}{2}(\rho)*\kappa^\ve)$.  Let $x\in\TT^d$ and $k\in\N$.  Since $\Phi^\frac{1}{2}$ is concave, Jensen's inequality proves that
\[\left(\Phi^\frac{1}{2}(\rho)\right)^\ve(y+x,t)\leq \Phi^\frac{1}{2}(\rho^\ve(y+x,t)).\]
Therefore, for every $(y,t)\in \left(\left(A_k\setminus A_{k+1}\right)\cap \Supp[\psi(\cdot+x,\rho^\ve(\cdot+x,\cdot))]\right)$,
\begin{equation}\label{equiv_53}\Phi^\frac{1}{2}(\rho(y,t))-\left(\Phi^\frac{1}{2}(\rho)\right)^\ve(y+x,t)\geq \Phi^\frac{1}{2}(\rho(y,t))-\Phi^\frac{1}{2}(\rho^\ve(y+x,t))\geq k.\end{equation}
It follows from \eqref{equiv_052} and \eqref{equiv_53} that
\begin{equation}\label{equiv_54}\left(\left(A_k\setminus A_{k+1}\right)\cap \Supp[\psi(\cdot+x,\rho^\ve(\cdot+x,\cdot))]\right)\subseteq B^\ve_{k,x}.\end{equation}
Returning to \eqref{equiv_51}, it follows from $\psi\in\C^\infty_c(\TT^d\times\R)$ and \eqref{equiv_54} that there exists $c\in(0,\infty)$ such that
\begin{equation}\label{equiv_55}\begin{aligned}
& \int_0^t\int_{\TT^d}\int_{A_1}\Phi(\rho(y,t))\abs{\ve\nabla_x\kappa^\ve(y-x)}\abs{\psi(x,\rho^\ve(x,t))}\dy\dx\dt
\\ & \leq c\int_{\TT^d}\left(\sum_{k=1}^\infty\left(\Phi^\frac{1}{2}(M)+k+1\right)^2\abs{B^\ve_{k,x}}\right)\abs{\ve\nabla\kappa^\ve(x)}\dx.
\end{aligned}\end{equation}
It remains to estimate the measure of the sets $\{B^\ve_{k,x}\}_{k\in\N,x\in\TT^d}$.

Let $x\in\TT^d$ and $k\in\N$.  It follows from Chebyshev's inequality, Jensen's inequality, and the fundamental theorem of calculus that, omitting the integration variables,
\begin{equation}\label{equiv_56}\begin{aligned}
\abs{B^\ve_{k,x}} & \leq \frac{1}{k^2}\int_0^t\int_{\TT^d}\abs{\Phi^\frac{1}{2}(\rho(y,t))-\left(\Phi^\frac{1}{2}(\rho)\right)^\ve(y+x,t)}^2\mathbf{1}_{k,k+1}(y,t)
\\ & \leq  \frac{1}{k^2}\int_0^t\int_{(\TT^d)^2}\abs{\Phi^\frac{1}{2}(\rho(y,t))-\Phi^\frac{1}{2}(\rho(y',t))}^2\kappa^\ve(y+x-y')\mathbf{1}_{k,k+1}(y,t)
\\ & \leq \frac{1}{k^2}\int_0^t\int_{(\TT^d)^2}\left(\int_0^1\abs{\nabla(\Phi^\frac{1}{2}(\rho))(sy+(1-s)y',t)}^2\ds\right)\abs{y-y'}^2\kappa^\ve(y+x-y')\mathbf{1}_{k,k+1}(y,t).
\end{aligned}\end{equation}
On the support of the convolution kernel, since the triangle inequality proves that $\abs{y-y'}\leq \abs{x}+\ve$, it follows from \eqref{equiv_56} that there exists $c\in(0,\infty)$ such that, omitting the integration variables,
\begin{equation}\label{equiv_58}
\abs{B^\ve_{k,x}} \leq \frac{c(\ve^2+\abs{x}^2)}{k^2}\int_0^t\int_{(\TT^d)^2}\left(\int_0^1\abs{\nabla(\Phi^\frac{1}{2}(\rho))(sy+(1-s)y',t)}^2\ds\right)\kappa^\ve(y+x-y')\mathbf{1}_{k,k+1}(y,t).
\end{equation}
Returning to \eqref{equiv_55}, it follows from \eqref{equiv_58} that there exists $c\in(0,\infty)$ depending on $M\in(0,\infty)$ such that, omitting the integration variables,
\begin{equation}\label{equiv_59}\begin{aligned}
\sum_{k=1}^\infty & \left(\Phi^\frac{1}{2}(M)+k+1\right)^2\abs{B^\ve_{k,x}}  \leq \sum_{k=1}^\infty\left(\frac{c(\ve^2+\abs{x}^2)(\Phi^\frac{1}{2}(M)+k+1)^2}{k^2}\right.
\\ & \quad \cdot \left.\int_0^t\int_{(\TT^d)^2}\left(\int_0^1\abs{\nabla(\Phi^\frac{1}{2}(\rho))(sy+(1-s)y',t)}^2\ds\right)\kappa^\ve(y+x-y')\mathbf{1}_{k,k+1}(y,t)\right)
\\ & \leq c\left(\ve^2+\abs{x}^2\right)\int_0^t\int_{(\TT^d)^2}\left(\int_0^1\abs{\nabla(\Phi^\frac{1}{2}(\rho))(sy+(1-s)y',t)}^2\ds\right)\kappa^\ve(y+x-y')\mathbf{1}_{A_1}(y,t)
\\ & \leq c\left(\ve^2+\abs{x}^2\right)\norm{\nabla\Phi^\frac{1}{2}(\rho)}^2_{L^2(\TT^d\times[0,T])}.
\end{aligned}\end{equation}
Returning to \eqref{equiv_55}, it follows from the definition of the convolution kernel and \eqref{equiv_59} that, for $c\in(0,\infty)$,
\begin{equation}\label{equiv_60}
\int_0^t\int_{\TT^d}\int_{A_1}\Phi(\rho(y,t))\abs{\ve\nabla_x\kappa^\ve(y-x)}\abs{\psi(x,\rho^\ve(x,t))}\dy\dx\dt \leq c\ve^2\norm{\nabla\Phi^\frac{1}{2}(\rho)}^2_{L^2(\TT^d\times[0,T])}.
\end{equation}
Returning to \eqref{equiv_50}, it follows from \eqref{equiv_60} that there exists $c\in(0,\infty)$ such that
\begin{equation}\label{equiv_61}\begin{aligned}
& \abs{\int_0^t\int_{(\TT^d)^2}F(y,t)\Phi^\frac{1}{2}(\rho(y,t))\cdot \nabla_x\kappa^\ve(y-x)\psi(x,\rho^\ve(x,t))\mathbf{1}_{A_1}(y,t)\dy\dx\dt}
\\ & \leq c\norm{\nabla\Phi^\frac{1}{2}(\rho)}_{L^2(\TT^d\times[0,T])}\left(\int_0^t\int_{(\TT^d)^2}\abs{F}^2(y,t)\abs{\ve\nabla_x\kappa^\ve(y-x)}\abs{\psi(x,\rho^\ve(x,t))}\mathbf{1}_{A_1}(y,t)\dy\dx\dt\right)^\frac{1}{2}.
\end{aligned}\end{equation}
Since for $c_0=\int_{\TT^d}\abs{\nabla\kappa(x)}\dx$ the functions $\{c_0^{-1}\abs{\ve\nabla\kappa^\ve}\}_{\ve\in(0,1)}$ are a Dirac sequence on $\TT^d$ and since $\psi\in\C^\infty_c(\TT^d\times\R)$, it follows from the definition of $A_1$ and the dominated convergence theorem that
\begin{equation}\label{equiv_00062} \limsup_{\ve\rightarrow 0}\left(\int_0^t\int_{(\TT^d)^2}\abs{F}^2(y,t)\abs{\ve\nabla_x\kappa^\ve(y-x)}\abs{\psi(x,\rho^\ve(x,t))}\mathbf{1}_{A_1}(y,t)\dy\dx\dt\right)=0,\end{equation}
which in combination with \eqref{equiv_61} completes the proof of \eqref{equiv_48}. \end{proof}

We will now analyze the final two terms on the righthand side of \eqref{equiv_004}.  Let $\ve\in(0,1)$.  The chain rule proves that, after integrating by parts in the $x$-variable,
\begin{equation}\label{equiv_4}\begin{aligned}
 & \int_0^t\int_{(\TT^d)^2}2\Phi^\frac{1}{2}(\rho(y,t))\nabla\Phi^\frac{1}{2}(\rho(y,t))\cdot \nabla_x\kappa^\ve(y-x)\psi(x,\rho^\ve(x,t))\mathbf{1}_{A_0}(y,t)\dy\dx\dt
\\ & \quad - \int_0^t\int_{(\TT^d)^2}g(y,t)\Phi^\frac{1}{2}(\rho(y,t))\cdot \nabla_x\kappa^\ve(y-x)\psi(x,\rho^\ve(x,t))\mathbf{1}_{A_0}(y,t)\dy\dx\dt
\\ & =  -\int_0^t\int_{(\TT^d)^2}2\Phi^\frac{1}{2}(\rho(y,t))\nabla\Phi^\frac{1}{2}(\rho(y,t))\kappa^\ve(y-x)\cdot \left(\nabla_x\psi\right)(x,\rho^\ve(x,t))\mathbf{1}_{A_0}(y,t)\dy\dx\dt
\\ & \quad -\int_0^t\int_{(\TT^d)^2}2\Phi^\frac{1}{2}(\rho(y,t))\nabla\Phi^\frac{1}{2}(\rho(y,t))\kappa^\ve(y-x)\left(\partial_\xi\psi\right)(x,\rho^\ve(x,t))\cdot \nabla_x\rho^\ve(x,t)\mathbf{1}_{A_0}(y,t)\dy\dx\dt
\\ & \quad +\int_0^t\int_{(\TT^d)^2}g(y,t)\Phi^\frac{1}{2}(\rho(y,t))\kappa^\ve(y-x)\cdot \left(\nabla_x\psi\right)(x,\rho^\ve(x,t))\mathbf{1}_{A_0}(y,t)\dy\dx\dt
\\ & \quad +\int_0^t\int_{(\TT^d)^2}g(y,t)\Phi^\frac{1}{2}(\rho(y,t))\kappa^\ve(y-x)\left(\partial_\xi\psi\right)(x,\rho^\ve(x,t))\cdot \nabla_x\rho^\ve(x,t)\mathbf{1}_{A_0}(y,t)\dy\dx\dt.
\end{aligned}\end{equation}
Since $\Phi^\frac{1}{2}(\rho(x,t))\in L^2([0,T];H^1(\TT^d))$ and $g(x,t)\in L^2(\TT^d\times[0,T];\R^d)$, it follows from the definition of $A_0$, $\psi\in\C^\infty_c(\TT^d\times\R)$, and the dominated convergence theorem that the first and third terms appearing on the righthand side of \eqref{equiv_4} satisfy
\begin{equation}\label{equiv_75}\begin{aligned}
& \lim_{\ve\rightarrow 0}\left(\int_0^t\int_{(\TT^d)^2}2\Phi^\frac{1}{2}(\rho(y,t))\nabla\Phi^\frac{1}{2}(\rho(y,t))\kappa^\ve(y-x)\cdot \left(\nabla_x\psi\right)(x,\rho^\ve(x,t))\mathbf{1}_{A_0}(y,t)\dy\dx\dt\right)
\\ & = \int_0^t\int_{\TT^d}2\Phi^\frac{1}{2}(\rho(x,t))\nabla\Phi^\frac{1}{2}(\rho(x,t))\cdot \left(\nabla_x\psi\right)(x,\rho(x,t))\dx\dt,
\end{aligned}\end{equation}
and
\begin{equation}\label{equiv_76}\begin{aligned}
& \lim_{\ve\rightarrow 0}\left(\int_0^t\int_{(\TT^d)^2}g(y,t)\Phi^\frac{1}{2}(\rho(y,t))\kappa^\ve(y-x)\cdot \left(\nabla_x\psi\right)(x,\rho^\ve(x,t))\mathbf{1}_{A_0}(y,t)\dy\dx\dt\right)
\\ & = \int_0^t\int_{\TT^d}g(y,t)\Phi^\frac{1}{2}(\rho(x,t))\cdot \left(\nabla_x\psi\right)(x,\rho(x,t))\dx\dt.
\end{aligned}\end{equation}
The second and fourth terms appearing on the righthand side of \eqref{equiv_4} are handled by choosing $F=g$ and $F=\nabla\Phi^\frac{1}{2}(\rho)$ in Lemma~\ref{aux_lemma_3} below.

\begin{lem}\label{aux_lemma_3} Let $\Phi\in\C([0,\infty))\cap\C^1_{\textrm{loc}}((0,\infty))$ satisfy Assumption~\ref{as_unique} and \eqref{skel_continuity_5}.  Then, for every $F\in L^2(\TT^d\times[0,T];\R^d)$ and $\psi\in\C^\infty_c(\TT^d\times\R)$,
\begin{equation}\label{al_1}\begin{aligned}
& \lim_{\ve\rightarrow 0}\left(\int_0^t\int_{(\TT^d)^2}F(y,t)\Phi^\frac{1}{2}(\rho(y,t))\kappa^\ve(y-x)\left(\partial_\xi\psi\right)(x,\rho^\ve(x,t))\cdot \nabla_x\rho^\ve(x,t)\mathbf{1}_{A_0}(y,t)\dy\dx\dt\right)
\\ & = \int_0^t\int_{\TT^d}F(y,t)\frac{2\Phi(\rho(y,t))}{\Phi'(\rho(y,t))}\left(\partial_\xi\psi\right)(y,\rho(y,t))\cdot \nabla\Phi^\frac{1}{2}(\rho(y,t))\dy\dt.
\end{aligned}\end{equation}
\end{lem}

\begin{proof}  We will first make a preliminary calculation.  Let $(\Phi^\frac{1}{2})^{-1}$ denote the function inverse to $\Phi^\frac{1}{2}$.  It follows from \eqref{skel_continuity_5} that $(\Phi^\frac{1}{2})^{-1}$ is locally Lipschitz continuous since, for every $\xi\in[0,\infty)$,
\begin{equation}\label{equiv_6}\frac{\dd}{\dxi}\left[\left(\Phi^\frac{1}{2}\right)^{-1}\right]\left(\Phi^\frac{1}{2}(\xi)\right)=\frac{1}{\left(\Phi^\frac{1}{2}\right)'(\xi)}=\frac{2\Phi^\frac{1}{2}(\xi)}{\Phi'(\xi)}.\end{equation}
The chain rule, the fact that $\Phi^\frac{1}{2}(\rho)\in L^2([0,T];H^1(\TT^d))$, and \eqref{equiv_6} prove that
\begin{equation}\label{equiv_3}\begin{aligned}
\nabla \rho^\ve(x,t)  & = \int_{\TT^d}\rho(z,t)\nabla_x \kappa^\ve(z-x)\dz = \int_{\TT^d}\frac{2\Phi^\frac{1}{2}(\rho(z,t))}{\Phi'(\rho(z,t))}\nabla\left(\Phi^\frac{1}{2}(\rho(z,t))\right)\kappa^\ve(z-x)\dz.
\end{aligned}\end{equation}
It follows from  H\"older's inequality and the nonnegativity of $\rho$ that there exists $c\in(0,\infty)$ such that, for $p\in[2,\infty)$ as in \eqref{skel_continuity_5}, for every $(x,t)\in\TT^d\times[0,T]$, 
\begin{equation}\label{equiv_63}\begin{aligned}
& \abs{\int_{\TT^d}\frac{2\Phi^\frac{1}{2}(\rho(z,t))}{\Phi'(\rho(z,t))}\nabla\left(\Phi^\frac{1}{2}(\rho(z,t))\right)\kappa^\ve(z-x)\dz}
\\  & \leq 2\left(\int_{\TT^d}\left(\frac{\Phi^\frac{1}{2}(\rho(z,t))}{\Phi'(\rho(z,t))}\right)^p\kappa^\ve(z-x)\dz\right)^\frac{1}{p}\left(\int_{\TT^d}\abs{\nabla\Phi^\frac{1}{2}(\rho(z,t))}^2\kappa^\ve(z-x)\dz\right)^\frac{1}{2}
\\ & \leq c\left(\int_{\TT^d}\left(\rho(z,t)+1\right)\kappa^\ve(z-x)\dz\right)^\frac{1}{p}\left(\int_{\TT^d}\abs{\nabla\Phi^\frac{1}{2}(\rho(z,t))}^2\kappa^\ve(z-x)\dz\right)^\frac{1}{2}
\\ & = c \left(\rho^\ve(x,t)+1\right)^\frac{1}{p}\left(\int_{\TT^d}\abs{\nabla\Phi^\frac{1}{2}(\rho(z,t))}^2\kappa^\ve(z-x)\dz\right)^\frac{1}{2}.
\end{aligned}\end{equation}
We will now prove the statement.  Let $\ve\in(0,1)$.  It follows from \eqref{equiv_3} that
\begin{equation}\label{equiv_065}\begin{aligned}
& \int_0^t\int_{(\TT^d)^2}F(y,t)\Phi^\frac{1}{2}(\rho(y,t))\kappa^\ve(y-x)\left(\partial_\xi\psi\right)(x,\rho^\ve(x,t))\cdot \nabla_x\rho^\ve(x,t)\mathbf{1}_{A_0}(y,t)\dy\dx\dt
\\ &=\int_0^t\int_{(\TT^d)^2}F(y,t)\Phi^\frac{1}{2}(\rho(y,t))\kappa^\ve(y-x)\left(\partial_\xi\psi\right)(x,\rho^\ve(x,t))
\\ &\quad \quad \quad \quad \quad \quad \cdot \Big(\int_{\TT^d}\frac{2\Phi^\frac{1}{2}(\rho(z,t))}{\Phi'(\rho(z,t))}\nabla\Phi^\frac{1}{2}(\rho(z,t))\kappa^\ve(z-x)\dz\Big)\mathbf{1}_{A_0}(y,t)\dy\dx\dt.
\end{aligned}\end{equation}
The integral with respect to the $z$-variable will be decomposed in terms of $A_0$ and $A_1$.  We will first prove that
\begin{equation}\label{equiv_00070}\begin{aligned}
& \lim_{\ve\rightarrow 0}\left|\int_0^t\int_{(\TT^d)^2}\Big(F(y,t)\Phi^\frac{1}{2}(\rho(y,t))\kappa^\ve(y-x)\left(\partial_\xi\psi\right)(x,\rho^\ve(x,t))\mathbf{1}_{A_0}(y,t)\right.
\\ & \quad \quad \quad \quad \quad \quad \cdot \left.\int_{\TT^d}\frac{2\Phi^\frac{1}{2}(\rho(z,t))}{\Phi'(\rho(z,t))}\nabla\Phi^\frac{1}{2}(\rho(z,t))\kappa^\ve(z-x)\mathbf{1}_{A_1}(z,t)\dz\Big)\dy\dx\dt\right|=0.
\end{aligned}\end{equation}
Let $\ve\in(0,1)$.  After repeating the derivation leading to \eqref{equiv_63}, for $p\in[2,\infty)$ satisfying \eqref{skel_continuity_5}, there exists $c\in(0,\infty)$ such that \eqref{equiv_00070} is bounded by
\begin{equation}\label{equiv_00068}\begin{aligned}
&c\int_0^t\int_{(\TT^d)^2}\Big(\abs{F(y,t)}\Phi^\frac{1}{2}(\rho(y,t))\kappa^\ve(y-x)\abs{\left(\partial_\xi\psi\right)(x,\rho^\ve(x,t))}(\rho^\ve(x,t)+1)^\frac{1}{p}\mathbf{1}_{A_0}(y,t)
\\ & \quad \quad \quad \quad \quad \quad \cdot \Big(\int_{\TT^d}\abs{\nabla\Phi^\frac{1}{2}(\rho(z,t))}^2\kappa^\ve(z-x)\mathbf{1}_{A_1}(z,t)\dz\Big)^\frac{1}{2}\Big)\dy\dx\dt.
\end{aligned}\end{equation}
Then, since $\psi\in\C^\infty_c(\TT^d\times\R)$ implies that $\rho^\ve(x,t)$ is bounded on the support of $\psi(x,\rho^\ve(x,t))$, it follows from H\"older's inequality and the definition of $A_0$ that there exists $c\in(0,\infty)$ such that \eqref{equiv_00068} is bounded by
\begin{equation}\label{equiv_69}\begin{aligned}
& c\left(\int_0^t\int_{(\TT^d)^2}\abs{F}^2(y,t)\abs{\left(\partial_\xi\psi\right)(x,\rho^\ve(x,t))}\kappa^\ve(y-x)\mathbf{1}_{A_0}(y,t)\dy\dx\dt\right)^\frac{1}{2}
\\ & \cdot \left(\int_0^t\int_{(\TT^d)^3}\abs{\nabla\Phi^\frac{1}{2}(\rho(z,t))}^2\abs{\left(\partial_\xi\psi\right)(x,\rho^\ve(x,t))}\kappa^\ve(z-x)\kappa^\ve(y-x)\mathbf{1}_{A_0}(y,t)\mathbf{1}_{A_1}(z,t)\dx\dy\dz\dt\right)^\frac{1}{2}
\\ & \leq c\norm{F}_{L^2(\TT^d\times[0,T])}\left(\int_0^t\int_{(\TT^d)^2}\abs{\nabla\Phi^\frac{1}{2}(\rho(z,t))}^2\abs{\left(\partial_\xi\psi\right)(x,\rho^\ve(x,t))}\kappa^\ve(z-x)\mathbf{1}_{A_0}(y,t)\dx\dz\dt\right)^\frac{1}{2}.
\end{aligned}\end{equation}
Since, as $\ve\rightarrow 0$, the definition of $A_1$ and $\psi\in\C^\infty_c(\TT^d\times\R)$ prove that
\begin{equation}\label{equiv_70}\begin{aligned}
& \lim_{\ve\rightarrow 0}\left(\int_0^t\int_{(\TT^d)^2}\abs{\nabla\Phi^\frac{1}{2}(\rho(z,t))}^2\abs{\left(\partial_\xi\psi\right)(x,\rho^\ve(x,t))}\kappa^\ve(z-x)\mathbf{1}_{A_1}(z,t)\dx\dz\dt\right)
\\ & = \int_0^t\int_{\TT^d}\abs{\nabla\Phi^\frac{1}{2}(\rho(z,t))}^2\abs{\left(\partial_\xi\psi\right)(z,\rho(z,t))}\mathbf{1}_{A_1}(z,t)\dz\dt = 0,
\end{aligned}\end{equation}
in combination \eqref{equiv_69} and \eqref{equiv_70} complete the proof of \eqref{equiv_00070}.

For the integral over $A_0$ in the $z$-variable, since the definition of $A_0$ implies that $\nicefrac{2\Phi^\frac{1}{2}(\rho)}{\Phi'(\rho)}\mathbf{1}_{A_0}$ is bounded, it follows from $\psi\in\C^\infty_c(\TT^d\times[0,T])$, $\Phi^\frac{1}{2}(\rho)\in L^2([0,T];H^1(\TT^d))$, $F\in L^2(\TT^d\times[0,T];\R^d)$, H\"older's inequality, and the dominated convergence theorem that
\begin{equation}\label{equiv_71}\begin{aligned}
& \lim_{\ve\rightarrow 0}\left(\int_0^t\int_{(\TT^d)^2}F(y,t) \Phi^\frac{1}{2}(\rho(y,t))\kappa^\ve(y-x)\left(\partial_\xi\psi\right)(x,\rho^\ve(x,t))\mathbf{1}_{A_0}(y,t)\right.
\\ &\quad\quad\quad\quad\quad\quad \cdot \left.\Big(\int_{\TT^d}\frac{2\Phi^\frac{1}{2}(\rho(z,t))}{\Phi'(\rho(z,t))}\cdot \nabla\Phi^\frac{1}{2}(\rho(z,t))\kappa^\ve(z-x)\mathbf{1}_{A_0}(z,t)\dz\Big)\dy\dx\dt\right)
\\ & = 2\int_0^t\int_{\TT^d}F(x,t)\frac{2\Phi(\rho(x,t))}{\Phi'(\rho(x,t))}\cdot \nabla\Phi(\rho(x,t))(\partial_\xi\psi)(x,\rho(x,t))\dx\dt.
\end{aligned}\end{equation}
In combination \eqref{equiv_065}, \eqref{equiv_00070}, \eqref{equiv_71}, and $\TT^d\times[0,T]=A_0\cup A_1$ complete the proof of \eqref{al_1}.  \end{proof}

Returning to \eqref{equiv_004}, it follows from $\TT^d\times[0,T]=A_0\cup A_1$, Lemma~\ref{aux_lemma_1}, \eqref{equiv_75}, \eqref{equiv_76}, and Lemma~\ref{aux_lemma_3} that
\[\begin{aligned}
\left.\int_{\TT^d}\Psi(x,\rho(x,r))\dx\right|_{r=0}^{t} & = \lim_{\ve\rightarrow 0} \left.\int_{\TT^d}\Psi(x,\rho^\ve(x,r))\dx\right|_{r=0}^{t}
\\ & =  -\int_0^t\int_{\TT^d}2\Phi^\frac{1}{2}(\rho(x,t))\nabla\Phi^\frac{1}{2}(\rho(x,t))\cdot\left(\nabla_x\psi\right)(x,\rho(x,t))\dx\dt
\\ & \quad -\int_0^t\int_{\TT^d}\frac{4\Phi(\rho(x,t))}{\Phi'(\rho(x,t))}\abs{\nabla\Phi^\frac{1}{2}(\rho(x,t))}^2\left(\partial_\xi\psi\right)(x,\rho(x,t))\dx\dt
\\ & \quad +\int_0^t\int_{\TT^d}g(x,t)\Phi^\frac{1}{2}(\rho(x,t))\cdot \left(\nabla_x\psi\right)(x,\rho(x,t))\dx\dt
\\ & \quad +\int_0^t\int_{\TT^d}g(x,t)\frac{2\Phi(\rho(x,t))}{\Phi'(\rho(x,t))}\left(\partial_\xi\psi\right)(x,\rho(x,t))\cdot \nabla\Phi^\frac{1}{2}(\rho(x,t))\dx\dt.
\end{aligned}\]
It then follows from the definition of the kinetic function, Lemma~\ref{ibp} after integrating by parts in the $x$-variable, and the definition of the parabolic defect measure that $\rho$ is a renormalized kinetic solution of \eqref{skel_eq}.  This completes the proof in the case that $\Phi$ satisfies \eqref{skel_continuity_5}.

\textbf{The case that $\Phi^\frac{1}{2}$ satisfies \eqref{skel_continuity_6} and \eqref{skel_continuity_7}.}  In this case, we form a slightly different decomposition.  For every $k\in\N$ let $A_k\subseteq\left(\TT^d\times[0,T]\right)$ be defined by
\[ A_k=\{(y,t)\in\TT^d\times[0,T]\colon \rho(y,t)\geq M+k\},\]
and let $A_0=(\TT^d\times[0,T])\setminus A_1$.  The proof now follows identically to the previous case, provided we can prove the analogues of Lemmas~\ref{aux_lemma_1} and \ref{aux_lemma_3}.  We prove these statements in Lemmas~\ref{aux_lemma_2} and \ref{aux_lemma_4} below, which completes the proof. \end{proof}

\begin{lem}\label{aux_lemma_2} Let $\Phi\in\C([0,\infty))\cap\C^1_{\textrm{loc}}((0,\infty))$ satisfy Assumption~\ref{as_unique}, \eqref{skel_continuity_6}, and \eqref{skel_continuity_7}.  Then, for every $F\in L^2(\TT^d\times[0,T];\R^d)$ and $\psi\in\C^\infty_c(\TT^d\times\R)$,
\begin{equation}\label{equiv_30} \limsup_{\ve\rightarrow 0}\abs{\int_0^t\int_{(\TT^d)^2}F(y,t)\Phi^\frac{1}{2}(\rho(y,t))\cdot \nabla_x\kappa^\ve(y-x)\psi(x,\rho^\ve(x,t))\mathbf{1}_{A_1}(y,t)\dy\dx\dt}=0.\end{equation}
\end{lem}

\begin{proof} Let $\ve\in(0,1)$.  H\"older's inequality proves that
\begin{equation}\label{equiv_32}\begin{aligned}
& \abs{\int_0^t\int_{(\TT^d)^2}F(y,t)\Phi^\frac{1}{2}(\rho(y,t))\cdot \nabla_x\kappa^\ve(y-x)\psi(x,\rho^\ve(x,t))\mathbf{1}_{A_1}(y,t)\dy\dx\dt}
\\ & \leq \left(\int_0^t\int_{(\TT^d)^2}\abs{F}^2(y,t)\abs{\ve\nabla_x\kappa^\ve(y-x)}\abs{\psi(x,\rho^\ve(x,t))}\mathbf{1}_{A_1}(y,t)\dy\dx\dt\right)^\frac{1}{2}
\\ & \quad \cdot \ve^{-1}\left(\int_0^t\int_{(\TT^d)^2}\Phi(\rho(y,t))\abs{\ve\nabla_x\kappa^\ve(y-x)}\abs{\psi(x,\rho^\ve(x,t))}\mathbf{1}_{A_1}(y,t)\dy\dx\dt\right)^\frac{1}{2}.
\end{aligned}\end{equation}
We will first prove that the final term on the righthand side of \eqref{equiv_32} remains bounded.

For every $k\in\N$, let $\mathbf{1}_{k,k+1}\colon\TT^d\times[0,T]\rightarrow\{0,1\}$ denote the indicator function of the set $A_{k}\setminus A_{k+1}$.  The change of variables formula, the definition of the $\{A_k\}_{k\in\N}$, and the fact that $\Phi$ is increasing prove that
\[\begin{aligned}
& \int_0^t\int_{(\TT^d)^2}\Phi(\rho(y,t))\abs{\ve\nabla_x\kappa^\ve(y-x)}\abs{\psi(x,\rho^\ve(x,t))}\mathbf{1}_{A_1}(y,t)\dy\dx\dt
\\ & = \sum_{k=1}^\infty \int_0^t\int_{(\TT^d)^2}\Phi(\rho(y,t))\abs{\ve\nabla_x\kappa^\ve(x)}\abs{\psi(y+x,\rho^\ve(y+x,t))}\mathbf{1}_{k,k+1}(y,t)\dy\dx\dt
\\ & \leq \sum_{k=1}^\infty \int_0^t\int_{(\TT^d)^2}\Phi(M+k+1)\abs{\ve\nabla_x\kappa^\ve(x)}\abs{\psi(y+x,\rho^\ve(y+x,t))}\mathbf{1}_{k,k+1}(y,t)\dy\dx\dt.
\end{aligned}\]
For every $x\in \TT^d$ and $k\in\N$ let $B^\ve_{k,x}\subseteq\left(\TT^d\times[0,T]\right)$ be defined by
\[B^\ve_{k,x}=\left\{(y,t)\in \TT^d\times[0,T]\colon \abs{\rho(y,t)-\rho^\ve(y+x,t)}\mathbf{1}_{k,k+1}(y,t)\geq k\right\}.\]
The definition of the $\{B^\ve_{k,x}\}_{k\in\N,x\in\TT^d}$ and the definition of $M\in(0,\infty)$ prove for every $k\in\N$ and $x\in\TT^d$ that
\begin{equation}\label{equiv_34}\left((A_k\setminus A_{k+1})\cap\Supp[\psi(\cdot+x,\rho^\ve(\cdot+x,\cdot))]\right)\subseteq B^\ve_{k,x}.\end{equation}
Since $\psi\in\C^\infty_c(\TT^d\times\R)$ it follows from \eqref{equiv_34} that there exists $c\in(0,\infty)$ such that
\begin{equation}\label{equiv_35}\begin{aligned}
& \int_0^t\int_{(\TT^d)^2}\Phi(\rho(y,t))\abs{\ve\nabla_x\kappa^\ve(y-x)}\abs{\psi(x,\rho^\ve(x,t))}\mathbf{1}_{A_1}(y,t)\dy\dx\dt
\\ & \leq c\int_{\TT^d}\left(\sum_{k=1}^\infty\Phi(M+k+1)\abs{B^\ve_{k,x}}\right)\abs{\ve\nabla_x\kappa^\ve(x)}\dx.
\end{aligned} \end{equation}
It remains to estimate the measure of the $\{B^\ve_{k,x}\}_{k\in\N,x\in\TT^d}$.

Let $k\in\N$ and $x\in\TT^d$.  Chebyshev's inequality and Jensen's inequality prove that
\begin{equation}\label{equiv_36}\begin{aligned}
\abs{B^\ve_{k,x}} & \leq \frac{1}{k^2}\int_0^t\int_{\TT^d} \abs{\rho(y,t)-\rho^\ve(y+x,t)}^2\mathbf{1}_{k,k+1}(y,t)\dy\dt
\\ & \leq \frac{1}{k^2}\int_0^t\int_{(\TT^d)^2}\abs{\rho(y,t)-\rho(y',t)}^2\kappa^\ve(y+x-y')\mathbf{1}_{k,k+1}(y,t)\dyp\dy\dt.
\end{aligned}\end{equation}
Since the convexity proves that the derivative of the inverse of $\Phi^\frac{1}{2}$ is non-increasing, it follows from the nonnegativity of $\rho$ that, for every $y,y'\in\TT^d$ and $t\in[0,T]$,
\begin{equation}\label{equiv_37}\begin{aligned}
& \abs{\rho(y,t)-\rho(y',t)} =\abs{\left(\Phi^\frac{1}{2}\right)^{-1}\left(\Phi^\frac{1}{2}(\rho(y,t))\right)-\left(\Phi^\frac{1}{2}\right)^{-1}\left(\Phi^\frac{1}{2}(\rho(y',t))\right)}
\\ & \quad = \abs{\int_0^1\left(\left(\Phi^\frac{1}{2}\right)^{-1}\right)'\left(s\Phi^\frac{1}{2}(\rho(y,t))+(1-s)\Phi^\frac{1}{2}(\rho(y',t))\right)\ds} \cdot \abs{\Phi^\frac{1}{2}(\rho(y,t))-\Phi^\frac{1}{2}(\rho(y',t))}
\\ & \quad \leq \abs{\int_0^1\left(\left(\Phi^\frac{1}{2}\right)^{-1}\right)'\left(s\Phi^\frac{1}{2}(\rho(y,t))\right)\ds}\cdot \abs{\Phi^\frac{1}{2}(\rho(y,t))-\Phi^\frac{1}{2}(\rho(y',t))}.
\end{aligned}\end{equation}
It then follows from the fundamental theorem of calculus, the nonnegativity of $\rho$, the fact that $(\Phi^\frac{1}{2})^{-1}(0)=0$, and \eqref{equiv_37} that, for every $y,y'\in\TT^d$ and $t\in[0,T]$,
\begin{equation}\label{equiv_38}
\abs{\rho(y,t)-\rho(y',t)} \leq \abs{\frac{\rho(y,t)}{\Phi^\frac{1}{2}(\rho(y,t))}}\abs{\Phi^\frac{1}{2}(\rho(y,t))-\Phi^\frac{1}{2}(\rho(y',t))}.
\end{equation}
Returning to \eqref{equiv_36}, it follows from the definition of the $\{A_k\}_{k\in\N}$, the fact that $\Phi$ is increasing, and \eqref{equiv_38} that
\begin{equation}\label{equiv_39}\begin{aligned}
\abs{B^\ve_{k,x}} & \leq \frac{1}{k^2}\int_0^t\int_{(\TT^d)^2}\abs{\frac{\rho(y,t)}{\Phi^\frac{1}{2}(\rho(y,t))}}^2\abs{\Phi^\frac{1}{2}(\rho(y,t))-\Phi^\frac{1}{2}(\rho(y',t))}^2\kappa^\ve(y+x-y')\mathbf{1}_{k,k+1}(y,t)\dyp\dy\dt
\\ & \leq \frac{(M+k+1)^2}{k^2\Phi(M+k)}\int_0^t\int_{(\TT^d)^2}\abs{\Phi^\frac{1}{2}(\rho(y,t))-\Phi^\frac{1}{2}(\rho(y',t))}^2\kappa^\ve(y+x-y')\mathbf{1}_{k,k+1}(y,t)\dyp\dy\dt.
\end{aligned}\end{equation}
It follows from Jensen's inequality and the fundamental theorem of calculus that
\begin{equation}\label{equiv_40}\begin{aligned}
& \int_0^t\int_{(\TT^d)^2}\abs{\Phi^\frac{1}{2}(\rho(y,t))-\Phi^\frac{1}{2}(\rho(y',t))}^2\kappa^\ve(y+x-y')\mathbf{1}_{k,k+1}(y,t)\dyp\dy\dt
\\ & \leq \int_0^t\int_{(\TT^d)^2}\left(\int_0^1\abs{\nabla\Phi^\frac{1}{2}(\rho(sy+(1-s)y',t))}^2\ds\right)\abs{y-y'}^2\kappa^\ve(y+x-y')\mathbf{1}_{k,k+1}(y,t)\dyp\dy\dt.
\end{aligned}\end{equation}
Since on the support of the convolution kernel we have $\abs{y-y'}\leq \abs{x}+\ve$, it follows from \eqref{equiv_39} and  \eqref{equiv_40} that there exists $c\in(0,\infty)$ such that
\begin{equation}\label{equiv_42} \begin{aligned}
 \abs{B^\ve_{k,x}} & \leq \frac{c\left(\ve^2+\abs{x}^2\right)(M+k+1)^2}{k^2\Phi(M+k)}
\\ & \quad \cdot\int_0^t\int_{(\TT^d)^2}\left(\int_0^1\abs{\nabla\Phi^\frac{1}{2}(\rho(sy+(1-s)y',t))}^2\ds\right)\kappa^\ve(y+x-y')\mathbf{1}_{k,k+1}(y,t)\dyp\dy\dt.
\end{aligned}\end{equation}
Therefore, returning to \eqref{equiv_35}, it follows from the definition of $M\in(0,\infty)$, \eqref{skel_continuity_6}, and \eqref{equiv_42} that, for $c\in(0,\infty)$ depending on $M\in(0,\infty)$, for every $x\in\TT^d$,
\begin{equation}\label{equiv_43}\begin{aligned}
& \sum_{k=1}^\infty\Phi(M+k+1)\abs{B^\ve_{k,x}}
\\ & \leq \sum_{k=1}^\infty \left(\frac{c\left(\ve^2+\abs{x}^2\right)(M+k+1)^2\Phi(M+k+1)}{k^2\Phi(M+k)}\right.
\\ & \quad  \left.\cdot\int_0^t\int_{(\TT^d)^2}\left(\int_0^1\abs{\nabla\Phi^\frac{1}{2}(\rho(sy+(1-s)y',t))}^2\ds\right)\kappa^\ve(y+x-y')\mathbf{1}_{k,k+1}(y,t)\dyp\dy\dt\right)
\\ & \leq c\left(\ve^2+\abs{x}^2\right)\int_0^t\int_{(\TT^d)^2}\left(\int_0^1\abs{\nabla\Phi^\frac{1}{2}(\rho(sy+(1-s)y',t))}^2\ds\right)\kappa^\ve(y+x-y')\mathbf{1}_{A_1}(y,t)\dyp\dy\dt
\\ & \leq c\left(\ve^2+\abs{x}^2\right)\norm{\nabla\Phi^\frac{1}{2}(\rho)}_{L^2(\TT^d\times[0,T])}^2.
\end{aligned}\end{equation}
Returning to \eqref{equiv_35}, it follows from the definition of the convolution kernel and \eqref{equiv_43} that, for $c\in(0,\infty)$, for every $\ve\in(0,1)$,
\begin{equation}\label{equiv_44}
\int_0^t\int_{(\TT^d)^2}\Phi(\rho(y,t))\abs{\ve\nabla_x\kappa^\ve(y-x)}\abs{\psi(x,\rho^\ve(x,t))}\mathbf{1}_{A_1}(y,t)\dy\dx\dt\leq c\ve^2\norm{\nabla\Phi^\frac{1}{2}(\rho)}^2_{L^2(\TT^d\times[0,T])}.
\end{equation}
Then, returning to \eqref{equiv_32}, it follows from \eqref{equiv_44} that, for that $c\in(0,\infty)$, for every $\ve\in(0,1)$,
\begin{equation}\label{equiv_000080}\begin{aligned}
& \abs{\int_0^t\int_{(\TT^d)^2}F(y,t)\Phi^\frac{1}{2}(\rho(y,t))\cdot\nabla_x\kappa^\ve(y-x)\psi(x,\rho^\ve(x,t))\mathbf{1}_{A_1}(y,t)\dy\dx\dt}
\\ & \leq c\norm{\nabla\Phi^\frac{1}{2}(\rho)}_{L^2(\TT^d\times[0,T])}\left(\int_0^t\int_{(\TT^d)^2}\abs{F}^2(y,t)\abs{\ve\nabla_x\kappa^\ve(y-x)}\abs{\psi(x,\rho^\ve(x,t))}\mathbf{1}_{A_1}(y,t)\dy\dx\dt\right)^\frac{1}{2}.
\end{aligned}\end{equation}
A repetition of the argument leading to \eqref{equiv_00062} proves that
\[\limsup_{\ve\rightarrow 0}\abs{\int_0^t\int_{(\TT^d)^2}F(y,t)\Phi^\frac{1}{2}(\rho(y,t))\cdot\nabla_x\kappa^\ve(y-x)\psi(x,\rho^\ve(x,t))\mathbf{1}_{A_1}(y,t)\dy\dx\dt}=0,\]
which with \eqref{equiv_000080} completes the proof of \eqref{equiv_30}.  \end{proof}

\begin{lem}\label{aux_lemma_4}  Let $\Phi\in\C([0,\infty))\cap\C^1_{\textrm{loc}}((0,\infty))$ satisfy Assumption~\ref{as_unique}, \eqref{skel_continuity_6}, and \eqref{skel_continuity_7}.  Then, for every $F\in L^2(\TT^d\times[0,T];\R^d)$ and $\psi\in\C^\infty_c(\TT^d\times\R)$,
\[\begin{aligned}
& \lim_{\ve\rightarrow 0}\left(\int_0^t\int_{(\TT^d)^2}F(y,t)\Phi^\frac{1}{2}(\rho(y,t))\kappa^\ve(y-x)\left(\partial_\xi\psi\right)(x,\rho^\ve(x,t))\cdot\nabla_x\rho^\ve(x,t)\mathbf{1}_{A_0}(y,t)\dy\dx\dt\right)
\\ & = \int_0^t\int_{\TT^d}F(x,t)\frac{2\Phi(\rho(x,t))}{\Phi'(\rho(x,t))}\cdot\nabla\Phi^\frac{1}{2}(\rho(x,t))(\partial_\xi\psi)(x,\rho(x,t))\dx\dt.
\end{aligned}\]
\end{lem}

\begin{proof}  We will first define a cutoff function that will be used to treat the degeneracy of the diffusion.  Let $\phi\colon\R\rightarrow[0,1]$ be defined by
\[\phi (\xi)=\left\{1\;\;\textrm{if}\;\;\abs{\xi}\leq 1,\;\;\textrm{and}\;\;1-\abs{\xi}\;\; \textrm{if}\;\; 1 \leq \abs{\xi}\leq 2,\;\;\textrm{and}\;\;0\;\;\textrm{if}\;\;\abs{\xi}>2\right\},\]
and for every $\eta\in(0,1)$ let $\phi^\eta\colon\R\rightarrow\R$ be defined by $\phi^\eta(\xi)=\phi(\nicefrac{\xi}{\eta})$.  Let $\ve,\eta\in(0,1)$.  The definition of the kinetic function proves that
\begin{equation}\label{equiv_11} \begin{aligned} & \nabla \rho^\ve(x,t) = \int_{\TT^d}\rho(z,t)\nabla_x \kappa^\ve(z-x)\dz
\\& = \int_\R\int_{\TT^d}\chi(z,\xi,t)\nabla_x \kappa^\ve(z-x)\phi^\eta(\xi)\dz\dxi +\int_\R\int_{\TT^d}\chi(z,\xi,t)\nabla_x \kappa^\ve(z-x)(1-\phi^\eta(\xi))\dz\dxi.
\end{aligned} \end{equation}
It follows from Lemma~\ref{ibp} and the definition of $\phi^\eta$ that the final term of \eqref{equiv_11} satisfies
\begin{equation}\label{equiv_12}\begin{aligned} & \int_\R\int_{\TT^d}\chi(z,\xi,t)\nabla_x \kappa^\ve(z-x)(1-\phi^\eta(\xi))\dz\dxi
\\ & = -\int_\R\int_{\TT^d}\frac{\Phi'(\xi)}{2\Phi^\frac{1}{2}(\xi)}\chi(z,\xi,t)\nabla_z\left( \kappa^\ve(z-x)\frac{2\Phi^\frac{1}{2}(\xi)}{\Phi'(\xi)}(1-\phi^\eta(\xi))\right)\dz\dxi
\\ & = \int_{\TT^d}\nabla_z\Phi^\frac{1}{2}(\rho(z,t))\kappa^\ve(z-x)\frac{2\Phi^\frac{1}{2}(\rho(z,t))}{\Phi'(\rho(z,t))}(1-\phi^\eta(\rho(z,t)))\dz.
\end{aligned}\end{equation}
It follows from \eqref{equiv_11} and \eqref{equiv_12} that
\begin{equation}\label{equiv_012}\begin{aligned}
& \int_0^t\int_{(\TT^d)^2}F(y,t)\Phi^\frac{1}{2}(\rho(y,t))\kappa^\ve(y-x)\left(\partial_\xi\psi\right)(x,\rho^\ve(x,t))\cdot\nabla_x\rho^\ve(x,t)\mathbf{1}_{A_0}(y,t)
\\ & = \int_0^t\int_{(\TT^d)^3}F(y,t)\Phi^\frac{1}{2}(\rho(y,t))\kappa^\ve(y-x)\left(\partial_\xi\psi\right)(x,\rho^\ve(x,t))\mathbf{1}_{A_0}(y,t)
\\ & \quad \quad \quad \quad \quad \cdot \nabla_z\Phi^\frac{1}{2}(\rho(z,t))\kappa^\ve(z-x)\frac{2\Phi^\frac{1}{2}(\rho(z,t))}{\Phi'(\rho(z,t))}(1-\phi^\eta(\rho(z,t)))
\\ & + \int_0^t\int_{(\TT^d)^3}\int_\R F(y,t)\Phi^\frac{1}{2}(\rho(y,t))\kappa^\ve(y-x)\left(\partial_\xi\psi\right)(x,\rho^\ve(x,t))\mathbf{1}_{A_0}(y,t)\chi(z,\xi,t)\nabla_x \kappa^\ve(z-x)\phi^\eta(\xi)\Big).
\end{aligned}\end{equation}
The two terms of on the righthand side of \eqref{equiv_012} will be treated separately.  For the first term on the righthand side of \eqref{equiv_012}, since the definition of $\phi^\eta$ and \eqref{skel_continuity_7} prove that $\nicefrac{2\Phi^\frac{1}{2}(\rho)}{\Phi'(\rho)}(1-\phi^\eta(\rho))$ is bounded, since the definition of $A_0$ proves that $\Phi^\frac{1}{2}(\rho(y,t))$ is bounded on $A_0$, and since \eqref{skel_continuity_3} implies that $\lim_{\xi\rightarrow 0^+}\left(\nicefrac{\Phi(\xi)}{\Phi'(\xi)}\right)=0,$ it follows by the dominated convergence theorem, $\psi\in\C^\infty_c(\TT^d\times\R)$, the definition of $\{\phi^\eta\}_{\eta\in(0,1)}$, and \eqref{skel_continuity_3} that
\begin{equation}\label{equiv_084}\begin{aligned}
& \lim_{\eta\rightarrow 0}\left[\lim_{\ve\rightarrow 0}\left(\int_0^t\int_{(\TT^d)^3}F(y,t)\Phi^\frac{1}{2}(\rho(y,t))\kappa^\ve(y-x)\left(\partial_\xi\psi\right)(x,\rho^\ve(x,t))\mathbf{1}_{A_0}(y,t)\right.\right.
\\ & \quad \quad \quad \quad \quad \cdot \left.\left.\nabla_z\Phi^\frac{1}{2}(\rho(z,t))\kappa^\ve(z-x)\frac{2\Phi^\frac{1}{2}(\rho(z,t))}{\Phi'(\rho(z,t))}(1-\phi^\eta(\rho(z,t)))\dz\dy\dx\dt\right)\right]
\\ & = \lim_{\eta\rightarrow 0}\left[\int_0^t\int_{\TT^d}F(x,t)\frac{2\Phi(\rho(x,t))}{\Phi'(\rho(x,t))}\cdot \nabla\Phi^\frac{1}{2}(\rho(x,t))(\partial_\xi\psi)(x,\rho(x,t))(1-\phi^\eta(\rho(x,t)))\dx\dt\right]
\\ & = \int_0^t\int_{\TT^d}F(x,t)\frac{2\Phi(\rho(x,t))}{\Phi'(\rho(x,t))}\cdot\nabla\Phi^\frac{1}{2}(\rho(x,t))(\partial_\xi\psi)(x,\rho(x,t))\dx\dt.
\end{aligned}\end{equation}
It remains to treat the second term on the righthand side of \eqref{equiv_012}.

We have, for every $\ve,\eta\in(0,1)$,
\begin{equation}\label{equiv_15}\begin{aligned}
& \int_0^t\int_{(\TT^d)^3}\int_\R F(y,t)\Phi^\frac{1}{2}(\rho(y,t))\kappa^\ve(y-x)\mathbf{1}_{A_0}(y,t)
\\ & \quad \quad \quad \quad \quad \quad \cdot (\partial_\xi\psi)(x,\rho^\ve(x,t))\chi(z,\xi,t)\nabla_x \kappa^\ve(z-x)\phi^\eta(\xi)\dxi\dz\dy\dx\dt
\\ & = \int_0^t\int_{(\TT^d)^3}\int_\R\Big(F(y,t)\left(\Phi^\frac{1}{2}(\rho(y,t))-\Phi^\frac{1}{2}(\rho(x,t))\right)\kappa^\ve(y-x)\mathbf{1}_{A_0}(y,t)
\\ & \quad \quad \quad \quad \quad \cdot (\partial_\xi\psi)(x,\rho^\ve(x,t))\chi(z,\xi,t)\nabla_x \kappa^\ve(z-x)\phi^\eta(\xi)\Big)\dxi\dz\dy\dx\dt
\\ &  \quad + \int_0^t\int_{(\TT^d)^3}\int_\R \Big(F(y,t)\left(\Phi^\frac{1}{2}(\rho(x,t))-\Phi^\frac{1}{2}(\rho(z,t))\right)\kappa^\ve(y-x)\mathbf{1}_{A_0}(y,t)
\\ & \quad \quad \quad \quad \quad \cdot (\partial_\xi\psi)(x,\rho^\ve(x,t))\chi(z,\xi,t)\nabla_x \kappa^\ve(z-x)\phi^\eta(\xi)\Big)\dxi\dz\dy\dx\dt
\\ & \quad + \int_0^t\int_{(\TT^d)^3}\int_\R F(y,t)\Phi^\frac{1}{2}(\rho(z,t))\kappa^\ve(y-x)
\\ & \quad \quad \quad \quad \quad \quad \cdot (\partial_\xi\psi)(x,\rho^\ve(x,t))\chi(z,\xi,t)\nabla_x \kappa^\ve(z-x)\phi^\eta(\xi)\dxi\dz\dy\dx\dt.
\end{aligned}\end{equation}
The definition of the convolution kernel, the fundamental theorem of calculus, and $\psi\in\C^\infty_c(\TT^d\times\R)$ prove that there exists $c\in(0,\infty)$ such that the first term on the righthand side of \eqref{equiv_15} is equal to and bounded by
\begin{equation}\label{equiv_16}\begin{aligned}
& \Big| \int_0^t\int_{(\TT^d)^3}\int_\R \Big(F(y,t)\left(\int_0^1\nabla(\Phi^\frac{1}{2}(\rho))(x+s(y-x),t)\cdot(y-x)\ds\right)\kappa^\ve(y-x)\mathbf{1}_{A_0}(y,t)
\\ & \quad \quad \quad \quad \quad \cdot (\partial_\xi\psi)(x,\rho^\ve(x,t))\chi(z,\xi,t)\nabla_x \kappa^\ve(z-x)\phi^\eta(\xi)\Big)\dxi\dz\dy\dx\dt\Big|
\\ & \leq c   \int_0^t\int_{(\TT^d)^3}\int_\R\int_0^1 \Big(\abs{F(y,t)}\abs{\nabla(\Phi^\frac{1}{2}(\rho))(x+s(y-x),t)}\kappa^\ve(y-x)\mathbf{1}_{A_0}(y,t)
\\ & \quad \quad \quad \quad \quad \cdot  \chi(z,\xi,t)\abs{\ve \nabla_x \kappa^\ve(z-x)}\phi^\eta(\xi)\ds\dxi\dz\dy\dx\dt\Big).
\end{aligned}\end{equation}
The definition of $\{\phi^\eta\}_{\eta\in(0,1)}$ proves that, for every $\eta\in(0,1)$,
\begin{equation}\label{equiv_016} \int_\R\phi^\eta(\xi)\chi(z,\xi,t)\dxi\leq 2\eta.\end{equation}
It follows from \eqref{equiv_16}, \eqref{equiv_016}, H\"older's inequality, and the definition of the convolution kernels $\{\kappa^\ve\}_{\ve\in(0,1)}$ that there exists $c\in(0,\infty)$ independent of $\eta\in(0,1)$ such that
\begin{equation}\label{equiv_17}\begin{aligned}
& \limsup_{\ve\rightarrow 0}\left(\Big|\int_0^t\int_{(\TT^d)^3}\int_\R \Big(F(y,t)\left(\Phi^\frac{1}{2}(\rho(y,t))-\Phi^\frac{1}{2}(\rho(x,t))\right)\kappa^\ve(y-x)\mathbf{1}_{A_0}(y,t)\right.
\\ & \left.\quad \quad \quad \quad \quad \cdot (\partial_\xi\psi)(x,\rho^\ve(x,t))\chi(z,\xi,t)\nabla_x \kappa^\ve(z-x)\phi^\eta(\xi)\dxi\dz\dy\dx\dt\Big)\Big|\right)
\\ & \leq c\eta\limsup_{\ve\rightarrow 0}\left(\int_0^t\int_0^1\int_{(\TT^d)^2}\abs{F(y,t)}\abs{\nabla(\Phi^\frac{1}{2}(\rho))(x+s(y-x),t)}\kappa^\ve(y-x)\dx\dy\ds\dt\right)
\\ &\leq  c\eta \norm{F}_{L^2(\TT^d\times[0,T])}\norm{\nabla\Phi^\frac{1}{2}(\rho)}_{L^2(\TT^d\times[0,T])}.
\end{aligned}\end{equation}
The second term on the righthand side of \eqref{equiv_15} is treated identically.  It remains to treat the third term on the righthand side of \eqref{equiv_15}.

For every $\eta\in(0,1)$ let $\overline{\phi}^\eta\colon\R\rightarrow\R$ be defined by $\overline{\phi}^\eta(\xi)=\int_\R\phi^\eta(\xi)\dxi$.  The definitions of the kinetic function and of $\{\overline{\phi}^\eta\}_{\eta\in(0,1)}$ prove that the third term on the righthand side of \eqref{equiv_15} is equal to, omitting the integration variables,
\begin{equation}\label{equiv_18}
-\int_0^t\int_{(\TT^d)^3}F(y,t)\kappa^\ve(y-x)\mathbf{1}_{A_0}(y,t)(\partial_\xi\psi)(x,\rho^\ve(x,t))\Phi^\frac{1}{2}(\rho(z,t))\overline{\phi}^\eta(\rho(z,t))\nabla_z \kappa^\ve(z-x).
\end{equation}
Since $\overline{\phi}^\eta(\rho(z,t))=\overline{\phi}^\eta(((\Phi^\frac{1}{2})^{-1})(\Phi^\frac{1}{2}(\rho(z,t))))$,
the $L^2([0,T];H^1(\TT^d))$ regularity of $\Phi^\frac{1}{2}(\rho)$, \eqref{equiv_6}, the definition of $\{\phi^\eta\}_{\eta\in(0,1)}$, the definition of $\{\overline{\phi}^\eta\}_{\eta\in(0,1)}$, and \eqref{equiv_18} prove that, after integrating by parts in $z\in\TT^d$, omitting the integration variables, \eqref{equiv_18} is equal to
\begin{equation}\label{equiv_19}\begin{aligned}
&\int_0^t\int_{(\TT^d)^3}F(y,t)\kappa^\ve(y-x)\mathbf{1}_{A_0}(y,t)(\partial_\xi\psi)(x,\rho^\ve(x,t))\cdot \nabla_z\Phi^\frac{1}{2}(\rho(z,t))\overline{\phi}^\eta(\rho(z,t))\kappa^\ve(z-x)
\\ & + \int_0^t\int_{(\TT^d)^3}F(y,t)\kappa^\ve(y-x)\mathbf{1}_{A_0}(y,t)(\partial_\xi\psi)(x,\rho^\ve(x,t))\frac{2\Phi(\rho(z,t))}{\Phi'(\rho(z,t)}\nabla\Phi^\frac{1}{2}(\rho(z,t))\phi^\eta(\rho(z,t))\kappa^\ve(z-x).
\end{aligned}\end{equation}
The definition of $\{\phi^\eta\}_{\eta\in(0,1)}$ proves that, for every $\eta\in(0,1)$, for every $\xi\in\R$,
\begin{equation}\label{equiv_020}\overline{\phi}^\eta(\xi)\leq 2\eta,\end{equation}
and since \eqref{skel_continuity_3} and the definition of $\{\phi^\eta\}_{\eta\in(0,1)}$ prove that there exists $c\in(0,\infty)$ independent of $\eta\in(0,1)$ such that, for every $(x,t)\in\TT^d\times[0,T]$,
\begin{equation}\label{equiv_0020}\frac{2\Phi(\rho(x,t))}{\Phi'(\rho(x,t))}\phi^\eta(\rho(x,t))\leq c\eta,\end{equation}
it follows from \eqref{equiv_19}, \eqref{equiv_020}, \eqref{equiv_0020}, $\psi\in\C^\infty_c(\TT^d\times\R)$, and H\"older's inequality that there exists $c\in(0,\infty)$ independent of $\eta\in(0,1)$ such that, for every $\ve,\eta\in(0,1)$, omitting the integration variables, the first term of \eqref{equiv_19} is bounded by
\begin{equation}\label{equiv_00020}\begin{aligned}
&c\eta\int_0^t\int_{(\TT^d)^3}\abs{F(y,t)}\kappa^\ve(y-x)\mathbf{1}_{A_0}(y,t)\abs{\nabla\Phi^\frac{1}{2}(\rho(z,t))}\kappa^\ve(x-z)
\\ & \leq c\eta\norm{F}_{L^2(\TT^d\times[0,T])}\norm{\nabla\Phi^\frac{1}{2}(\rho)}_{L^2(\TT^d\times[0,T])}.
\end{aligned}\end{equation}
In combination \eqref{equiv_15}, \eqref{equiv_17}, and \eqref{equiv_00020} prove that
\begin{equation}\label{equiv_000020}
\begin{aligned}
 & \limsup_{\eta\rightarrow 0}\Big(\limsup_{\ve\rightarrow 0}\Big(\Big|\int_0^t\int_{(\TT^d)^3}\int_\R F(y,t)\Phi^\frac{1}{2}(\rho(y,t))\kappa^\ve(y-x)\mathbf{1}_{A_0}(y,t)
\\ & \quad \quad \quad \quad \quad \quad \cdot (\partial_\xi\psi)(x,\rho^\ve(x,t))\chi(z,\xi,t)\nabla_x \kappa^\ve(z-x)\phi^\eta(\xi)\dxi\dz\dy\dx\dt\Big|\Big)\Big)=0.
\end{aligned}\end{equation}
In combination \eqref{equiv_012}, \eqref{equiv_084}, and \eqref{equiv_000020} complete the proof.  \end{proof}

\section{Existence of weak solutions}\label{sec_existence}

In this section, we will prove that there exists a weak solution to the skeleton equation
\begin{equation}\label{skeleton_equation}\partial_t\rho = \Delta \Phi(\rho)-\nabla\cdot(\Phi^\frac{1}{2}(\rho)g),\end{equation}
for controls $g\in L^2(\TT^d\times[0,T];\R^d)$ and initial data $\rho_0\in\Ent_\Phi(\TT^d)$ in the sense of Definition~\ref{classical_weak}.  The existence of renormalized kinetic solutions (see Definition~\ref{skel_sol_def}) is then a consequence of Theorem~\ref{equiv}.  However, we emphasize that these methods can also be used to prove the existence of renormalized kinetic solutions in the sense of Definition~\ref{skel_sol_def}, but we omit these details since this more general result will not be used here.  The section is broken into three subsections.  In Section~\ref{subsection_preliminaries} we introduce the assumptions.  We prove the existence of a solution in Section~\ref{subsection_existence}, and we prove the weak-strong continuity property of the equation in Section~\ref{subsection_weak_strong}.

\subsection{Preliminaries}\label{subsection_preliminaries}  In this section, we introduce in Assumption~\ref{as_compact} conditions on $\Phi$ that will be used to guarantee the compactness of approximate solutions.  These conditions are satisfied by every fast diffusion and porous media nonlinearities.  In particular, for condition (ii), case \eqref{5_00} is satisfied by fast diffusion and case \eqref{5_0} by porous media nonlinearities.  Lemmas~\ref{as_interpolate}, \ref{as_pre}, and \ref{as_phi_reg} collect some important consequences of Assumption~\ref{as_compact}.

\begin{assumption}\label{as_compact}  Let $\Phi\in\C([0,\infty))\cap\C^1_{\textrm{loc}}((0,\infty)))$ satisfy $\Phi(0)=0$ with $\Phi'(\xi)>0$ for every $\xi\in(0,\infty)$.  Assume that $\Phi$ satisfies the following four properties.

\begin{enumerate}[(i)]
\item There exists $m\in[1,\infty)$ and $c\in(0,\infty)$ such that
\[\Phi(\xi)\leq c(1+\xi^m)\;\;\textrm{for every}\;\;\xi\in[0,\infty).\]
\item Either there exists $c\in(0,\infty)$ and $\theta\in[0,\nicefrac{1}{2}]$ such that
\begin{equation}\label{5_00} \frac{\Phi'(\xi)}{\Phi^\frac{1}{2}(\xi)}\leq c\xi^{-\theta}\;\;\textrm{for every}\;\;\xi\in(0,\infty),\end{equation}
or there exists $c\in(0,\infty)$ and $q\in[1,\infty)$ such that
\begin{equation}\label{5_0}\abs{\xi-\xi'}^q\leq c\abs{\Phi^\frac{1}{2}(\xi)-\Phi^\frac{1}{2}(\xi')}^2\;\;\textrm{for every}\;\;\xi,\xi'\in[0,\infty).\end{equation}
\item There exists $c\in(0,\infty)$ such that
\[\Phi'(\xi)\leq c(1+\xi+\Phi(\xi))\;\;\textrm{for every}\;\;\xi\in(0,\infty).\]
\item We have that $\log(\Phi)$ is locally integrable on $[0,\infty)$.
\end{enumerate}
\end{assumption}

\begin{lem}\label{as_interpolate}  Let $\Phi\in\C([0,\infty))\cap\C^1_{\textrm{loc}}((0,\infty))$ satisfy Assumption~\ref{as_compact}.  If a nonnegative $\rho\in L^\infty([0,T];L^1(\TT^d))$ satisfies $\nabla\Phi^\frac{1}{2}(\rho)\in L^2([0,T];L^2(\TT^d))$, then for every $\ve\in(0,1)$ there exists $c\in(0,\infty)$ such that
\[\norm{\Phi^\frac{1}{2}(\rho)}^2_{L^2([0,T];H^1(\TT^d))}\leq c\left(1+\norm{\rho}_{L^\infty([0,T];L^1(\TT^d))}^m\right)+\ve\norm{\nabla\Phi^\frac{1}{2}(\rho)}^2_{L^2([0,T];L^2(\TT^d))}.\]
\end{lem}

\begin{proof}  The proof is a consequence of the Sobolev embedding theorem and an interpolation estimate.  The details can be found in \cite[Lemma~5.4]{FehGes21}.  \end{proof}

\begin{lem}\label{as_pre}  Let $\Phi\in\C([0,\infty))\cap\C^1_{\textrm{loc}}((0,\infty))$ satisfy Assumption~\ref{as_compact} and let $\{\rho_n\}_{n\in\N}$ be a sequence that satisfies, for some $c\in(0,\infty)$ and $s\geq \nicefrac{d}{2}+1$ independent of $n$,
\[\norm{\rho_n}_{L^\infty([0,T];L^1(\TT^d))}+\norm{\Phi^\frac{1}{2}(\rho_n)}_{L^2([0,T];H^1(\TT^d))}+\norm{\partial_t \rho_n}_{L^1([0,T];H^{-s}(\TT^d))}\leq c.\]
Then,
\[\{\rho_n\}_{n\in\mathbb{N}}\;\;\textrm{is relatively pre-compact in}\;\;L^1([0,T];L^1(\TT^d)),\]
and
\[\{\Phi^\frac{1}{2}(\rho_n)\}_{n\in\N}\;\;\textrm{is relatively pre-compact in}\;\;L^2([0,T];L^2(\TT^2)).\]
\end{lem}

\begin{proof}  The proof is a consequence of the Aubin-Lions-Simon Lemma \cite{Aubin,pLions,Simon}, the compact embedding of the fractional Sobolev space $W^{\beta,1}(\TT^d)$ into $L^1(\TT^d)$, the continuous embedding of $L^1(\TT^d)$ into $H^{-s}(\TT^d)$ for every $s\geq \nicefrac{d}{2}+1$, and \cite[Lemma~5.11]{FehGes21}.  \end{proof}

\begin{lem}\label{as_phi_reg}  Let $\Phi\in\C([0,\infty))\cap\C^1_{\textrm{loc}}((0,\infty))$ be strictly increasing with $\Phi(0)=0$.  Then, there exist nondecreasing functions
\[\{\Phi^{\frac{1}{2},\eta}\colon[0,\infty)\rightarrow[0,\infty)\}_{\eta\in(0,1)}\subseteq\C^\infty([0,\infty)),\]
which satisfy the following four properties.
\begin{enumerate}[(i)]
\item For every $\eta\in(0,1)$,
\begin{equation}\label{fd_00}\Phi^{\frac{1}{2},\eta}(0)=0.\end{equation}
\item There exists $c\in(0,\infty)$ such that, for every $\eta\in(0,1)$ and $\xi\in[0,\infty)$,
\begin{equation}\label{fd_0} 0\leq \Phi^{\frac{1}{2},\eta}(\xi)\leq c\Phi^\frac{1}{2}(\xi)\;\;\textrm{and}\;\;0\leq \left(\Phi^{\frac{1}{2},\eta}\right)'(\xi)\leq c\left(\Phi^\frac{1}{2}\right)'(\xi).\end{equation}
\item For every $\eta\in(0,1)$ there exists $c\in(0,\infty)$ depending on $\eta$ such that
\begin{equation}\label{fd_1}\norm{\Phi^{\frac{1}{2},\eta}}_{L^\infty([0,\infty))}+\norm{(\Phi^{\frac{1}{2},\eta})'}_{L^\infty([0,\infty))}\leq c.\end{equation}
\item  For every compact set $A\subseteq[0,\infty)$,
\begin{equation}\label{fd_2}\lim_{\eta\rightarrow 0}\norm{\Phi^{\frac{1}{2},\eta}-\Phi^\frac{1}{2}}_{L^\infty(A)}=0.\end{equation}
\end{enumerate}
\end{lem}

\begin{proof}  Extend $\Phi\in\C([0,\infty))\cap\C^1_{\textrm{loc}}((0,\infty))$ to $\R$ by defining $\Phi(\xi)=0$ for every $\xi\in(-\infty,0)$.  For every $\ve\in(0,1)$ let $\kappa^\ve\in\C^\infty_c(\R)$ be a standard one-dimensional convolution kernel of scale $\ve\in(0,1)$ that is supported on $(-\ve,0)$.  For every $\eta,\ve\in(0,1)$ define $\Phi^{\frac{1}{2},\eta,\ve}=(((\Phi^\frac{1}{2}\wedge \Phi(\nicefrac{1}{\eta}))*\kappa^{\ve}).$
It follows by definition that, for every $\eta\in(0,1)$ there exists $\ve_\eta\in(0,1)$ such that $\Phi^{\frac{1}{2},\eta}=\Phi^{\frac{1}{2},\eta,\ve_\eta}$ satisfies \eqref{fd_00}, \eqref{fd_0}, \eqref{fd_1}, and \eqref{fd_2}.  This completes the proof.  \end{proof}

\subsection{Existence of solutions}\label{subsection_existence}  In this section, we will construct a solution to \eqref{skeleton_equation}.  In Proposition~\ref{smooth_exist}, we will first prove that there exists a weak solution to the regularized equation, for $\eta_1,\eta_2,\eta_3\in(0,1)$,
\[\partial_t\rho=\Delta\Phi^{\eta_1}(\rho)+\eta_2\Delta\rho-\nabla\cdot (\Phi^{\frac{1}{2},\eta_3}(\rho)g),\]
where $\Phi^{\eta_1}=(\Phi^{\frac{1}{2},\eta_1})^2$, and then pass to the limit $\eta_1\rightarrow 0$.  In Proposition~\ref{gen_exist} below, based on the estimates of Proposition~\ref{frac_est} below, we pass to the limits $\eta_2,\eta_3\rightarrow 0$ to complete the proof.

\begin{prop}\label{smooth_exist} Let $T\in(0,\infty)$, let $\Phi\in \C([0,\infty))\cap\C^1_{\textrm{loc}}((0,\infty))$ satisfy Assumption~\ref{as_compact}, and let $\{\Phi^{\frac{1}{2},\eta}\}_{\eta\in(0,1)}$ be defined in Lemma~\ref{as_phi_reg}.  Let $\overline{\Phi}\colon[0,\infty)\rightarrow[0,\infty)$ be defined by $\overline{\Phi}(0)=0$ and $\overline{\Phi}'(\xi)=\Phi(\xi)$.  Then, for every nonnegative $\rho_0\in L^\infty(\TT^d)$, $g\in L^2(\TT^d\times[0,T];\R^d)$, and $\eta_2,\eta_3\in(0,1)$ there exists a nonnegative weak solution $\rho\in L^2([0,T];H^1(\TT^d))$ of the equation
\begin{equation}\label{se_1} \partial_t\rho = \Delta\Phi(\rho)+\eta_2\Delta\rho-\nabla\cdot (\Phi^{\frac{1}{2},\eta_3}(\rho)g)\;\; \textrm{in}\;\;\TT^d\times(0,T)\;\;\textrm{with}\;\;\rho(\cdot,0)=\rho_0,\end{equation}
which satisfies that, for almost every $t\in[0,T]$,
\begin{equation}\label{se_00002}\norm{\rho(\cdot,t)}_{L^1(\TT^d)}=\norm{\rho_0}_{L^1(\TT^d)},\end{equation}
that, for some $c\in(0,\infty)$ independent of $\eta_2,\eta_3\in(0,1)$,
\begin{equation}\label{se_2}\begin{aligned}
& \norm{\rho}^2_{L^\infty([0,T];L^2(\TT^d))}+\eta_2 \norm{\nabla\rho}^2_{L^2([0,T];L^2(\TT^d;\R^d))}
\\ & \leq c\left(\norm{\rho_0}^2_{L^2(\TT^d)}+\frac{1}{\eta_2}\norm{\Phi^{\frac{1}{2},\eta_3}}_{L^\infty([0,\infty))}^2\norm{g}^2_{L^2(\TT^d\times[0,T];\R^d)}\right),
\end{aligned}\end{equation}
and that, for some $c\in(0,\infty)$ independent of $\eta_2,\eta_3\in(0,1)$,
\begin{equation}\label{se_0002}\begin{aligned}
& \norm{\overline{\Phi}(\rho)}_{L^\infty([0,T];L^1(\TT^d))}+\norm{\nabla\Phi(\rho)}^2_{L^2([0,T];L^2(\TT^d;\R^d))}
\\ & \leq c\left(\norm{\overline{\Phi}(\rho_0)}_{L^1(\TT^d)}+\norm{\Phi^{\frac{1}{2},\eta_3}}^2_{L^\infty([0,\infty))}\norm{g}^2_{L^2(\TT^d\times[0,T];\R^d)}\right).
\end{aligned}\end{equation}
\end{prop}

\begin{proof}  Let $\eta_2,\eta_3\in(0,1)$.  We first consider the equation, for every $\eta_1\in(0,1)$,
\begin{equation}\label{se_000} \partial_t\rho^{\eta_1} = \Delta\Phi^{\eta_1}(\rho^{\eta_1})+\eta_2\Delta\rho^{\eta_1}-\nabla\cdot (\Phi^{\frac{1}{2},\eta_3}(\rho^{\eta_1})g)\;\; \textrm{in}\;\;\TT^d\times(0,T)\;\;\textrm{with}\;\;\rho(\cdot,0)=\rho_0,\end{equation}
for $\Phi^{\eta_1}=(\Phi^{\frac{1}{2},\eta_1})^2$.  For fixed $\eta_1,\eta_2,\eta_3\in(0,1)$ let
\[S\colon L^2([0,T];L^2(\TT^d))\rightarrow L^2([0,T];H^1(\TT^d))\subseteq L^2([0,T];L^2(\TT^d)),\]
be defined for every $v\in L^2([0,T];L^2(\TT^d))$ as the unique weak solution of the equation
\[\partial_t S(v) =\nabla\cdot\left((\Phi^{\eta_1})'(v)\nabla S(v)+\eta_2\nabla S(v)\right)-\nabla\cdot\left(\Phi^{\frac{1}{2},\eta_3}(v)g\right)\;\; \textrm{in}\;\;\TT^d\times(0,T),\]
with $S(v)(\cdot,0) =\rho_0$.  Estimates of the type \eqref{se_3} below prove that there exists $c\in(0,\infty)$ such that
\begin{equation}\label{se_00003}\begin{aligned}
& \norm{S(v)}^2_{L^\infty([0,T];L^2(\TT^d))}+\eta_2\norm{\nabla S(v)}^2_{L^2([0,T];L^2(\TT^d;\R^d))}
\\ & \leq c\left(\norm{\rho_0}^2_{L^2(\TT^d)}+\frac{1}{\eta_2}\norm{\Phi^{\frac{1}{2},\eta_3}}^2_{L^\infty([0,\infty))}\norm{g}^2_{L^2(\TT^d\times[0,T];\R^d)}\right),
\end{aligned} \end{equation}
and estimates of the type \eqref{se_5} prove that
\begin{equation}\label{se_000003}\begin{aligned}
& \norm{\partial_t S(v)}_{L^2([0,T]:H^{-1}(\TT^d))}
\\ & \leq \left(\eta_2^\frac{1}{2}+\norm{(\Phi^{\eta_1})'}_{L^\infty([0,\infty))}\right)\norm{\eta_2^\frac{1}{2}\nabla S(v)}_{L^2([0,T];L^2(\TT^d;\R^d))}+\norm{\Phi^{\frac{1}{2},\eta_3}}_{L^\infty([0,\infty))}\norm{g}_{L^2([0,T];L^2(\TT^d;\R^d))}.
\end{aligned}\end{equation}
The Aubin-Lions-Simons lemma \cite{Aubin,pLions,Simon}, \eqref{se_00003}, and \eqref{se_000003} prove that the image of $S$ lies in a compact subset of $L^2([0,T];L^2(\TT^d))$.  The Schauder fixed point theorem therefore implies that $S$ has a fixed point, which proves the existence of a weak solution to \eqref{se_000}.

Let $\rho^{\eta_1}\in L^2([0,T];H^1(\TT^d))$ denote a weak solution of \eqref{se_000}.  Estimate \eqref{se_00002} follows from the nonnegativity of $\rho^{\eta_1}$ by testing \eqref{se_1} with the constant function $1$.  By testing \eqref{se_1} with $\rho^{\eta_1}$, it follows from H\"older's inequality, Young's inequality, and the fact that $\Phi^{\frac{1}{2},\eta_1}$ is non-decreasing that, for some $c\in(0,\infty)$ independent of $\eta_1,\eta_2,\eta_3\in(0,1)$,
\begin{equation}\label{se_3}\begin{aligned}
& \norm{\rho^{\eta_1}}^2_{L^\infty([0,T];L^2(\TT^d))}+\eta_2\norm{\nabla \rho^{\eta_1}}^2_{L^2([0,T];L^2(\TT^d;\R^d))}
\\ & \leq c\left(\norm{\rho_0}^2_{L^2(\TT^d)}+\frac{1}{\eta_2}\norm{\Phi^{\frac{1}{2},\eta_3}}^2_{L^\infty([0,\infty))}\norm{g}^2_{L^2(\TT^d\times[0,T];\R^d)}\right).
\end{aligned} \end{equation}
Let $\overline{\Phi}^{\eta_1}\colon[0,\infty)\rightarrow[0,\infty)$ denote the anti-derivative
\[\overline{\Phi}^{\eta_1}(\xi)=\int_0^\xi\Phi^{\eta_1}(\xi')\dxp.\]
After testing \eqref{se_1} with the composition $\Phi^{\eta_1}(\rho^{\eta_1})$, which is justified by $\rho^{\eta_1}\in L^2([0,T];H^1(\TT^d))$ and \eqref{fd_1}, it follows from H\"older's inequality, Young's inequality, the fact that $\Phi^{\frac{1}{2},\eta_1}$ is non-decreasing, and \eqref{fd_0} that, for some $c\in(0,\infty)$ independent of $\eta_1,\eta_2,\eta_3\in(0,1)$,
\begin{equation}\begin{aligned}\label{se_4}
& \norm{\overline{\Phi}^{\eta_1}(\rho^{\eta_1})}_{L^\infty([0,T];L^1(\TT^d))}+\norm{\nabla\Phi^{\eta_1}(\rho^{\eta_1})}^2_{L^2([0,T];L^2(\TT^d;\R^d))}
\\ & \leq c\left(\norm{\overline{\Phi}^{\eta_1}(\rho_0)}_{L^1(\TT^d)}+\norm{\Phi^{\frac{1}{2},\eta_3}}^2_{L^\infty([0,\infty))}\norm{g}^2_{L^2(\TT^d\times[0,T];\R^d)}\right)
\\  & \leq c\left(\norm{\overline{\Phi}(\rho_0)}_{L^1(\TT^d)}+\norm{\Phi^{\frac{1}{2},\eta_3}}^2_{L^\infty([0,\infty))}\norm{g}^2_{L^2(\TT^d\times[0,T];\R^d)}\right).
\end{aligned}\end{equation}
It follows from the equation that, for some $c\in(0,\infty)$ independent of $\eta_1,\eta_2,\eta_3\in(0,1)$,
\begin{equation}\label{se_5}\begin{aligned}
& \norm{\partial_t\rho^{\eta_1}}_{L^2([0,T];H^{-1}(\TT^d))}
\\ & \leq \norm{\nabla\Phi^{\eta_1}(\rho^{\eta_1})}_{L^2([0,T];L^2(\TT^d))}+\eta_2\norm{\nabla\rho^{\eta_1}}_{L^2([0,T];L^2(\TT^d))}+\norm{\Phi^{\frac{1}{2},\eta_3}}_{L^\infty([0,\infty))}\norm{g}_{L^2([0,T];L^2(\TT^d))}.
\end{aligned}\end{equation}
It follows from \eqref{fd_2}, \eqref{se_3}, \eqref{se_4}, \eqref{se_5}, and the Aubin-Lions-Simon lemma \cite{Aubin,pLions,Simon} that there exists $\rho\in L^2([0,T];H^1(\TT^d))$ such that, after passing to a subsequence $\{\eta_1^k\rightarrow 0\}_{k\in\N}$, as $k\rightarrow\infty$,
\begin{equation}\label{se_6}\rho^{\eta_1^k}\rightarrow \rho\;\;\textrm{strongly in}\;\;L^2([0,T];L^2(\TT^d)),\end{equation}
that
\begin{equation}\label{se_7}\rho^{\eta_1^k}\rightharpoonup \rho\;\;\textrm{weakly in}\;\;L^2([0,T];H^1(\TT^d)),\end{equation}
and that
\begin{equation}\label{se_8}\nabla\Phi^{\eta^k_1}(\rho^{\eta^k_1})\rightharpoonup\nabla\Phi(\rho)\;\;\textrm{weakly in}\;\;L^2([0,T];L^2(\TT^d;\R^d)).\end{equation}
In combination \eqref{se_6}, \eqref{se_7}, and \eqref{se_8} prove that $\rho\in L^2([0,T];H^1(\TT^d))$ is a nonnegative weak solution of \eqref{se_1}.  Estimate \eqref{se_2} follows from \eqref{se_3} and the weak lower-semicontinuity of the Sobolev norm.  Estimate \eqref{se_0002} follows from \eqref{fd_2}, \eqref{se_4}, \eqref{se_8}, and the weak lower-semicontinuity of the Sobolev norm.  This completes the proof.  \end{proof}

It remains to pass to the limit $\eta_2,\eta_3\rightarrow 0$ in \eqref{se_1}.  For this we will use the following estimate, which relies crucially on the nonnegativity of the initial data.  The estimate is based on testing the equation with $\log(\Phi(\rho))$, as already explained in Section~\ref{sec_a_priori}.  We recall, for every $\Phi\in\C([0,\infty))\cap\C^1_{\textrm{loc}}((0,\infty))$,
\begin{equation}\label{psi_phi}\Psi_\Phi(\xi)=\int_0^\xi\log(\Phi(\xi'))\dxip,\end{equation}
and we recall the space
\[\Ent_\Phi(\TT^d)=\left\{\rho_0\in L^1(\TT^d)\colon \rho_0\geq 0\;\;\textrm{a.e.}\;\;\textrm{and}\;\;\int_{\TT^d}\Psi_\Phi(\rho_0(x))\dx<\infty\right\}.\]
The general existence result follows immediately after the next proposition.

\begin{prop}\label{frac_est}  Let $T\in(0,\infty)$, let $\Phi\in\C([0,\infty))\cap\C^1_{\textrm{loc}}((0,\infty))$ satisfy Assumption~\ref{as_compact}, and let $\{\Phi^{\frac{1}{2},\eta}\}_{\eta\in(0,1)}$ be as in Lemma~\ref{as_phi_reg}.  Then, there exists $c\in(0,\infty)$ such that, for every nonnegative $\rho_0\in L^\infty(\TT^d)$ and $g\in L^2(\TT^d\times[0,T];\R^d)$, for every $\eta_2,\eta_3\in(0,1)$, the solution $\rho\in L^2([0,T];H^1(\TT^d))$ from Proposition~\ref{smooth_exist} satisfies, for almost every $t\in[0,T]$,
\begin{equation}\label{fe_1}
\int_{\TT^d}\Psi_\Phi(\rho(x,t))\dx+\int_0^t\int_{\TT^d}\abs{\nabla\Phi^\frac{1}{2}(\rho(x,s))}^2\dx\ds \leq c\left(\int_{\TT^d}\Psi_\Phi(\rho_0(x))\dx+\norm{g}^2_{L^2(\TT^d\times[0,T];\R^d)}\right).
\end{equation}
\end{prop}

\begin{proof}   For every $M\in \N$, let $\log^M\in\C(\R)$ be defined by
\[\log^M(\xi)=\left\{\begin{aligned}
& (-M\vee \log(\xi))\wedge M && \textrm{if}\;\;\xi>0,
\\  & -M && \textrm{if}\;\;\xi<0,
\end{aligned}\right. \]
and for every $M\in\N$ and $\ve\in(0,1)$, for a standard convolution kernel $\kappa^\ve\in \C^\infty(\R)$ of scale $\ve\in(0,1)$, let $\log^{M,\ve}\in\C^\infty(\R)$ be defined by
\[\log^{M,\ve}(\xi)=(\log^M*\kappa^\ve)(\xi).\]
For every $M\in\N$ and $\ve\in(0,1)$, let
\[\Psi^M_\Phi(\xi)=\int_0^\xi\log^M(\Phi(\xi'))\dxip\;\;\textrm{and}\;\;\Psi^{M,\ve}_\Phi(\xi)=\int_0^\xi\log^{M,\ve}(\Phi(\xi'))\dxip.\]
Estimate \eqref{se_0002} implies for every $M\in\N$ and $\ve\in(0,1)$ that $\log^{M,\ve}(\Phi(\rho))$ is an admissible test function for \eqref{se_1}.  It follows from H\"older's inequality, Young's inequality, and the fact that $\log^{M,\ve}$ is nondecreasing that, for every $\d\in(0,1)$, for almost every $t\in[0,T]$,
\[\begin{aligned}
& \left.\int_{\TT^d}\Psi^{M,\ve}_\Phi(\rho(x,s))\dx\right|_{s=0}^{s=t}+\int_0^t\int_{\TT^d}\left(\log^{M,\ve}\right)'(\Phi(\rho))\abs{\nabla\Phi(\rho)}^2\dx\ds
\\ & \leq \frac{1}{2\d}\norm{g}^2_{L^2(\TT^d\times[0,T];\R^d)}+\frac{\d}{2}\int_0^t\int_{\TT^d}\abs{\left(\log^{M,\ve}\right)'(\Phi(\rho))}^2\Phi^{\eta_3}(\rho)\abs{\nabla\Phi(\rho)}^2\dx\dt.
\end{aligned}\]
The dominated convergence theorem and the monotone convergence theorem prove that, after passing to the limits $\ve\rightarrow 0$ and $M\rightarrow\infty$, for every $\d\in(0,1)$, for almost every $t\in[0,T]$,
\begin{equation}\label{fe_2}\begin{aligned}
& \left.\int_{\TT^d}\Psi_\Phi(\rho(x,s))\dx\right|_{s=0}^{s=t}+\int_0^t\int_{\TT^d}\frac{1}{\Phi(\rho)}\abs{\nabla\Phi(\rho)}^2\dx\ds
\\ & \leq \frac{1}{2\d}\norm{g}^2_{L^2(\TT^d\times[0,T];\R^d)}+\frac{\d}{2}\int_0^t\int_{\TT^d}\frac{\Phi^{\eta_3}(\rho)}{\Phi(\rho)^2}\abs{\nabla\Phi(\rho)}^2\dx\ds.
\end{aligned} \end{equation}
We now use \eqref{fd_0} to fix $c_1\in(0,\infty)$ independent of $\eta_3\in(0,1)$ such that $0\leq\Phi^{\eta_3}\leq c_1\Phi$, and choose $\delta_1=\frac{1}{c_1}$ to conclude from \eqref{fe_2} that, for almost every $t\in[0,T]$,
\begin{equation}\label{fe_3}
\int_{\TT^d}\Psi_\Phi(\rho(x,t))\dx+\frac{1}{2}\int_0^t\int_{\TT^d}\frac{1}{\Phi(\rho)}\abs{\nabla\Phi(\rho)}^2\dx\ds \leq \int_{\TT^d}\Psi_\Phi(\rho_0(x))\dx+\frac{c_1}{2}\norm{g}^2_{L^2(\TT^d\times[0,T];\R^d)}.
\end{equation}
Estimate \eqref{fe_1} will follow from \eqref{fe_3} after we prove the distributional inequality
\begin{equation}\label{fe_4}\nabla\Phi^\frac{1}{2}(\rho)=\frac{1}{2\Phi^\frac{1}{2}(\rho)}\nabla\Phi(\rho).\end{equation}
To prove \eqref{fe_4}, let $\psi\in\C^\infty(\TT^d\times[0,T])$ and observe that the estimate \eqref{se_0002} implies the equality, for every $\delta\in(0,1)$,
\[\int_0^T\int_{\TT^d}\left(\Phi(\rho)+\delta\right)^\frac{1}{2}\nabla\psi\dx\dt=-\frac{1}{2}\int_0^T\int_{\TT^d}\left(\Phi(\rho)+\delta\right)^{-\frac{1}{2}}\nabla\Phi(\rho)\psi\dx\dt.\]
Estimate \eqref{fe_3} and the dominated convergence theorem prove that, after taking the limit $\delta\rightarrow 0$,
\[\int_0^T\int_{\TT^d}\Phi^\frac{1}{2}(\rho)\nabla\psi\dx\dt=-\int_0^T\int_{\TT^d}\frac{1}{2\Phi^\frac{1}{2}(\rho)}\nabla\Phi(\rho)\psi\dx\dt,\]
which completes the proof of \eqref{fe_4}, and therefore the proof. \end{proof}

\begin{prop}\label{gen_exist}  Let $T\in(0,\infty)$ and let $\Phi\in\C([0,\infty))\cap\C^1_{\textrm{loc}}((0,\infty))$ satisfy Assumptions~\ref{as_unique} and \ref{as_compact}.  Then, for every $\rho_0\in \Ent_\Phi(\TT^d)$ and $g\in L^2(\TT^d\times[0,T];\R^d)$ there exists a nonnegative weak solution $\rho\in L^\infty([0,T];L^1(\TT^d))$ to the equation
\begin{equation}\label{ge_1} \partial_t\rho = \Delta\Phi(\rho)-\nabla\cdot(\Phi^\frac{1}{2}(\rho)g)\;\; \textrm{in}\;\;\TT^d\times(0,T)\;\;\textrm{with}\;\;\rho(\cdot,0) = \rho_0,\end{equation}
which satisfies, for almost every $t\in[0,T]$,
\begin{equation}\label{ge_02}\norm{\rho(\cdot,t)}_{L^1(\TT^d)}=\norm{\rho_0}_{L^1(\TT^d)},\end{equation}
and which satisfies, for some $c\in(0,\infty)$, for almost every $t\in[0,T]$,
\begin{equation}\label{ge_2}
\int_{\TT^d}\Psi_\Phi(\rho(x,t))\dx+\int_0^t\int_{\TT^d}\abs{\nabla\Phi^\frac{1}{2}(\rho(x,s))}^2\dx\ds \leq c\left(\int_{\TT^d}\Psi_\Phi(\rho_0(x))\dx+\norm{g}^2_{L^2(\TT^d\times[0,T];\R^d)}\right).
\end{equation}
\end{prop}

\begin{proof}  We will first prove the existence of a weak solution for initial data $\rho_0\in L^\infty(\TT^d)$.  For every $\eta_2,\eta_3\in(0,1)$ let $\rho^{\eta_2,\eta_3}\in L^2([0,T];H^1(\TT^d))$ denote a nonnegative solution of \eqref{se_1} constructed in Proposition~\ref{smooth_exist} with initial data $\rho_0$ and control $g$.  It follows from \eqref{se_00002} that
\begin{equation}\label{ge_3}\norm{\rho^{\eta_2,\eta_3}}_{L^\infty([0,T];L^1(\TT^d))}=\norm{\rho_0}_{L^1(\TT^d)}.\end{equation}
It follows from Assumption~\ref{as_compact}, Lemma~\ref{as_interpolate}, Proposition~\ref{frac_est}, and \eqref{ge_3} that there exists $c\in(0,\infty)$ such that, for every $\eta_2,\eta_3\in(0,1)$,
\begin{equation}\label{ge_4}\begin{aligned}
\norm{\Phi^\frac{1}{2}(\rho^{\eta_2,\eta_3})}^2_{L^2([0,T];H^1(\TT^d))} & \leq c\left(\norm{\nabla\Phi^\frac{1}{2}(\rho^{\eta_2,\eta_3})}^2_{L^2([0,T];L^2(\TT^d;\R^d))}+\norm{\rho_0}^m_{L^1(\TT^d)}\right)
\\ & \leq c_3\left(1+\int_{\TT^d}\Psi_\Phi(\rho_0)\dx+\norm{g}^2_{L^2(\TT^d\times[0,T];\R^d)}+\norm{\rho_0}^m_{L^1(\TT^d)}\right),
\end{aligned}\end{equation}
where Assumption~\ref{as_compact}, the definition of $\Psi_\Phi$, and $\rho_0\in L^\infty(\TT^d)$ imply that the righthand side of \eqref{ge_4} is finite.  Estimate \eqref{se_3} proves that there exists $c\in(0,\infty)$ independent of $\eta_2,\eta_3\in(0,1)$ such that
\begin{equation}\label{ge_6}\begin{aligned}
& \eta_2^\frac{1}{2}\norm{\rho^{\eta_2,\eta_3}}_{L^\infty([0,T];L^2(\TT^d))}+\eta_2\norm{\nabla \rho^{\eta_2,\eta_3}}_{L^2([0,T];L^2(\TT^d;\R^d))}
\\ & \leq c\left(\eta_2^\frac{1}{2}\norm{\rho_0}_{L^2(\TT^d)}+\norm{g}_{L^2(\TT^d\times[0,T])}\norm{\Phi^{\frac{1}{2},\eta_3}}_{L^\infty([0,\infty))}\right).
\end{aligned}\end{equation}
It follows from estimate \eqref{fd_0}, equation \eqref{se_1}, H\"older's inequality and the Sobolev embedding theorem that there exists $c\in(0,\infty)$ such that, for every $\eta_2,\eta_3\in(0,1)$, for every $\psi\in\C^\infty(\TT^d\times[0,T])$, omitting the domains of the respective Sobolev spaces,
\[\begin{aligned}
& \abs{\int_{\TT^d}\partial_t\rho^{\eta_2,\eta_3}(x,t)\psi(x)\dx\dt}
\\ & \leq \left(2\norm{\Phi^\frac{1}{2}(\rho^{\eta_2,\eta_3})\nabla\Phi^\frac{1}{2}(\rho^{\eta_2,\eta_3})}_{L^1}+\eta_2\norm{\nabla \rho^{\eta_2,\eta_3}}_{L^1}+\norm{g\Phi^{\frac{1}{2},\eta_3}(\rho^{\eta_2,\eta_3})}_{L^1}\right)\norm{\nabla\psi}_{L^\infty}
\\ & \leq c\left(2\norm{\Phi^\frac{1}{2}(\rho^{\eta_2,\eta_3})\nabla\Phi^\frac{1}{2}(\rho^{\eta_2,\eta_3})}_{L^1}+\eta_2\norm{\nabla \rho^{\eta_2,\eta_3}}_{L^2}+\norm{g\Phi^\frac{1}{2}(\rho^{\eta_2,\eta_3})}_{L^1}\right)\norm{\nabla\psi}_{H^{(\nicefrac{d}{2}+2)}}.
\end{aligned}\]
Therefore, H\"older's inequality and Young's inequality prove that, for $c\in(0,\infty)$ independent of $\eta_2,\eta_3\in(0,1)$,
\begin{equation}\label{ge_8}\begin{aligned}
& \norm{\partial_t\rho^{\eta_2,\eta_3}}_{L^1([0,T];H^{-(\nicefrac{d}{2}+2)}(\TT^d))}
\\ & \leq c\left(\norm{\Phi^\frac{1}{2}(\rho^{\eta_2,\eta_3})}^2_{L^2([0,T];H^1(\TT^d))}+\eta_2\norm{\nabla \rho^{\eta_2,\eta_3}}_{L^2([0,T];L^2(\TT^d))}+\norm{g}^2_{L^2([0,T];L^2(\TT^d))}\right).
\end{aligned}\end{equation}
Estimate \eqref{ge_4} proves that for every $\eta_3\in(0,1)$ the righthand side of \eqref{ge_8} is uniformly bounded in $\eta_2\in(0,1)$.

In combination \eqref{ge_3}, \eqref{ge_4}, \eqref{ge_6}, \eqref{ge_8}, and Lemma~\ref{as_pre} prove that for every $\eta_3\in(0,1)$ there exists $\rho^{\eta_3}\in L^1([0,T];L^1(\TT^d))$ such that, after passing to a subsequence $\{\eta^k_2\rightarrow 0\}_{k\in\N}$, as $k\rightarrow\infty$,
\begin{equation}\label{ge_9}\rho^{\eta^k_2,\eta_3}\rightarrow \rho^{\eta_3}\;\;\textrm{almost everywhere and strongly in}\;\;L^1([0,T];L^1(\TT^d)).\end{equation}
It then follows from \eqref{ge_4} and \eqref{ge_9} that, as $k\rightarrow\infty$,
\begin{equation}\label{ge_10} \Phi^\frac{1}{2}(\rho^{\eta^k_2,\eta_3})\rightharpoonup\Phi^\frac{1}{2}(\rho^{\eta_3})\;\;\textrm{weakly in}\;\;L^2([0,T];H^1(\TT^d)),\end{equation}
and it follows from Lemma~\ref{as_pre}, \eqref{fd_0}, \eqref{fd_2}, \eqref{ge_4}, \eqref{ge_8}, and \eqref{ge_9} and the dominated convergence theorem that
\begin{equation}\label{ge_11} \lim_{k\rightarrow\infty}\Phi^{\frac{1}{2},\eta_k}(\rho^{\eta_k})=\Phi^\frac{1}{2}(\rho)\;\;\textrm{strongly in}\;\;L^2([0,T];L^2(\TT^d)).\end{equation}
Since it follows from \eqref{ge_6} that, for every $\eta_3\in(0,1)$ and $\psi\in\C^\infty(\TT^d\times[0,T])$,
\[\limsup_{\eta_2 \rightarrow \infty}\abs{\int_0^T\int_{\TT^d}\eta_2\nabla \rho^{\eta_2,\eta_3}\cdot \nabla\psi\dx\dt}=0,\]
estimates \eqref{ge_9}, \eqref{ge_10}, and \eqref{ge_11} prove that $\rho^{\eta_3}$ is a nonnegative weak solution of the equation
\begin{equation}\label{ge_012} \partial_t\rho^{\eta_3} = \Delta\Phi(\rho^{\eta_3})-\nabla\cdot\left(\Phi^{\frac{1}{2},\eta_3}(\rho^{\eta_3})g\right)\;\;\textrm{in}\;\;\TT^d\times(0,T)\;\;\textrm{with}\;\;\rho^{\eta_3}(\cdot,0)= \rho_0.\end{equation}
It follows from \eqref{ge_10}, \eqref{ge_11}, and the weak lower-semicontinuity of the Sobolev norm that there exists $c_1,c_2\in(0,\infty)$ such that, for every $\eta_3\in(0,1)$,
\[
\norm{\Phi^\frac{1}{2}(\rho^{\eta_3})}^2_{L^2([0,T];H^1(\TT^d))} \leq c_1\left(1+\int_{\TT^d}\Psi_\Phi(\rho_0)\dx+\norm{g}^2_{L^2(\TT^d\times[0,T];\R^d)}+\norm{\rho_0}^{c_2}_{L^1(\TT^d)}\right),
\]
and that there exists $c\in(0,\infty)$ such that, for every $\eta_3\in(0,1)$,
\begin{equation}\label{ge_21}
\norm{\partial_t\rho^{\eta_3}}_{L^1([0,T];H^{-(\nicefrac{d}{2}+2)}(\TT^d))} \leq c\left(\norm{\Phi^\frac{1}{2}(\rho^{\eta_2,\eta_3})}^2_{L^2([0,T];H^1(\TT^d))}+\norm{g}^2_{L^2([0,T];L^2(\TT^d))}\right).
\end{equation}
And since \eqref{ge_3} and \eqref{ge_9} prove that, for every $\eta_3\in(0,1)$,
\[\norm{\rho^{\eta_3}}_{L^\infty([0,T];L^1(\TT^d))}=\norm{\rho_0}_{L^1(\TT^d)},\]
a repetition of the arguments leading from \eqref{ge_9} to \eqref{ge_012} prove that there exists a solution $\rho\in L^\infty([0,T];L^1(\TT^d))$ to \eqref{ge_1} which satisfies \eqref{ge_02} and \eqref{ge_2}.  This completes the proof for initial data $\rho_0\in L^\infty(\TT^d)$.

To complete the proof for general initial data, let $\rho_0\in \Ent_\Phi(\TT^d)$ and let $\{\rho_0^n\}_{n\in\N}\subseteq L^\infty(\TT^d)$ be a sequence with uniformly bounded entropy such that, as $n\rightarrow\infty$, $\rho_0^n\rightarrow \rho_0$ strongly in $L^1(\TT^d)$.  The existence of such a sequence follows from \cite[Proposition~7.12]{FehGes19}.  Indeed, it suffices to choose $\rho^n_0=\rho_\wedge n$.  Then, for every $n\in\N$ let $\rho^n\in L^\infty([0,T];L^1(\TT^d))$ denote a solution of \eqref{ge_1} that satisfies \eqref{ge_02} and \eqref{ge_2}.  By Theorem~\ref{thm_unique} and Theorem~\ref{equiv}, the solutions $\rho^n$ are unique and form a Cauchy sequence in $L^\infty([0,T];L^1(\TT^d))$.  Therefore, there exists $\rho\in L^\infty([0,T];L^1(\TT^d))$ such that, as $n\rightarrow\infty$, $\rho^n\rightarrow\rho$ strongly in $L^\infty([0,T];L^1(\TT^d))$.  A repetition of the arguments leading from \eqref{ge_9} to \eqref{ge_11} prove that $\rho$ is a weak solution of \eqref{ge_1} with initial data $\rho_0$ that satisfies \eqref{ge_02} and \eqref{ge_2}.  This completes the proof.  \end{proof}

\subsection{Weak-strong continuity}\label{subsection_weak_strong}  We conclude this section with the proof of weak-strong continuity.  Precisely, we will prove that a weakly convergent sequence of controls induces a strongly convergent sequence of solutions.

\begin{prop}\label{weak_strong}  Let $T\in(0,\infty)$, let $\Phi\in\C([0,\infty))\cap\C^1_{\textrm{loc}}((0,\infty))$ satisfy Assumptions~\ref{as_unique}, ~\ref{as_equiv}, \ref{as_compact}, and let $\rho\in \Ent_\Phi(\TT^d)$.  Assume that $\{g_n\}_{n\in\N}\subseteq L^2(\TT^d\times[0,T];\R^d)$ and $g\in L^2(\TT^d\times[0,T];\R^d)$ satisfy, as $n\rightarrow\infty$,
\[g_n\rightharpoonup g\;\;\textrm{weakly in}\;\;L^2([0,T]\times \TT^d;\R^d).\]
Then, the renormalized kinetic solutions $\{\rho_n\}_{n\in\N}$ and $\rho$ with controls $\{g_n\}_{n\in\N}$ and $g$, and with initial data $\rho_0$ satisfy that, as $n\rightarrow\infty$,
\[\rho_n\rightarrow \rho\;\;\textrm{strongly in}\;\;L^1([0,T];L^1(\TT^d)).\]
\end{prop}

\begin{proof} For every $n\in\N$ let $\rho_n\in L^\infty([0,T];L^1(\TT^d))$ denote the weak solution in the sense of Definition~\ref{classical_weak} constructed in Proposition~\ref{gen_exist} with control $g_n$ and initial data $\rho_0$, and let $\rho\in L^\infty([0,T];L^1(\TT^d))$ denote the weak solution with control $g$ and initial data $\rho_0$.  Theorem~\ref{equiv} proves that the $\{\rho_n\}_{n\in\N}$ and $\rho$ are renormalized kinetic solutions in the sense of Definition~\ref{skel_sol_def}, and Theorem~\ref{thm_unique} proves that they are the unique renormalized kinetic solutions.  A repetition of the argument leading to the proof of Proposition~\ref{gen_exist}, based on estimates \eqref{ge_02}, \eqref{ge_2}, and \eqref{ge_21} proves that, as $n\rightarrow\infty$, $\rho_n\rightarrow\rho$ strongly in $L^1([0,T];L^1(\TT^d))$.
This completes the proof. \end{proof}

\section{The uniform large deviations principle for conservative SPDE}\label{sec_uniform_ldp}  In this section, we will establish a uniform large deviations principle for the solutions of the equation
\begin{equation}\label{sldp_1}\partial_t\rho^\ve = \Delta\Phi(\rho^\ve)-\sqrt{\ve}\nabla\cdot(\Phi^\frac{1}{2}(\rho^\ve)\circ\xi^{K(\ve)})\;\;\textrm{in}\;\;\TT^d\times(0,T),\end{equation}
with respect to weakly compact subsets of $L^1(\TT^d)$ with bounded entropy, for suitable spatially correlated noise $\xi^{K(\ve)}$ that converges to space-time white noise $\xi$ as $\ve\rightarrow 0$.  This section is separated into two subsections.  Section~\ref{subsection_noise} defines the noise $\xi^{K(\ve)}$ and the well-posedness theory for \eqref{sldp_1}, and Section~\ref{subsection_ldp} proves the uniform large deviations principle.

\subsection{The noise and well-posedness of \eqref{sldp_1}}\label{subsection_noise}  We will focus on spectral approximations of space-time white noise.  However, we emphasize that our methods are general, and apply without any essential changes to noise of a general type.  See Remark~\ref{smooth_white_noise} below.

Let $\{B^k,W^k\}_{k\in\N_0}$ be independent $d$-dimensional Brownian motions defined on some probability space $(\O,\F,\P)$ with respect to the canonical filtration $(\F_t)_{t\in[0,\infty)}$.  We consider the Brownian motions $\{B^k,W^k\}_{k\in\N_0}$ to be taking values in the space
\[\R^\infty=\{(x_i)_{i\in\N}\colon x_i\in \R^d,\;\forall\;i\in\N\},\]
equipped with the metric topology of coordinate-wise convergence.  We then define for every $K\in\N$ the spectral approximation of the noise $\xi^K$, for the sum over elements $\abs{k}\leq K$ in $\Z^d$,
\begin{equation}\label{final_noise}\dd\xi^K = \sum_{\abs{k}\leq K}\left(\sin(k\cdot x)\dd B^k_t + \cos(k\cdot x)\dd W^k_t\right).\end{equation}
For this noise, we will understand the Stratonovich SPDE
\[\partial_t\rho =\Delta\Phi(\rho)-\sqrt{\ve}\nabla\cdot\left(\Phi^\frac{1}{2}(\rho)\circ \xi^K\right),\]
in the It\^o formulation, for $N_K = \#\{k\in\Z^d\colon \abs{k}\leq K\}$,
\begin{equation}\label{entropy_eq}
\partial_t\rho =\Delta\Phi(\rho)-\sqrt{\ve}\nabla\cdot\left(\Phi^\frac{1}{2}(\rho)\xi^K\right)+\frac{\ve N_K}{2}\nabla\cdot\left(\left(\Phi^\frac{1}{2}\right)'(\rho)\nabla \Phi^\frac{1}{2}(\rho)\right).
\end{equation}
To prove the large deviations result, we will also consider the controlled equation
\begin{equation}\label{entropy_con_eq}
\partial_t\rho^{\ve,K} =\Delta\Phi(\rho^{\ve,K})-\sqrt{\ve}\nabla\cdot\left(\Phi^\frac{1}{2}(\rho^{\ve,K})\circ \xi^K\right)-\nabla\cdot(\Phi^\frac{1}{2}(\rho^{\ve,K})P_Kg),
\end{equation}
where $P_Kg$ denotes the Fourier projection of a random control $g\in L^2(\O\times[0,T];L^2(\TT^d;\R^d))$ onto the span of $\{\sin(k\cdot x),\cos(k\cdot x)\}_{\{\abs{k}\leq K\}}$.  The well-posedness of \eqref{entropy_eq} and \eqref{entropy_con_eq} is a consequence of \cite{FehGes21} and the Girsanov theorem, which is summarized in Theorem~\ref{thm_entropy} and Proposition~\ref{prop_prob_con_strong} below.  Finally, in Proposition~\ref{weak_est} we recall a priori estimates for the solutions, and in Proposition~\ref{law_tight} we prove the tightness of the laws of the solutions in a certain scaling regime.

\begin{thm}\label{thm_entropy}  Let $\Phi\in\C([0,\infty))\cap\C^1_{\textrm{loc}}((0,\infty))$ satisfy Assumptions~\ref{as_unique},  \ref{as_equiv}, and \ref{as_compact} and let $\rho_0\in\Ent_\Phi(\TT^d)$.  Then, there exists a unique stochastic kinetic solution $\rho^{\ve,K}\in L^\infty(\Omega\times[0,T];L^1(\TT^d))$ of \eqref{entropy_eq} with initial condition $\rho_0$ in the sense of \cite[Definition~3.4]{FehGes21}.  If $\rho_1,\rho_2$ are two solutions of \eqref{entropy_eq} with initial data $\rho_0^1,\rho_0^2\in \Ent_\Phi(\TT^d)$, then almost surely
\begin{equation}\label{thm_entropy_contraction}\sup_{t\in[0,T]}\norm{\rho^1(\cdot,t)-\rho^2(\cdot,t)}_{L^1(\TT^d)}\leq\norm{\rho_0^1-\rho_0^2}_{L^1(\TT^d)}.\end{equation}
Furthermore, there exists a measurable map
\[S_{\ve,K}\colon \Ent_\Phi(\TT^d)\times \C([0,T];\R^\infty)\rightarrow L^\infty([0,T];L^1(\TT^d))\]
such that, for every $\rho_0\in\Ent_\Phi(\TT^d)$, $\P$-a.s.,
\[S_{\ve,K}(\rho_0,(B^k,W^k)_{k\in\N_0})=\rho^{\ve,K},\]
where $\rho^{\ve,K}$ is the unique entropy solution of \eqref{entropy_eq} with initial data $\rho_0$.
\end{thm}

\begin{proof}  Fix the Brownian motions $\overline{B}=(B^k,W^k)_{k\in\N_0}\in\C([0,\infty);\R^\infty)$ on a probability space $(\O,\F,\P)$ with respect to the filtration $(\mathcal{F}_t)_{t\in[0,\infty)}$ and fix $K\in\N$ and $\ve\in(0,1)$.  The existence of an $\F_t$-measurable solution and the almost sure $L^1$-contraction estimate \eqref{thm_entropy_contraction} is shown in \cite[Theorem~4.7, Theorem~5.29, Corollary~5.31]{FehGes21}.  It follows from the measurability that, for every $\rho\in \Ent_\Phi(\TT^d)$, there exists a measurable function $S^{\ve,K}_\rho\colon\C([0,\infty);\R^\infty)\rightarrow L^1([0,T];L^1(\TT^d))$ such that, $\P$-a.s.,
\begin{equation}\label{tec_0}S^{\ve,K}_\rho(\overline{B})=\rho^{\ve,K}.\end{equation}
Since it follows from Proposition~\ref{bounded_entropy} below that $\Ent_\Phi(\TT^d)$ equipped with strong $L^1(\TT^d)$-topology is separable, let $\{\rho_n\}_{n\in\N}$ be a countable dense subset of $\Ent_\Phi(\TT^d)$.  It follows from \eqref{thm_entropy_contraction} and the countability of the set $\N\times\N$ that, on a measurable 
subset of full probability,
\begin{equation}\label{tec_1}\sup_{t\in[0,T]}\norm{S^{\ve,K}_{\rho_n}(\overline{B})-S^{\ve,K}_{\rho_m}(\overline{B})}_{L^1(\TT^d)}\leq \norm{\rho_n-\rho_m}_{L^1(\TT^d)}\;\;\textrm{for every}\;\;n,m\in\N.\end{equation}
It follows from the density of the $\{\rho_n\}_{n\in\N}$ with respect to the strong $L^1(\TT^d)$-norm and \eqref{tec_1} that there almost surely exists a strongly continuous function $\overline{S}^{\ve,K}_{\overline{B}}\colon\Ent_\Phi(\TT^d)\rightarrow L^1([0,T];L^1(\TT^d))$ such that, for every $n\in\N$, on a subset of full probability independent of $n$,
\begin{equation}\label{tec_02} \overline{S}^{\ve,K}_{\overline{B}}(\rho_n)=S^{\ve,K}_{\rho_n}(\overline{B}).\end{equation}
It then follows from \eqref{tec_1}, \eqref{tec_02}, and a simplified version of \cite[Theorem~5.29, Corollary~5.31]{FehGes21} that, for every $\rho\in\Ent_\Phi(\TT^d)$, on a subset of full probability depending on $\rho$,
\begin{equation}\label{tec_2} \overline{S}^{\ve,K}_{\overline{B}}(\rho)=S^{\ve,K}_{\rho}(\overline{B}).\end{equation}
We define $S^{\ve,K}\colon\Ent_\Phi(\TT^d)\times \C([0,T];\R^\infty)\rightarrow L^\infty([0,T];L^1(\TT^d))$ by the rule
\[S^{\ve,K}(\rho,\overline{B})=\overline{S}^{\ve,K}_{\overline{B}}(\rho),\]
from which it follows from \eqref{tec_0} that, for every $\rho\in\Ent_\Phi(\TT^d)$,
\begin{equation}\label{tec_3} S^{\ve,K}(\rho,\cdot)\colon\C([0,\infty);\R^\infty)\rightarrow L^1([0,T];L^1(\TT^d))\;\;\textrm{is measurable,} \end{equation}
and from \eqref{tec_2} that, for almost every realization of $\overline{B}$,
\begin{equation}\label{tec_4} S^{\ve,K}(\cdot,\overline{B})\colon\Ent_\Phi(\TT^d)\rightarrow L^1([0,T];L^1(\TT^d))\;\;\textrm{is strongly continuous.}\end{equation}
In combination \eqref{tec_3}, \eqref{tec_4}, and the separability of $\Ent_\Phi(\TT^d)$ prove that $S^{\ve,K}$ is measurable, which with \eqref{tec_0} and \eqref{tec_2} completes the proof.  \end{proof}

\begin{prop}\label{prop_prob_con_strong}  Let $T\in(0,\infty)$ and let $\Phi\in\C([0,\infty))\cap\C^1_{\textrm{loc}}((0,\infty))$ satisfy Assumptions~\ref{as_unique},  \ref{as_equiv}, and \ref{as_compact}.  For every $\ve\in(0,1)$ and $K\in\N$ let
\[S_{\ve,K}\colon \Ent_\Phi(\TT^d)\times \C([0,T];\R^\infty)\rightarrow L^\infty([0,T];L^1(\TT^d))\]
be defined in Theorem~\ref{thm_entropy}.  Then, for every predictable process $g\in L^2(\O\times[0,T];L^2(\TT^d;\R^d))$ and $\rho_0\in\Ent_\Phi(\TT^d)$, for $g_k(\o,s)=\langle g(\o,s),\cos(k\cdot x)\rangle_{L^2(\TT^d)}$ and $g'_k(\o,s)=\langle g(\o,s),\sin(k\cdot x)\rangle_{L^2(\TT^d)}$,
\[\rho^{\ve,K}=S_{\ve,K}\left(\rho_0,(B^k_\cdot,W^k_\cdot)_{k\in\N_0}+\ve^{-\frac{1}{2}}\left(\int_0^\cdot g_k(\o,s)\ds, \int_0^\cdot g'_k(\o,s)\ds\right)_{k\in\N_0}\right),\]
is the unique solution of \eqref{entropy_con_eq} in the sense of \cite[Definition~3.4]{FehGes21}.\end{prop}
\begin{proof}  The proof is a consequence of Theorem~\ref{thm_entropy} and the Girsanov theorem (see, for example, \cite[Theorem~10]{BudDupMar2008}).\end{proof}

\begin{prop}\label{weak_est}  Let $\Phi\in\C([0,\infty))\cap\C^1_{\textrm{loc}}((0,\infty))$ satisfy Assumptions~\ref{as_unique},  \ref{as_equiv}, \ref{as_compact}, let $\Psi_\Phi$ be as in \eqref{psi_phi}, let $g\in L^2(\O\times[0,T];L^2(\TT^d;\R^d))$ be a predictable process, and let $\rho_0\in\Ent_\Phi(\TT^d)$.  Assume that $\ve\in(0,1)$ and $K\in\N$ satisfy $\ve K^{d+2}\leq 1$, and let $\rho^{\ve,K}\in L^\infty(\Omega\times[0,T];L^1(\TT^d))$ be the unique solution of \eqref{entropy_con_eq} in the sense of Proposition~\ref{prop_prob_con_strong}.  Then there exist $c\in(0,\infty)$ such that
\begin{align*}
& \E\left[\sup_{t\in[0,T]}\int_{\TT^d}\Psi_\Phi(\rho^{\ve,K}(x,t))\dx\right]+\E\left[\int_0^T\int_{\TT^d}\abs{\nabla\Phi^\frac{1}{2}(\rho^{\ve,K})}^2\dx\ds\right]
\\ & \leq c\left(1+T\E\left[\norm{\rho_0}_{L^1(\TT^d)}^m\right]+\E\left[\int_{\TT^d}\Psi_\Phi(\rho_0(x))\dx\right]+\E\left[\int_0^T\int_{\TT^d}\abs{P_Kg}^2\dx\ds\right]\right).
\end{align*}

\end{prop}

\begin{proof}  The proof is a consequence of \cite[Proposition~5.18, Corollary~5.31]{FehGes21}. \end{proof}

\begin{prop}\label{law_tight}  Let $\Phi\in\C([0,\infty))\cap\C^1_{\textrm{loc}}((0,\infty))$ satisfy Assumptions~\ref{as_unique},  \ref{as_equiv}, \ref{as_compact}, let $g\in L^2(\O\times[0,T];L^2(\TT^d;\R^d))$ be a predictable process, and let $\rho_0\in L^2(\TT^d)$.  Let $\{\ve,K(\ve)\}_{\ve\in(0,1)}$ be a sequence that satisfies
\[\ve K(\ve)^{d+2}\rightarrow 0\;\;\textrm{and}\;\;K(\ve)\rightarrow \infty\;\;\textrm{as}\;\;\ve\rightarrow 0.\]
Then the laws of the solutions $\{\rho^{\ve,K(\ve)}\}_{\ve\in(0,1)}$ of \eqref{entropy_con_eq} with initial data $\rho_0$ and control $g$ are tight on $L^1([0,T];L^1(\TT^d))$.
\end{prop}

\begin{proof}  The proof is a consequence of Lemma~\ref{as_pre}, Proposition~\ref{weak_est}, and \cite[Proposition~5.26]{FehGes21}. \end{proof}

\subsection{The uniform large deviations principle}\label{subsection_ldp}  We will now use the well-posedness results of Theorem~\ref{thm_entropy} and Proposition~\ref{prop_prob_con_strong}, the estimates of Proposition~\ref{weak_est}, and the convergence of Theorem~\ref{weak_thm_con} to establish a large deviations principle for the solutions of the equation
\begin{equation}\label{collapse_eq}
\partial_t\rho^\ve =\Delta\Phi(\rho^\ve)-\sqrt{\ve}\nabla\cdot\left(\Phi^\frac{1}{2}(\rho^\ve)\circ\xi^K\right)\;\;\textrm{with}\;\;\rho^\ve(\cdot,0)=\rho_0,
\end{equation}
about the solution to the limiting equation
\[ \partial_t\rho =\Delta\Phi(\rho)\;\;\textrm{with}\;\;\rho(\cdot,0)=\rho_0,\]
along the scaling regime identified in Proposition~\ref{law_tight}.  The large deviations principle is a consequence of the weak approach to large deviations established in \cite{BudDupMar2008}, as well as Dupuis and Ellis \cite{DupEll1997} and Budhiraja and Dupuis \cite{BudDup2019}.  Precisely, we will show that the solutions satisfy a large deviations principle with rate function $I_{\rho_0}\colon L^1([0,T];L^1(\TT^d))\rightarrow[0,\infty]$ defined by
\begin{equation}\label{collapse_ldp_rate}I_{\rho_0}(\rho)=\frac{1}{2}\inf\left\{\norm{g}^2_{L^2([0,T];L^2(\TT^d))}\colon \partial_t\rho=\Delta\Phi(\rho)-\nabla\cdot \left(\Phi^\frac{1}{2}(\rho)g\right)\;\;\textrm{and}\;\;\rho(\cdot,0)=\rho_0\right\}.\end{equation}
The solutions appearing in the definition of the rate function are understood in the sense of Definition~\ref{classical_weak}.  The methods of \cite[Theorem~6]{BudDupMar2008} and Budhiraja, Dupuis, and Salins \cite[Theorem~4.3]{BudDupSal} in fact prove the stronger statement that the solutions of \eqref{collapse_eq} satisfy a large deviations principle uniformly with respect to subsets of the initial data with bounded entropy (see \eqref{def_ent_r} and Definition~\ref{def_uniform_Laplace} below), which we now explain.

We define the space of nonnegative $L^1(\TT^d)$-functions with bounded entropy, for every $R\in(0,\infty)$,
\begin{equation}\label{def_ent_r}\Ent_{\Phi,R}(\TT^d)=\left\{\rho\in L^1(\TT^d)\colon \rho\geq 0\;\;\textrm{a.e.\ and}\;\;\int_{\TT^d}\Psi_\Phi(\rho)\dx\leq R\right\}.\end{equation}
Proposition~\ref{bounded_entropy} below proves that $\Ent_{\Phi,R}(\TT^d)$ equipped with the induced $L^1(\TT^d)$-topology is a Polish space.  The proof of the large deviations principle follows the methods of \cite[Theorem~6]{BudDupMar2008}, which relies on establishing two facts.  First, we require the compactness of the family of solutions to the skeleton equation with $L^2$-bounded controls.  This is a straightforward consequence of Lemma~\ref{as_pre} and Proposition~\ref{frac_est}.  Second, in the scaling regime that $\ve K(\ve)^{d+2}\rightarrow 0$, we require that the solutions of the controlled equation \eqref{entropy_con_eq} converge to the solution of the skeleton equation.  We prove this statement in Theorem~\ref{weak_thm_con} below.  We prove the uniform large deviations principle in Theorem~\ref{prop_collapse_3}.

\begin{prop}\label{bounded_entropy} Let $\Phi\in\C([0,\infty))\cap\C^1_{\textrm{loc}}((0,\infty))$ satisfy Assumptions~\ref{as_unique},  \ref{as_equiv}, and \ref{as_compact}, let $R\in(0,\infty)$, and let $\Ent_{\Phi,R}(\TT^d)$ be defined by \eqref{def_ent_r}.  Then, $\Ent_{\Phi,R}(\TT^d)$ equipped with the $L^1(\TT^d)$-topology is a complete separable metric space.  \end{prop}

\begin{proof}  For complete details see \cite[Proposition~7.12]{FehGes19}.  The completeness is a consequence of the completeness of $L^1(\TT^d)$ and Fatou's lemma.  The separability follows from the convexity of $\Psi_\Phi$ and Assumption~\ref{as_compact}, which implies for each $R\in(0,\infty)$ that $L^2(\TT^d)\cap\Ent_{\Phi,R}(\TT^d)$ is dense in $\Ent_{\Phi,R}(\TT^d)$.  Since $L^2(\TT^d)\cap\Ent_{\Phi,R}(\TT^d)$ is separable as a closed, convex subset of the separable, reflexive, strictly convex Banach space $L^2(\TT^d)$, this completes the proof.\end{proof}

\begin{thm}\label{weak_thm_con}  Let $\Phi\in\C([0,\infty))\cap\C^1_{\textrm{loc}}((0,\infty))$ satisfy Assumptions~\ref{as_unique},  \ref{as_equiv}, \ref{as_compact}, let  $g^\ve,g\in L^2(\O\times[0,T];L^2(\TT^d;\R^d))$, for $\ve\in(0,1)$, be predictable processes that satisfy
\[\sup_{\ve\in(0,1)}\norm{g^\ve}_{L^\infty(\O;L^2(\TT^d\times[0,T];\R^d))}<\infty\;\;\textrm{and, as $\ve\rightarrow 0$,}\;\;g^\ve\rightarrow g\;\;\textrm{in distribution in}\;\;L^2(\TT^d\times[0,T];\R^d),\]
with $L^2(\TT^d\times[0,T];\R^d)$ endowed with the weak topology, let $\rho^\ve_0,\rho_0\in\Ent_\Phi(\TT^d)$, for $\ve\in(0,1)$, be such that
\[\rho^\ve_0\rightharpoonup\rho_0\;\;\textrm{as $\ve\rightarrow 0$ weakly in}\;\;L^1(\TT^d)\;\;\textrm{and}\;\;\sup_{\ve\in(0,1)}\left[\int_{\TT^d}\Psi_\Phi(\rho^\ve_0)\dx\right]<\infty,\]
and let $\{K(\ve)\}_{\ve\in(0,1)}$ be such that $\ve K(\ve)^{d+2}\rightarrow 0$ and $K(\ve)\rightarrow \infty$ as $\ve\rightarrow 0$.  Then the solutions $\{\rho^\ve=\rho^{\ve,K(\ve)}\}_{\ve\in(0,1)}$ of \eqref{entropy_con_eq} with initial data $\{\rho^\ve_0\}_{\ve\in(0,1)}$ satisfy, as $\ve\rightarrow 0$,
\[\rho^\ve\rightarrow \rho\;\;\textrm{in distribution on}\;\;L^1([0,T];L^1(\TT^d)),\] 
for $\rho$ the unique solution of \eqref{skeleton_equation} with control $g$ and initial data $\rho_0$.  \end{thm}

\begin{proof}  We will first consider nonnegative initial data $\rho^\ve_0$ and $\rho_0$ that are uniformly bounded in $L^2(\TT^d)$.  It follows from Lemma~\ref{as_pre}, Proposition~\ref{frac_est}, and \cite[Proposition~5.26]{FehGes21} that the laws of the solutions $\{\rho^\ve\}_{\ve\in(0,1)}$ are tight on $L^1([0,T];L^1(\TT^d))$, and it follows from Lemma~\ref{as_interpolate} and Proposition~\ref{frac_est} that the laws of $\{\Phi^\frac{1}{2}(\rho^\ve)\}_{\ve\in(0,1)}$ are tight in the metric topology of weak convergence on $L^2([0,T];H^1(\TT^d))$.  Finally, due to the $L^2$-integrability of the data, there exists a sequence of finite nonnegative kinetic measures $\{q^\ve\}_{\ve\in(0,1)}$ in the sense of \cite[Definitions~3.1, ~3.4]{FehGes21} that are uniformly bounded in expectation in the sense that
\[\sup_{\ve\in(0,1)}\E\left[q^\ve(\TT^d\times[0,T])\right]<\infty,\]
such that, for the kinetic function $\chi^\ve$ of $\rho^\ve$, for every $t\in[0,T]$ and $\psi\in\C^\infty_c(\TT^d\times(0,\infty))$, $\P$-a.s.,
\begin{align}\label{2_5000}
& \int_\R\int_{\TT^d}\chi^\ve(x,\xi,t)\psi(x,\xi) = \int_\R\int_{\TT^d}\overline{\chi}(\rho^\ve_0)\psi(x,\xi)-2\int_0^t\int_{\TT^d}\Phi^\frac{1}{2}(\rho^\ve)\nabla\Phi^\frac{1}{2}(\rho^\ve)\cdot(\nabla\psi)(x,\rho^\ve)
\\ \nonumber & + \int_0^t\int_{\TT^d} \Phi^\frac{1}{2}(\rho^\ve)P_{K(\ve)}g^\ve(x,t)\cdot (\nabla_x\psi)(x,\rho^\ve)+ \int_0^t\int_{\TT^d} \frac{2\Phi(\rho^\ve)}{\Phi'(\rho^\ve)}\nabla_x\Phi^\frac{1}{2}(\rho^\ve)\cdot P_{K(\ve)}g^\ve(x,t)(\partial_\xi\psi)(x,\rho^\ve)
\\ \nonumber & -\frac{\ve N_{K(\ve)}}{2}\int_0^t\int_{\TT^d}(\Phi^\frac{1}{2})'(\rho^\ve)\nabla\Phi^\frac{1}{2}(\rho^\ve)\cdot(\nabla\psi)(x,\rho^\ve)-\int_0^t\int_\R\int_{\TT^d}\partial_\xi\psi(x,\xi)\dd q^\ve
\\ \nonumber & +\frac{\ve M_{K(\ve)}}{2}\int_0^t\int_{\TT^d} \Phi(\rho^\ve)(\partial_\xi\psi)(x,\rho^\ve)-\sqrt{\ve}\int_0^t\int_{\TT^d}\psi(x,\rho^\ve)\nabla\cdot\left(\Phi^\frac{1}{2}(\rho^\ve)\dd\xi^{K(\ve)}\right),
\end{align}
for $M_{K(\ve)}=\sum_{\abs{k}\leq K(\ve)}k^2\simeq K(\ve)^{d+2}$ and for $P_{K(\ve)}g^\ve$ the Fourier projection of $g$. It is this term that determines the scaling for the large deviations principle, and it is the first term of the third line which necessitates that we consider test functions compactly supported in $\TT^d\times (0,\infty)$.  In general, for example when $\Phi^\frac{1}{2}(\xi)=\xi^\frac{1}{2}$, this term with $\psi=1$ is not integrable on the zero set of the solution.  We will first use the Skorokhod representation theorem in order to pass almost surely to the limit $\ve\rightarrow 0$ on an auxiliary probability space.

The difficulty in passing to the $\ve\rightarrow 0$ limit appears in the second term of the second line of \eqref{2_5000}, which as $\ve\rightarrow 0$ contains the product of weakly convergent sequences.  For every $\ve\in(0,1)$ let $p^\ve$ denote the nonnegative, almost surely finite measure on $\TT^d\times\R\times[0,T]$ defined by
\[\dd p^\ve = \delta_0(\xi-\rho^\ve)\abs{\nabla_x\Phi^\frac{1}{2}(\rho^\ve)}\abs{P_{K(\ve)}g^\ve}\dx\dt\;\;\textrm{on}\;\;\TT^d\times\R\times[0,T],\]
and observe that, for every $t\in[0,T]$,
\[\abs{\int_0^t\int_{\TT^d} \frac{2\Phi(\rho^\ve)}{\Phi'(\rho^\ve)}\nabla_x\Phi^\frac{1}{2}(\rho^\ve)\cdot P_{K(\ve)}g^\ve(x,t)(\partial_\xi\psi)(x,\rho^\ve)}\leq \int_0^t\int_{\TT^d}\frac{2\Phi(\xi)}{\Phi'(\xi)}\abs{\partial_\xi\psi(x,\xi)}\dd p^\ve.\]
We will characterize the limiting behavior of this term using the limiting behavior of these auxiliary measures.  Let $\overline{X}$ be the state space
\[L^1([0,T];L^1(\TT^d))\times L^2([0,T];H^1(\TT^d))\times\mathcal{M}(\TT^d\times\R\times[0,T])^2\times L^2(\TT^d\times[0,T];\R^d)\times \C([0,\infty);\R^\infty)\]
where $L^1([0,T];L^1(\TT^d))$ is equipped with the strong topology, where $L^2([0,T];H^1(\TT^d))$, the space of nonnegative Borel measures $\mathcal{M}(\TT^d\times\R\times[0,T])$, and $L^2(\TT^d\times[0,T];\R^d)$ are equipped with the metric topology of weak convergence, and $\C([0,\infty);\R^\infty)$ is equipped with the metric topology of component-wise convergence (see, for example, the metrics defined in \cite[Theorem~5.29]{FehGes21}).  It follows from the above and the assumptions on the controls that the laws of the random variables
\[\{X^\ve = (\rho^\ve,\Phi^\frac{1}{2}(\rho^\ve),q^\ve,p^\ve,g^\ve,(W^k,B^k)_{k\in\N_0})\}_{\ve\in(0,1)},\]
for the Brownian motions $(W^k,B^k)_{k\in\N_0}$ defined in Section~\ref{sldp_1}, are tight on $\overline{X}$.  Prokhorov's theorem (see, for example, Billingsley \cite[Chapter~1, Theorem~5.1]{Bil1999}) proves that, after passing to a subsequence still denoted $\ve\rightarrow 0$, there exists a probability measure $\mu$ on $\overline{X}$ such that, as $\ve\rightarrow 0$, we have $X^\ve\rightarrow\mu $ in law.  Since the space $\overline{X}$ is separable, the Skorokhod representation theorem (see, for example, \cite[Chapter~1, Theorem~6.7]{Bil1999}) proves that that there exists a stochastic basis $(\tilde{\O},\tilde{\F},\tilde{\P},(\tilde{\F}_t)_{t\in[0,\infty)})$ and $\overline{X}$-valued random variables $\{\tilde{X}^\ve\}_{\ve\in(0,1)}$ and $\tilde{X}$ on $\tilde{\O}$ such that, for every $\ve\in(0,1)$,
\[\tilde{X}^\ve = X^\ve\;\;\textrm{and}\;\;\tilde{X}=\mu\;\;\textrm{in law}.\]
And such that, $\tilde{\P}$-almost surely as $\ve\rightarrow 0$,
\[\tilde{X}^\ve\rightarrow\tilde{X}\;\;\textrm{in $\overline{X}$.}\]
It follows from \cite[Theorem~5.29]{FehGes21} that, for every $\ve\in(0,1)$, $\tilde{\P}$-almost surely there exists $\tilde{\rho}^\ve$, a kinetic measure $\tilde{q}^\ve$ for $\tilde{\rho}^\ve$ in the sense of \cite[Definitions~3.1, ~3.4]{FehGes21}, and Brownian motions $(\tilde{B}^{k,\ve},\tilde{W}^{k,\ve})_{k\in\N_0}$ such that
\[\tilde{X}^\ve = (\tilde{\rho}^\ve,\Phi^\frac{1}{2}(\tilde{\rho}^\ve),\tilde{q}^\ve,\tilde{p}^\ve,\tilde{g}^\ve,(\tilde{B}^{k,\ve},\tilde{W}^{k,\ve})_{k\in\N_0}),\]
that $\tilde{\P}$-almost satisfies \eqref{2_5000} with respect to the noise $\dd\tilde{\xi}^{K(\ve),\ve}$ defined by $(\tilde{B}^{k,\ve},\tilde{W}^{k,\ve})$ as in \eqref{final_noise}.  Similarly, following \cite[Theorem~5.29]{FehGes21}, it follows that
\[\tilde{X} = (\tilde{\rho},\Phi^\frac{1}{2}(\tilde{\rho}),\tilde{q},\tilde{p},\tilde{g},(\tilde{B}^{k},\tilde{W}^{k})_{k\in\N_0}),\]
for a kinetic measure $\tilde{q}$ for $\tilde{\rho}$.  In order to characterize $\tilde{\rho}$, will pass to the $\ve\rightarrow 0$ limit in \eqref{2_5000}.  However, the difficulty in doing so is that we cannot identify the limit
\[\lim_{\ve\rightarrow 0}\left(\int_0^t\int_{\TT^d} \frac{2\Phi(\tilde{\rho}^\ve)}{\Phi'(\tilde{\rho}^\ve)}\nabla_x\Phi^\frac{1}{2}(\tilde{\rho}^\ve)\cdot P_{K(\ve)}\tilde{g}^\ve(x,t)(\partial_\xi\psi)(x,\tilde{\rho}^\ve)\right),\]
due to the product of the weakly convergent gradient and control.

The essential observation is that this ill-defined limit does not appear in the classical weak formulation of the skeleton equation.  We will therefore prove $\tilde{\P}$-almost surely that the $\tilde{\rho}^\ve$ converge to a weak solution of the skeleton equation in the sense of Definition~\ref{classical_weak}.  It then follows from Theorem~\ref{thm_unique} and Theorem~\ref{equiv} that this uniquely characterizes the limit.  Precisely, for every $\d\in(0,1)$, let $\phi_\d\colon[0,\infty)\rightarrow[0,1]$ be a smooth function that satisfies, for some $c\in(0,\infty)$ independent of $\delta\in(0,\nicefrac{1}{2})$,
\[\phi_\d=0\;\;\textrm{on}\;\;[0,\d]\cup [2\delta^{-1},\infty),\;\;\phi_\d=1\;\;\textrm{on}\;\;[2\d,\delta^{-1}],\;\;\textrm{and}\;\;\abs{\phi_\d'}\leq \nicefrac{c}{\d}\mathbf{1}_{\{\xi\in [\d,2\d]\}}+c\d\mathbf{1}_{\{\xi\in[\delta^{-1},2\delta^{-1}]\}},\]
and let $\eta\in\C^\infty(\TT^d)$ be arbitrary.  We will consider test functions of the form $\psi_\d(x,\xi)=\phi_\d(\xi)\eta(x)$ and pass first to the limit $\ve\rightarrow 0$ and then $\d\rightarrow 0$ to recover the classic weak formulation of the skeleton equation.

Returning to \eqref{2_5000}, since the initial data is only converging weakly, for every $\ve,\d\in(0,1)$ we use the definition of the kinetic function and $\psi_\d$ to rewrite the first term on the righthand side as
\[ \int_\R\int_{\TT^d}\overline{\chi}(\rho^\ve_0)\psi_\d(x,\xi) = \int_{\TT^d}\rho^\ve_0(x)\eta(x)+\int_\R\int_{\TT^d}\overline{\chi}(x,\xi)(\rho^\ve_0)(\phi_\d(\xi)-1)\eta(x).\]
It follows from the definition of $\phi_\d$ and the boundedness of $\eta$, the kinetic function, and the torus that, for some $c\in(0,\infty)$ independent of $\ve,\d\in(0,1)$,
\begin{equation}\label{990_1}\abs{ \int_\R\int_{\TT^d}\overline{\chi}(\rho^\ve_0)(x,\xi)\psi_\d(x,\xi)-\int_{\TT^d}\rho^\ve_0(x)\eta(x)}\leq c\d.\end{equation}
The scaling limit $\ve K(\ve)^{d+2}\rightarrow 0$, the fact that $M_{K(\ve)}\leq cK(\ve)^{d+2}$, the fact that $N_{K(\ve)}\leq cK(\ve)^d$, the estimates of Proposition~\ref{weak_est}, and the compact support of $\psi$ prove that, $\tilde{\P}$-almost surely,
\begin{equation}\label{990_2}\lim_{\ve\rightarrow 0}\left(\sup_{t\in[0,T]}\abs{\frac{\ve N_{K(\ve)}}{2}\int_0^t\int_{\TT^d}(\Phi^\frac{1}{2})'(\tilde{\rho}^\ve)\nabla\Phi^\frac{1}{2}(\tilde{\rho}^\ve)\cdot(\nabla\psi_\d)(x,\tilde{\rho}^\ve)}\right)=0,\end{equation}
and that, $\tilde{\P}$-almost surely,
\begin{equation}\label{990_3}\lim_{\ve\rightarrow 0}\left(\sup_{t\in[0,T]}\abs{\frac{\ve M_{K(\ve)}}{2}\int_0^t\int_{\TT^d} \Phi(\tilde{\rho}^\ve)(\partial_\xi\psi)(x,\tilde{\rho}^\ve)}\right)=0.\end{equation}
For the martingale term, the Burkholder-Davis-Gundy inequality (see, for example, \cite[Chapter~4, Theorem~4.1]{RevYor1999}) and the definition of $\psi_\d$ prove that there exists $c\in(0,\infty)$ such that, for every $\ve,\d\in(0,1)$,
\[\E\left[\sup_{t\in[0,T]}\abs{\sqrt{\ve}\int_0^t\int_{\TT^d}\psi_\d(x,\tilde{\rho}^\ve)\nabla\cdot\left(\Phi^\frac{1}{2}(\tilde{\rho}^\ve)\dd\xi^{K(\ve)}\right)}^2\right]\leq c\ve M_{K(\ve)}\E\left[\int_{\TT^d}\abs{\nabla\Phi^\frac{1}{2}(\tilde{\rho}^\ve)}^2+\Phi(\tilde{\rho}^\ve)\right],\]
and therefore, along the scaling limit $\ve K(\ve)^{d+2}\rightarrow 0$, it follows from the estimates of Proposition~\ref{weak_est} that, almost surely along a subsequence $\ve\rightarrow 0$,
\begin{equation}\label{990_4}\lim_{\ve\rightarrow 0}\sup_{t\in[0,T]}\abs{\sqrt{\ve}\int_0^t\int_{\TT^d}\psi_\d(x,\tilde{\rho}^\ve)\nabla\cdot\left(\Phi^\frac{1}{2}(\tilde{\rho}^\ve)\dd\xi^{K(\ve)}\right)} = 0.\end{equation}
The remaining terms are treated using the facts that, $\tilde{\P}$-almost surely, we have the strong convergence of $\tilde{\rho}^\ve$ to $\tilde{\rho}$ in $L^1([0,T];L^1(\TT^d))$, we have $P_{K(\ve)}\tilde{g}^\ve\rightharpoonup \tilde{g}$ and $\nabla\Phi^\frac{1}{2}(\tilde{\rho}^\ve)\rightharpoonup\nabla\Phi^\frac{1}{2}(\tilde{\rho})$ weakly in $L^2(\TT^d\times[0,T];\R^d)$, and the weak convergence of $\tilde{q}^\ve$ to $\tilde{q}$ and $\tilde{p}^\ve$ to $\tilde{p}$.  This proves with \eqref{990_1}, \eqref{990_2}, \eqref{990_3}, and \eqref{990_4} that $\tilde{\P}$-almost surely, for almost every $t\in[0,T]$ after passing along a subsequence $\ve\rightarrow 0$ in \eqref{2_5000}, for the kinetic function $\tilde{\chi}$ of $\tilde{\rho}$, for some $c\in(0,\infty)$ independent of $\d\in(0,1)$,
\begin{align}\label{990_6}
& \left|\int_\R\int_{\TT^d}\tilde{\chi}(x,\xi,t)\psi_\d(x,\xi)-\int_{\TT^d}\rho_0\eta+2\int_0^t\int_{\TT^d}\Phi^\frac{1}{2}(\tilde{\rho})\nabla\Phi^\frac{1}{2}(\tilde{\rho})\cdot(\nabla\psi_\d)(x,\tilde{\rho})\right.
\\ \nonumber & \left.-\int_0^t\int_{\TT^d} \Phi^\frac{1}{2}(\tilde{\rho})\tilde{g}(x,t)\cdot (\nabla_x\psi_\d)(x,\tilde{\rho})\right|\leq \abs{\int_0^t\int_\R\int_{\TT^d} \left(\frac{2\Phi(\xi)}{\Phi'(\xi)}\abs{\phi'_\d(\xi)}\eta(x)\dd\tilde{p}+\abs{\phi'_\d(\xi)}\eta(x)\dd \tilde{q}\right)}+c\d.
\end{align}
It follows from \cite[Definition~3.4, Proposition~4.6]{FehGes21}, the definition of $\phi_\d$, the boundedness of $\eta$, and the finiteness of $\tilde{q}$ that
\begin{equation}\label{990_7} \liminf_{\d\rightarrow 0}\E\left[\int_0^t\int_\R\int_{\TT^d}\abs{\phi'_\d(\xi)}\eta(x)\dd \tilde{q}\right]=0,\end{equation}
and it follows from the finiteness of $\tilde{p}$, the definition of $\phi_\d$, the boundedness of $\eta$, and  \eqref{skel_continuity_3} that, for some $c\in(0,\infty)$ independent of $\d\in(0,1)$,
\begin{equation}\label{990_8} \liminf_{\d\rightarrow 0}\E\left[\int_0^t\int_\R\int_{\TT^d}\frac{2\Phi(\xi)}{\Phi'(\xi)}\abs{\phi'_\d(\xi)}\eta(x)\dd\tilde{p}\right]\leq \liminf_{\d\rightarrow 0}\E\left[c\int_0^t\int_{[\d,2\d]\cup[\d^{-1},2\d^{-1}]}\int_{\TT^d}\dd\tilde{p}\right] =0.\end{equation}
Finally, it follows from the definition of $\phi_\d$ that, $\tilde{\P}$-almost surely for every $t\in[0,T]$,
\[\lim_{\d\rightarrow 0}\left(\int_0^t\int_{\TT^d}\Phi^\frac{1}{2}(\tilde{\rho})\nabla\Phi^\frac{1}{2}(\tilde{\rho})\cdot(\nabla\psi_\d)(x,\tilde{\rho})\right) = \int_{\TT^d}\Phi^\frac{1}{2}(\tilde{\rho})\nabla\Phi^\frac{1}{2}(\tilde{\rho})\cdot\nabla\eta\mathbf{1}_{\{\tilde{\rho}>0\}}.\]
Since by assumption $\Phi(0)=0$ and $\Phi$ is strictly increasing, we have that $\{\tilde{\rho}=0\}=\{\Phi^\frac{1}{2}(\tilde{\rho})=0\}$, and it follows from Stampacchia's lemma (see, for example, Evans \cite[Chapter~5, Exercises~17,18]{Eva2010}) that $\nabla\Phi^\frac{1}{2}(\tilde{\rho})=0$ almost everywhere on the set $\{\Phi^\frac{1}{2}(\tilde{\rho})=0\}$.  Therefore,
\begin{equation}\label{990_9}\lim_{\d\rightarrow 0}\left(\int_0^t\int_{\TT^d}\Phi^\frac{1}{2}(\tilde{\rho})\nabla\Phi^\frac{1}{2}(\tilde{\rho})\cdot(\nabla\psi_\d)(x,\tilde{\rho})\right)= \int_{\TT^d}\Phi^\frac{1}{2}(\tilde{\rho})\nabla\Phi^\frac{1}{2}(\tilde{\rho})\cdot\nabla\eta.\end{equation}
Similarly it follows from $\{\tilde{\rho}=0\}=\{\Phi^\frac{1}{2}(\tilde{\rho})=0\}$ that, $\tilde{\P}$-almost surely for every $t\in[0,T]$,
\begin{equation}\label{990_10} \lim_{\d\rightarrow 0}\int_0^t\int_{\TT^d} \Phi^\frac{1}{2}(\tilde{\rho})\tilde{g}(x,t)\cdot (\nabla_x\psi_\d)(x,\tilde{\rho}) = \int_0^t\int_{\TT^d}\Phi^\frac{1}{2}(\tilde{\rho})\tilde{g}(x,t)\cdot\nabla\eta,\end{equation}
and from the definition of $\phi_\d$ and the kinetic function that, $\tilde{\P}$-almost surely for every $t\in[0,T]$,
\begin{equation}\label{990_11}\lim_{\d\rightarrow 0}\int_\R\int_{\TT^d}\tilde{\chi}(x,\xi,t)\psi_\d(x,\xi) = \int_{\TT^d}\tilde{\rho}(x,t)\eta(x).\end{equation}
Returning to \eqref{990_6}, it follows from \eqref{990_7}, \eqref{990_8}, \eqref{990_9}, \eqref{990_10}, and \eqref{990_11} that, after passing to a subsequence $\d\rightarrow 0$, $\tilde{\P}$-almost surely for almost every $t\in[0,T]$,
\[\int_{\TT^d}\rho(x,t)\eta(x) = \int_{\TT^d}\rho_0\eta-2\int_0^t\int_{\TT^d}\Phi^\frac{1}{2}(\tilde{\rho})\nabla\Phi^\frac{1}{2}(\tilde{\rho})\cdot\nabla\eta +\int_0^t\int_{\TT^d} \Phi^\frac{1}{2}(\tilde{\rho})\tilde{g}\cdot \nabla \eta.\]
We conclude that $\tilde{\rho}$ is $\tilde{\P}$-almost surely a weak solution of the skeleton equation with control $\tilde{g}$ in the sense of Definition~\ref{classical_weak}, which by Theorem~\ref{thm_unique} and Theorem~\ref{equiv} is unique.  Since the laws of $\tilde{\rho}^\ve$ and $\rho^\ve$ are the same, this completes the proof for bounded initial data in $L^2(\TT^d)$.  The case that the $\rho^\ve_0$ and $\rho_0$ are bounded in $\Ent_\Phi(\TT^d)$ then follows by approximating the $\rho^\ve_0$ by $(\rho^\ve_0\wedge n)\in L^2(\TT^d)$, the almost sure $L^1$-contraction property of \cite[Theorem~4.7]{FehGes21}, the $L^1$-contraction property of Theorem~\ref{thm_unique}, and the triangle inequality.  This completes the proof. \end{proof}

\begin{definition}\label{def_uniform_Laplace} Let $(\O,\F,\P)$ be a probability space, let $Y_0$ and $Y$ be a Polish spaces, let $\dd_Y\colon Y^2\rightarrow[0,\infty)$ be the metric on $Y$, let $\{X^\ve_{y_0}\colon\O\rightarrow Y\}_{\ve\in(0,1),y_0\in Y_0}$ be random variables, and let $\{I_{y_0}\colon Y\rightarrow [0,\infty]\}_{y_0\in Y}$ be good rate functions with compact level sets on compact sets in the sense that, for every compact set $K\subseteq Y_0$ and $M\in(0,\infty)$,
\[\left(\cup_{y_0\in K}\{y\in Y\colon I_{y_0}(y)\leq M\}\right)\subseteq Y\;\;\textrm{is compact.}\]
We say that the random variables $\{X^\ve_{y_0}\colon\O\rightarrow Y\}_{\ve\in(0,1),y_0\in Y_0}$ satisfy a uniform large deviations principle with respect to compact subsets of $Y_0$, with good rate functions $\{I_{y_0}\}_{y_0\in Y}$ with compact level sets on compact sets, if the following two conditions are satisfied.
\begin{enumerate}
\item Lower bound:  for every compact set $K\subseteq Y_0$ and $\delta,M\in(0,\infty)$,
\[\liminf_{\ve\rightarrow 0}\left(\inf_{y_0\in K}\left[\inf_{y\in\{x\in Y\colon I_{y_0}(x)\leq M\}}\left(\ve\log(\P\left[\dd_Y(X^\ve_{y_0},y)<\delta\right] +I_{y_0}(y) \right)\right]\right)\geq 0.\]
\item Upper bound:  for every compact set $K\subseteq Y_0$ and $\delta,M\in(0,\infty)$,
\[\limsup_{\ve\rightarrow 0}\left(\sup_{y_0\in K}\left[\sup_{M'\leq M}\left(\ve\log\left(\P\left[\inf_{y\in\{x\in Y\colon I_{y_0}(x)\leq M'\}} d_Y(X^\ve,y)\geq \delta\right]\right) +M' \right)\right]\right)\leq 0.\]
\end{enumerate}
\end{definition}

\begin{thm}\label{prop_collapse_3}  Let $\Phi\in\C([0,\infty))\cap\C^1_{\textrm{loc}}((0,\infty))$ satisfy Assumptions~\ref{as_unique},  \ref{as_equiv}, and \ref{as_compact} and let $\{K(\ve)\}_{\ve\in(0,1)}$ satisfy, as $\ve\rightarrow 0$,
\[\ve K(\ve)^{d+2}\rightarrow 0\;\;\textrm{and}\;\;K(\ve)\rightarrow\infty.\]
Then, the rate functions $\{I_{\rho_0}\}_{\rho_0\in\Ent_\Phi(\TT^d)}$  defined in \eqref{collapse_ldp_rate} are good rate functions with compact level sets on compact sets.  For every $\rho_0\in\Ent_\Phi(\TT^d)$ the solutions $\{\rho^\ve(\rho_0)\}_{\ve\in(0,1)}$ of \eqref{collapse_eq} satisfy a large deviations principle with rate function $I_{\rho_0}$ on $L^1([0,T];L^1(\TT^d))$.  Furthermore, for every $R\in(0,\infty)$, the solutions satisfy a uniform large deviations principle on $L^1([0,T];L^1(\TT^d))$ with respect to weakly $L^1(\TT^d)$-compact subsets of $\Ent_{\Phi,R}(\TT^d)$.
\end{thm}

\begin{proof}  The proof relies on an application of the weak approach to large deviations \cite[Theorem~6]{BudDupMar2008}, the equivalence of uniform Laplace and large deviations principles with respect to compact subsets of the initial data \cite[Theorem~4.3]{BudDupSal}, Lemma~\ref{as_pre}, Proposition~\ref{frac_est}, and Theorem~\ref{weak_thm_con}.  To apply the framework of \cite[Theorem~6]{BudDupMar2008}, it is necessary to prove the compactness of the solution set of the skeleton equation with uniformly bounded initial data and uniformly bounded controls, and to prove the convergence in distribution of the controlled SPDE \eqref{entropy_con_eq} to the skeleton equation.  The compactness is an immediate consequence of Lemma~\ref{as_pre} and Proposition~\ref{frac_est}, and the required convergence of the controlled SPDE to the skeleton equation is exactly Theorem~\ref{weak_thm_con}.  This completes the proof.  \end{proof}

\begin{remark}\label{smooth_white_noise}  We emphasize that these techniques apply for general noise of the form $\xi^\d = \sum_{k=1}^\infty f^\d_k\dd B^k_t$ provided that, for every $\d\in(0,1)$, the sums $\sum_{k=1}^\infty (f^\d_k)^2$ and $\sum_{k=1}^\infty\abs{\nabla f^\d_k}^2$ are continuous on $\TT^d$ and provided that for every $\d\in(0,1)$ the noise is probabilistically stationary in the sense that the quadratic variation $\sum_{k=1}^\infty (f^\d_k)^2$ is constant on $\TT^d$.  In particular, these techniques apply without any essential changes to spatial smoothings $\xi^\d = (\xi*\eta^\d)$ of space-time white noise for a standard convolution kernel $\eta^\d$ of scale $\d\in(0,1)$ on $\TT^d$.  In this case, these methods prove that the solutions satisfy the uniform large deviations principle in the scaling regime $\ve\delta(\ve)^{-(d+2)}\rightarrow 0$, in exact analogy with Theorem~\ref{prop_collapse_3}.
\end{remark}

\section{$\Gamma$-convergence of rate functions}\label{sec_gamma}  In this section, we will first show that, for every $K\in\N$, the solutions of the equation
\begin{equation}\label{fd_eq} \partial_t\rho^\ve = \Delta\Phi(\rho^\ve)-\sqrt{\ve}\nabla\cdot\left(\Phi^\frac{1}{2}(\rho^\ve)\circ\xi^K\right)\;\; \textrm{in}\;\;\TT^d\times(0,\infty)\;\;\textrm{with}\;\;\rho^\ve(\cdot,0)=\rho_0, \end{equation}
satisfy a small-noise large deviations principle with rate function
\begin{equation}\label{collapse_ldp_rate_rate}I^K_{\rho_0}(\rho)=\frac{1}{2}\inf\left\{\norm{g}^2_{L^2([0,T];L^2(\TT^d))}\colon \partial_t\rho=\Delta\Phi(\rho)-\nabla\cdot \left(\Phi^\frac{1}{2}(\rho)P_Kg\right)\;\;\textrm{and}\;\;\rho(\cdot,0)=\rho_0\right\},\end{equation}
for $P_Kg$ the Fourier projection of $g\in L^2([0,T];L^2(\TT^d))$ onto the span of $\{\sin(k\cdot x),\cos(k\cdot x)\}_{\{\abs{k}\leq K\}}$.  Then, in Lemma~\ref{rf_min} and Theorem~\ref{thm:gamma-convergence} we prove that the rate functions \eqref{collapse_ldp_rate_rate} converge as $K\rightarrow\infty$ in the sense of $\Gamma$-convergence to the rate function \eqref{collapse_ldp_rate}.

\begin{prop}  Let $K\in\N$ and let $\Phi\in\C([0,\infty))\cap\C^1_{\textrm{loc}}((0,\infty))$ satisfy Assumptions~\ref{as_unique}, ~\ref{as_equiv}, and \ref{as_compact}, and let $\rho_0\in \Ent_\Phi(\TT^d)$.  Then, the rate functions $\{I^K_{\rho_0}\}_{\rho_0\in\Ent_\Phi(\TT^d)}$  defined in \eqref{collapse_ldp_rate_rate} are good rate functions with compact level sets on compact sets.  For every $\rho_0\in\Ent_\Phi(\TT^d)$ the solutions $\{\rho^{\ve,K}(\rho_0)\}_{\ve\in(0,1)}$ of \eqref{fd_eq} satisfy a large deviations principle with rate function $I^K_{\rho_0}$ on $L^1([0,T];L^1(\TT^d))$.  Furthermore, for every $R\in(0,\infty)$, the solutions satisfy a uniform large deviations principle on $L^1([0,T];L^1(\TT^d))$ with respect to weakly $L^1(\TT^d)$-compact subsets of $\Ent_{\Phi,R}(\TT^d)$.  \end{prop}

\begin{proof}  The proof is identical to the proof of Theorem~\ref{prop_collapse_3}, and is obtained by repeating the same argument in the scaling regime that $K\in\N$ is fixed and $\ve\rightarrow 0$.\end{proof}

\begin{lem}\label{rf_min}  Let $\Phi\in\C([0,\infty))\cap\C^1_{\textrm{loc}}((0,\infty))$ satisfy Assumptions~\ref{as_unique}, ~\ref{as_equiv}, and \ref{as_compact}, let $\rho_0\in\Ent_\Phi(\TT^d)$, let $\{I^K_{\rho_0}\}_{K\in\N}$ be defined in \eqref{collapse_ldp_rate_rate}, and let $I_{\rho_0}$ be defined in \eqref{collapse_ldp_rate}.  Then, for every $K\in\N$, the infimums appearing in \eqref{collapse_ldp_rate} and \eqref{collapse_ldp_rate_rate} are attained. Furthermore, for every $\rho\in L^1([0,T];L^1(\TT^d))$ satisfying $I_{\rho_0}(\rho)<\infty$ (alternately, $I^K_{\rho_0}(\rho)<\infty$ for $K\in\N$), there exists a unique $g\in L^2([0,T];L^2(\TT^d))$ satisfying
\[I_{\rho_0}(\rho) = \frac{1}{2}\norm{g}^2_{L^2([0,T];L^2(\TT^d;\R^d))}\;\;\left(\textrm{alternately,}\;\;I^K_{\rho_0}(\rho)=\frac{1}{2}\norm{g}^2_{L^2([0,T];L^2(\TT^d;\R^d))}\right).\]
\end{lem}

\begin{proof} The face that the infimums appearing in \eqref{collapse_ldp_rate_rate} and \eqref{collapse_ldp_rate} are achieved is an immediate consequence of Theorem~\ref{equiv}, Lemma~\ref{as_pre}, and Proposition~\ref{frac_est}, since Definition~\ref{classical_weak} is stable with respect to weak convergence of the control.  For the second statement, suppose that $\rho\in L^1([0,T];L^1(\TT^d))$ satisfies $I_{\rho_0}(\rho)<\infty$ (alternately, $I^K_{\rho_0}(\rho)<\infty$ for $K\in\N$).  Then, by the above, the set
\[\left\{g\in L^2([0,T];L^2(\TT^d;\R^d))\colon \partial_t\rho = \Delta(\Phi(\rho))-\nabla\cdot(\Phi^\frac{1}{2}(\rho)g)\right\},\]
is a non-empty, weaky closed, convex subset of $L^2([0,T];L^2(\TT^d;\R^d))$.  The uniqueness now follows from the uniform convexity of the $L^2$-norm, which completes the proof.  \end{proof}

\begin{thm}\label{thm:gamma-convergence}  Let $\Phi\in\C([0,\infty))\cap\C^1_{\textrm{loc}}((0,\infty))$ satisfy Assumptions~\ref{as_unique}, ~\ref{as_equiv}, and \ref{as_compact}, let $\rho_0\in\Ent_\Phi(\TT^d)$, let $\{I^K_{\rho_0}\}_{K\in\N}$ be defined in \eqref{collapse_ldp_rate_rate}, and let $I_{\rho_0}$ be defined in \eqref{collapse_ldp_rate}.  Then, as $K\rightarrow\infty$,
\[I^K_{\rho_0}\overset{\Gamma}{\longrightarrow}I_{\rho_0}.\]
\end{thm}

\begin{proof}  Let $\rho\in L^1([0,T];L^1(\TT^d))$.  In order to establish the $\Gamma$-convergence, it is necessary to prove the following two properties.
\begin{enumerate}[(i)]
\item  For every sequence $\{\rho_{K_n}\}_{n\in\N}\subseteq L^1([0,T];L^1(\TT^d))$ satisfying as $n\rightarrow\infty$ that $K_n\rightarrow \infty$ and
\[\rho_{K_n}\rightarrow \rho\;\;\textrm{strongly in}\;\;L^1([0,T];L^1(\TT^d)),\]
we have that
\begin{equation}\label{gc_1}\liminf_{n\rightarrow\infty}I^{K_n}_{\rho_0}(\rho_{K_n})\geq I_{\rho_0}(\rho).\end{equation}
\item  There exists a sequence $\{\rho_{K_n}\}_{n\in N}\subseteq L^1([0,T];L^1(\TT^d))$ satisfying as $n\rightarrow \infty$ that $K_n\rightarrow \infty$ and that
\[\rho_{K_n}\rightarrow \rho\;\;\textrm{strongly in}\;\;L^1([0,T];L^1(\TT^d)),\]
such that
\begin{equation}\label{gc_2}\limsup_{n\rightarrow\infty}I^{K_n}_{\rho_0}(\rho_{K_n})\leq I_{\rho_0}(\rho).\end{equation}
\end{enumerate}

\emph{Proof of \eqref{gc_1}.}  Let $\rho\in L^1([0,T];L^1(\TT^d))$ and suppose that $\{\rho_{K_n}\}_{n\in N}\subseteq L^1([0,T];L^1(\TT^d))$ satisfies as $n\rightarrow\infty$ that $K_n\rightarrow \infty$ and that
\[\rho_{K_n}\rightarrow \rho\;\;\textrm{strongly in}\;\;L^1([0,T];L^1(\TT^d)).\]
If $\liminf_{n\rightarrow\infty}I^{K_n}(\rho_{K_n})=\infty$ then \eqref{gc_1} is satisfied.  If not, fix a subsequence $\{K_m\}_{m\in\N}\subseteq \{K_n\}_{n\in\N}$ which satisfies for every $m\in\N$ that
\[I^{K_m}_{\rho_0}(\rho_{K_m})<\infty,\]
and that
\begin{equation}\label{gc_4}\lim_{m\rightarrow\infty} I^{K_m}_{\rho_0}(\rho_{K_m})=\liminf_{n\rightarrow\infty}I^{K_n}(\rho_{K_n}).\end{equation}
By Lemma~\ref{rf_min}, for every $m\in\N$ fix the unique $g_m\in L^2(\TT^d\times[0,T];\R^d)$ satisfying
\begin{equation}\label{gc_5} \frac{1}{2}\norm{g_m}^2_{L^2([0,T];L^2(\TT^d;\R^d))}=I^{K_m}_{\rho_0}(\rho_{K_m}).\end{equation}
It follows from \eqref{gc_4} and \eqref{gc_5} that there exists $g\in L^2([0,T];L^2(\TT^d;\R^d))$ such that, after passing to a further subsequence $\{K_{m'}\}_{m'\in\N}\subseteq \{K_m\}_{m\in\N}$, as $m'\rightarrow\infty$,
\[g_{m'}\rightharpoonup g\;\;\textrm{weakly in}\;\;L^2(\TT^d\times[0,T];\R^d).\]
Since by assumption, as $m'\rightarrow\infty$,
\[\rho_{K_{m'}}\rightarrow \rho\;\;\textrm{strongly in}\;\;L^1([0,T];L^1(\TT^d)),\]
it follows from Theorem~\ref{equiv}, Theorem~\ref{gen_exist}, and Proposition~\ref{weak_strong} that $\rho$ is a renormalized kinetic solution of \eqref{skeleton_equation} in the sense of Definition~\ref{skel_sol_def}.  The weak lower-semicontinuity of the $L^2$-norm, the definition of $I_{\rho_0}$, and \eqref{gc_4} prove that
\[\liminf_{n\rightarrow\infty}I^{K_n}_{\rho_0}(\rho_{K_n})\geq \frac{1}{2}\norm{g}_{L^2(\TT^d\times[0,T];\R^d)}^2\geq I_{\rho_0}(\rho),\]
which completes the proof of \eqref{gc_1}.

\emph{The proof of \eqref{gc_2}.}  Let $\rho\in L^1([0,T];L^1(\TT^d))$.  If $I_{\rho_0}(\rho)=\infty,$ then \eqref{gc_2} is satisfied.  If not, Lemma~\ref{rf_min} implies that there exists a unique $g\in L^2(\TT^d\times[0,T];\R^d)$ such that
\begin{equation}\label{gc_12} \frac{1}{2}\norm{g}^2_{L^2(\TT^d\times[0,T];\R^d)}=I_{\rho_0}(\rho).\end{equation}
For every $K\in\N$ let $\rho_K\in L^1([0,T];L^1(\TT^d))$ denote the unique solution of the equation
\[\partial_t \rho_K = \Delta\Phi(\rho_K)-\nabla\cdot\left(\Phi^{\frac{1}{2}}(\rho_K)P_Kg\right)\;\; \textrm{in}\;\;\TT^d\times(0,T)\;\;\textrm{with}\;\;\rho_K(\cdot,0)=\rho_0.\]
Since, as $K\rightarrow 0$,
\begin{equation}\label{gc_14} P_Kg\rightarrow g\;\;\textrm{strongly in}\;\;L^2(\TT^d\times[0,T];\R^d),\end{equation}
Proposition~\ref{weak_strong} proves that there exists $\tilde{\rho}\in L^1([0,T];L^1(\TT^d))$ such that, as $K\rightarrow\infty$,
\begin{equation}\label{gc_015}\rho_K\rightarrow\tilde{\rho}\;\;\textrm{strongly in}\;\;L^1([0,T];L^1(\TT^d)),\end{equation}
and such that $\tilde{\rho}$ solves \eqref{skeleton_equation} in the sense of Definition~\ref{skel_sol_def}.  Since Theorem~\ref{thm_unique} proves that the solution is unique,
\begin{equation}\label{gc_15} \tilde{\rho}=\rho\;\;\textrm{in}\;\;L^1([0,T];L^1(\TT^d)).\end{equation}
Therefore, it follows from \eqref{gc_015} and \eqref{gc_15} that, as $K\rightarrow\infty$,
\begin{equation}\label{gc_16} \rho_K\rightarrow \rho\;\;\textrm{strongly in}\;\;L^1([0,T];L^1(\TT^d)),\end{equation}
and the definition of the $\{I^K_{\rho_0}\}_{K\in\N}$, \eqref{gc_12}, and \eqref{gc_14} prove that
\begin{equation}\label{gc_17}\limsup_{K\rightarrow\infty}I^K_{\rho_0}(\rho_K)\leq\limsup_{K\rightarrow\infty}\frac{1}{2}\norm{P_Kg}^2_{L^2(\TT^d\times[0,T];\R^d)}=\frac{1}{2}\norm{g}^2_{L^2(\TT^d\times[0,T];\R^d)}=I_{\rho_0}(\rho).\end{equation}
In combination \eqref{gc_16} and \eqref{gc_17} complete the proof of \eqref{gc_2}.  This completes the proof.  \end{proof}

\section{The lower semicontinuous envelope of the restricted rate function}\label{sec_lsc_envelope}

As outlined in the introduction, the general approach to large deviations in interacting particle systems introduced in \cite{DV89,KOV89} relies on a separate derivation of large deviations lower and upper bounds. As a consequence of these arguments, both estimates lead to possibly different rate functions. The proof of a large deviations principle relies on establishing their identity, which proves to be a challenging problem. In particular, this requires the characterization of the l.s.c.\ envelope of the lower-bound rate function restricted to smooth fluctuations, which in the case of the zero range process has remained an open problem since \cite{BKL95}. The proofs of these facts is the main point of this section.

More precisely, let $I^0\colon L^1([0,T];L^1(\TT^d))\rightarrow[0,\infty]$ be the rate function defined by
\small
\begin{align}\label{lsc_0111}
 I^0(\rho) :=  \sup_{\psi\in\C^\infty_c(\TT^d\times[0,T))} & \left(-\int_0^T\int_{\TT^d}\Phi(\rho)\Delta\psi- \int_0^T\int_{\TT^d}\rho\partial_t\psi-\int_{\TT^d}\psi(x,0)\rho(x,0)-\frac{1}{2}\int_0^T\int_{\TT^d}\Phi(\rho)\abs{\nabla\psi}^2\right),
\end{align}
\normalsize
if $\rho\in \C([0,T];L^1(\TT^d))$ with $\rho(\cdot,0)\in\Ent_\Phi(\TT^d)$, and $I^0(\rho)=\infty$ otherwise.  It has been demonstrated in \cite[Theorem 1 and p.99]{BKL95} that, for the space $\mcS$ of smooth fluctuations defined in Definition~\ref{smooth_fluctuation} below, the application of the approach of \cite{DV89,KOV89} to the zero range process yields a large deviations lower bound with rate function $I^{\textrm{lo}}\colon L^1([0,T];L^1(\TT^d))\rightarrow[0,\infty]$ defined by
\begin{equation}\label{lsc_0112}  I^{\textrm{lo}}(\rho):= \overline{(I^0)_{|\mcS}}(\rho).   \end{equation}
That is, $I^{\textrm{lo}}$ is the l.s.c.\ envelope of $I^0$ restricted to smooth fluctuations $\rho \in \mcS$.

In contrast, the results of \cite[Theorem 1]{BKL95}, together with subsequent contributions \cite{Qu.Ya1998,Qu.Re.Va1999,La.Ts2018} in related settings, lead to the large deviations upper bound with rate function $I^{\textrm{up}}\colon L^1([0,T];L^1(\TT^d))\rightarrow[0,\infty]$ defined by
\begin{equation}\label{lsc_113} I^{\textrm{up}}(\rho):= I^0(\rho), \end{equation}
for $\rho\in \C([0,T];L^1(\TT^d))$ with finite entropy-dissipation.  That is, $I^{\textrm{up}}(\rho)=\infty$ if $\rho\notin \C([0,T];L^1(\TT^d))$ or if $\nabla\Phi^\frac{1}{2}(\rho)\not\in L^2([0,T];L^2(\TT^d;\R^d))$.

An essential difficulty and open problem since \cite{BKL95} is to identify the resulting rate functions $I^{\textrm{lo}}$ and  $I^{\textrm{up}}$, the key difficulty being the characterization of the l.s.c.\ envelope appearing in \eqref{lsc_0112}. This problem is solved in the present section, by Theorem~\ref{thm:envelope} below.

An important realization is that in the the proof of Theorem~\ref{thm:envelope} it is instrumental to work with an alternate characterization of the rate function. Precisely, we introduce the rate function
\begin{equation}\label{lsc_01111}
I(\rho)=\frac{1}{2}\inf\left\{\norm{g}^2_{L^2(\TT^d\times[0,T];\R^d)}\colon \partial_t\rho = \Delta\Phi(\rho)-\nabla\cdot(\Phi^\frac{1}{2}(\rho)g)\right\}
\end{equation}
if $\rho\in\C([0,T];L^1(\TT^d))$ and $\rho(\cdot,0)\in\Ent_{\Phi}(\TT^d)$, and we define $I(\rho)=\infty$ otherwise.  In particular, since the equation in \eqref{lsc_01111} is interpreted in the sense of Definition~\ref{classical_weak} with initial data $\rho(\cdot,0)$, this means that $I(\rho)=\infty$ if $\rho$ fails to satisfy the entropy dissipation estimate
\begin{equation}\label{lsc_0011}\Phi^\frac{1}{2}(\rho)\in L^2([0,T];H^1(\TT^d)),\end{equation}
which using \cite[Lemma~5.4]{FehGes21} and Assumption~\ref{as_phi_reg} is equivalent to $\nabla\Phi^\frac{1}{2}(\rho)\in L^2([0,T];L^2(\TT^d;\R^d)).$ 
Furthermore, we recall by Proposition~\ref{L1-continuity} that solutions of the skeleton equation are always $L^1(\TT^d)$-continuous in time, where by convention we always identify $\rho\in L^1([0,T];L^1(\TT^d))$ with its unique strongly $L^1(\TT^d)$-continuous representative when it exists.

We show that on the space of functions $\rho\in \C([0,T];L^1(\TT^d))$ satisfying the entropy dissipation estimate \eqref{lsc_0011}, we have for the rate functions \eqref{lsc_0111}, \eqref{lsc_113}, and \eqref{lsc_01111} that
\begin{align}\label{lsc_11}
I(\rho) &= I^0(\rho) = I^{\textrm{up}}(\rho) \\&= \frac{1}{2}\sup_{\{\norm{\psi}_{H^1_{\Phi(\rho)}}\leq 1\}}\left(-\int_0^T\int_{\TT^d}\Phi(\rho)\Delta\psi- \int_0^T\int_{\TT^d}\rho\partial_t\psi-\int_{\TT^d}\psi(x,0)\rho(x,0)\right)^2,\nonumber
\end{align}
where the supremum is taken over smooth functions $\psi\in\C^\infty_c(\TT^d\times[0,T))$, and the space  $H^1_{\Phi(\rho)}$ is defined in Definition~\ref{weighted_H1} below.  We show the equivalence of the final three expressions appearing in \eqref{lsc_11} in Lemma~\ref{dual_equiv}, and we prove that $I=I^{\textrm{up}}$ in Proposition~\ref{prop_rate_equivalent}.

\begin{definition}\label{weighted_H1} Let $\Phi\in\C([0,\infty))\cap\C^1_{\textrm{loc}}((0,\infty))$ and let $\rho\in L^\infty([0,T];L^1(\TT^d))$ satisfy \eqref{lsc_0011}.   We first introduce the equivalence relation $\sim$ on $\C^\infty(\TT^d\times[0,T])$ by
\[\psi\sim\phi\;\;\textrm{if and only if}\;\;\int_0^T\int_{\TT^d}\Phi(\rho)\abs{\nabla\psi-\nabla\phi}^2=0.\]
Let $H^1_{\Phi(\rho)}$ be the Hilbert space defined by the completion of the set of equivalence classes $\C^\infty(\TT^d\times[0,T])/\sim$ with respect to the positive definite inner product, for every $\phi,\psi\in \C^\infty(\TT^d\times[0,T])/\sim$,
\[\langle \phi,\psi \rangle_{H^1_{\Phi(\rho)}}=\int_0^T\int_{\TT^d}\Phi(\rho)\nabla\phi\cdot \nabla\psi\;\;\textrm{and}\;\;\norm{\phi}_{H^1_{\Phi(\rho)}}=\langle \phi,\phi\rangle_{H^1_{\Phi(\rho)}}^\frac{1}{2}.\]
\end{definition}

\begin{remark}  In the setting of Definition~\ref{weighted_H1}, let $L^2_{\Phi}(\TT^d\times[0,T];\R^d)$ denote the Sobolev space of measurable vector fields $v\colon\TT^d\times[0,T]\rightarrow\R^d$ satisfying
\[\int_0^T\int_{\TT^d}\Phi(\rho)\abs{v}^2<\infty,\]
which is the Hilbert space of equivalence classes of vector fields with respect to the inner product
\[[v,w]_{\Phi(\rho)}=\int_0^T\int_{\TT^d}\Phi(\rho)v\cdot w\;\;\textrm{for every}\;\;v,w\in L^2_{\Phi(\rho)}(\TT^d\times[0,T];\R^d).\]
It is an immediate consequence of Definition~\ref{weighted_H1} that the Hilbert space $H^1_{\Phi(\rho)}$ is canonically isomorphic to a closed subspace of $L^2_{\Phi}(\TT^d\times[0,T];\R^d)$.  For every $H\in H^1_{\Phi(\rho)}$ we will write $\nabla H$ for the vector field identified by this isomorphism, where for $H\in L^2([0,T];H^1(\TT^d))$ the generalized gradient $\nabla H$ is equal to the standard weak gradient of $H$.
\end{remark}

\begin{lem}\label{dual_equiv}  Let $\Phi\in\C([0,\infty))\cap\C^1_{\textrm{loc}}((0,\infty))$ and let $\rho\in \C([0,T];L^1(\TT^d))$ satisfy \eqref{lsc_0011}.  Then,
\begin{equation}\label{dual_equiv_1} I^{\textrm{\emph{up}}}(\rho)=\frac{1}{2}\sup_{\{\norm{\psi}_{H^1_{\Phi(\rho)}}\leq 1\}}\left(-\int_0^T\int_{\TT^d}\Phi(\rho)\Delta\psi- \int_0^T\int_{\TT^d}\rho\partial_t\psi-\int_{\TT^d}\psi(x,0)\rho(x,0)\right)^2, \end{equation}
where the supremum is taken over $\psi\in\C^\infty_c(\TT^d\times[0,T))$.
\end{lem}

\begin{proof}  It is a consequence of the definitions that the lefthand side of \eqref{dual_equiv_1} is infinite if and only if the righthand side is infinite.  In the case that both sides of \eqref{dual_equiv_1} are finite, it follows from the definitions that the linear functional
\[\psi\in\C^\infty_c(\TT^d\times[0,T))\rightarrow -\int_0^T\int_{\TT^d}\Phi(\rho)\Delta\psi- \int_0^T\int_{\TT^d}\rho\partial_t\psi-\int_{\TT^d}\psi(x,0)\rho(x,0),\]
extends by density to a unique continuous linear function on $H^1_{\Phi(\rho)}$.  The Riesz Representation Theorem proves that there exists $H_\rho\in H^1_{\Phi(\rho)}$ such that, for every $\psi\in\C^\infty_c(\TT^d\times(0,T))$,
\[-\int_0^T\int_{\TT^d}\Phi(\rho)\Delta\psi- \int_0^T\int_{\TT^d}\rho\partial_t\psi=\int_0^T\int_{\TT^d}\Phi(\rho)\nabla H_\rho\cdot\nabla\psi,\]
from which it follows by density of smooth functions in $H^1_{\Phi(\rho)}$ and the definitions that both sides of \eqref{dual_equiv_1} are equal to
\begin{equation}\label{lsc_20}\frac{1}{2}\int_0^T\int_{\TT^d}\Phi(\rho)\abs{\nabla H_\rho}^2=\frac{1}{2}\norm{H_\rho}^2_{H^1_{\Phi(\rho)}},\end{equation}
which completes the proof.  \end{proof}

\begin{prop}\label{prop_rate_equivalent}  Let $\Phi\in\C([0,\infty))\cap\C^1_{\textrm{loc}}((0,\infty))$ and let $\rho\in \C([0,T];L^1(\TT^d))$.  Then,
\[I(\rho) = I^{\textrm{\emph{up}}}(\rho),\]
for the rate functions defined in \eqref{lsc_113} and \eqref{lsc_01111}.
\end{prop}

\begin{proof}  If $\rho$ fails to  satisfy \eqref{lsc_0011} then both rate functions are infinite by definition.  It remains only to consider the case that $\rho$ satisfies \eqref{lsc_0011}.  We will first show that if
\begin{equation}\label{lsc_21}I^{\textrm{up}}(\rho)=\infty,\end{equation}
then $I(\rho)=\infty$.  Proceeding by contradiction, if $I(\rho)<\infty$ then there exists a control $g\in L^2([0,T];L^2(\TT^d;\R^d))$ such that
\begin{equation}\label{lsc_skel}\partial_t\rho = \Delta\Phi(\rho)-\nabla\cdot(\Phi^\frac{1}{2}(\rho)g)\;\;\textrm{in}\;\;\TT^d\times(0,T),\end{equation}
in the sense of Definition~\ref{classical_weak}.  It follows from an approximation argument and the Lebesgue differentiation theorem that $\rho$ satisfies \eqref{lsc_skel} in the sense of Definition~\ref{classical_weak} if and only if $\Phi^\frac{1}{2}(\rho)\in L^2([0,T];H^1(\TT^d))$ and, for every $\psi\in\C^\infty_c(\TT^d\times[0,T))$,
\[2\int_0^T\int_{\TT^d}\nabla\Phi^\frac{1}{2}(\rho)\cdot\Phi^\frac{1}{2}(\rho)\nabla\psi -\int_0^T\int_{\TT^d}\rho\partial_t\psi-\int_{\TT^d}\rho_0(x)\psi(x,0) = \int_0^T\int_{\TT^d}g\cdot\Phi^\frac{1}{2}(\rho)\nabla\psi,\]
which implies using \eqref{lsc_0011}, H\"older's inequality, and Young's inequality that
\[I^{\textrm{up}}(\rho)=\sup_{\psi\in\C^\infty_c(\TT^d\times[0,T))}\left(\int_0^T\int_{\TT^d}g\cdot \Phi^\frac{1}{2}(\rho)\nabla\psi-\frac{1}{2}\int_0^T\int_{\TT^d}\Phi(\rho)\abs{\nabla\psi}^2\right)\leq \frac{1}{2}\norm{g}^2_{L^2}<\infty,\]
contradicting \eqref{lsc_21}.  It remains to consider the case that \eqref{lsc_21} is finite.  A repetition of the argument appearing in Lemma~\ref{dual_equiv} and \eqref{lsc_20} proves that, in this case, there exists $H_\rho\in H^1_{\Phi(\rho)}$ such that
\begin{equation}\label{lsc_2020} I^{\textrm{up}}(\rho)=\frac{1}{2}\norm{H_\rho}^2_{H^1_{\Phi(\rho)}}, \end{equation}
and such that
\[\partial_t\rho = \Delta\Phi(\rho)-\nabla\cdot(\Phi(\rho)\nabla H_\rho)\;\;\textrm{in}\;\;\TT^d\times(0,T).\]
This proves that $I(\rho)\leq \frac{1}{2}\norm{\Phi^\frac{1}{2}(\rho)\nabla H_\rho}^2_{L^2}=\frac{1}{2}\norm{H_\rho}^2_{H^1_{\Phi(\rho)}}$.  To prove the reverse inequality, consider the control set
\[\textrm{Cont}_\rho = \left\{g\in L^2(\TT^d\times[0,T];\R^d)\colon \partial_t\rho = \Delta\Phi(\rho)-\nabla\cdot(\Phi^\frac{1}{2}(\rho)g)\right\},\]
which is equivalent to the set
\[\textrm{Cont}_\rho=\left\{g\in L^2(\TT^d\times[0,T];\R^d)\colon -\nabla\cdot(\Phi(\rho)\nabla H_\rho)=-\nabla\cdot(\Phi^\frac{1}{2}(\rho)g)\;\;\textrm{in}\;\;H^{-1}_{\Phi(\rho)}\right\},\]
for the dual space $H^{-1}_{\Phi(\rho)}$ of $H^1_{\Phi(\rho)}$.  Here, the action of $-\nabla\cdot(\Phi^\frac{1}{2}(\rho)g)$ on $H^1_{\Phi(\rho)}$ is defined by
\[[-\nabla\cdot(\Phi^\frac{1}{2}(\rho)g)](H) = \int_0^T\int_{\TT^d}g\cdot\Phi^\frac{1}{2}(\rho)\nabla H\;\;\textrm{for every}\;\; H \in H^1_{\Phi(\rho)}.\]
Since it follows by definition that $\Phi^\frac{1}{2}(\rho)\nabla H_\rho$ is an admissible control, we have for every $g\in \textrm{Cont}_\rho$ that
\[\int_0^T\int_{\TT^d}\Phi(\rho)\abs{\nabla H_\rho}^2 = \int_0^T\int_{\TT^d}g\cdot \Phi^\frac{1}{2}(\rho)\nabla H_\rho.\]
It then follows from H\"older's inequality that, for every $g\in \textrm{Cont}_\rho$,
\[\norm{H_\rho}_{H^1_{\Phi(\rho)}}\leq \norm{g}_{L^2(\TT^d\times[0,T];\R^d)},\]
and therefore that $I(\rho) = \frac{1}{2}\norm{H_\rho}^2_{H^1_{\Phi(\rho)}}$.  Together with \eqref{lsc_2020} this completes the proof. \end{proof}

\begin{definition}\label{smooth_fluctuation}  Let $\Phi\in\C([0,\infty))\cap\C^1_{\textrm{loc}}((0,\infty))$ satisfy Assumptions~\ref{as_unique}, ~\ref{as_equiv}, and \ref{as_compact}.  Let $\mathcal{S}\subseteq \C([0,T];L^1(\TT^d))$ denote the subset of all strictly positive and bounded $\rho$ such that there exist a strictly positive $\rho_0\in \C^\infty(\TT^d)$ and an $H\in \C^{3,1}(\TT^d\times[0,T])$ such that
\[\partial_t\rho  = \Delta\Phi(\rho)-\nabla\cdot\left(\Phi(\rho)\nabla H\right)\;\;\textrm{in}\;\;\TT^d\times(0,T)\;\;\textrm{with}\;\;\rho(\cdot,0)=\rho_0,\]
in the sense of Definition~\ref{classical_weak} with control $g=\Phi^\frac{1}{2}(\rho)\nabla H$.
\end{definition}

\begin{thm}\label{thm:envelope}  Let $\Phi\in\C([0,\infty))\cap\C^3_{\textrm{loc}}((0,\infty))$ satisfy Assumptions~\ref{as_unique}, ~\ref{as_equiv}, and \ref{as_compact}.  Then, for the rate functions defined in \eqref{lsc_0112} and \eqref{lsc_01111},
\[I^{\textrm{\emph{lo}}}=I.\]
That is, $I$ is the l.s.c.\ envelope of $I^0$ restricted to smooth fluctuations $\mcS$ and $I=I^{\textrm{\emph{lo}}}=I^{\textrm{\emph{up}}}$.
\end{thm}

\begin{proof}  It follows from Proposition~\ref{prop_rate_equivalent} and Definition~\ref{smooth_fluctuation} that $I=I^0=I^{\textrm{up}}$ on $\mcS$.  Since then by definition $I^{\textrm{lo}}$ is the l.s.c.\ envelope of $I$ restricted to $\mcS$, it suffices to show that $I$ is equal to the l.s.c.\ envelope of its restriction to $\mcS$.

Since $I$ is itself l.s.c.\ it follows that $I\leq I^{\textrm{lo}}$.  It remains to show that, for any $\rho\in \C([0,T];L^1(\TT^d))$ satisfying $I(\rho)<\infty$, there exists a sequence $\rho_n\in\mathcal{S}$ such that, as $n\rightarrow \infty$,
\[\rho_n\rightarrow\rho\;\;\textrm{strongly in $L^1([0,T];L^1(\TT^d))$}\;\;\textrm{and}\;\;I(\rho_n)\rightarrow I(\rho).\]
Using Lemma~\ref{rf_min}, let $g\in L^2([0,T];L^2(\TT^d))$ be the unique control satisfying
\[\partial_t\rho = \Delta \Phi(\rho)-\nabla\cdot(\Phi^\frac{1}{2}(\rho)g)\;\;\textrm{with}\;\;I(\rho)=\frac{1}{2}\norm{g}^2_{L^2([0,T];L^2(\TT^d;\R^d))}.\]
We will now construct the smooth approximation.  This occurs in three steps:  we first introduce a smoothing of the control $g$ and then a smoothing and perturbation of the initial data $\rho(\cdot,0)\in\Ent_{\Phi}(\TT^d)$ so that it becomes strictly positive.  Finally, we introduce a cutoff that effectively ``turns off'' the control when it forces the solution near zero or infinity.  This guarantees that the approximate solutions remain strictly bounded away from zero and infinity---where the nonlinearity $\Phi$ remains uniformly elliptic---and allows for the application of standard parabolic and elliptic regularity estimates to prove the regularity of the control in the sense of Definition~\ref{smooth_fluctuation}.

\textit{Step 1:  The smooth approximation.}  Extend $g$ to $\TT^d\times\R$ by defining $g=0$ on the complement of $\TT^d\times[0,T]$, for every $\ve\in(0,1)$ let $\kappa_{d+1}^\ve$ be a standard convolution kernel on $\TT^d\times\R$ of scale $\ve$, and for every $n\in\N$ let $g_n$ be the smooth control defined by
\[g_n = (g*\kappa_{d+1}^{\nicefrac{1}{n}}).\]
For every $\ve\in(0,1)$ let $\kappa^\ve_d$ be a smooth convolution kernel on $\TT^d$ and for every $n\in\N$ let $\rho_{0,n}$ be the smooth, bounded, and strictly positive initial data defined by
\[\rho_{0,n} = ((\rho(\cdot,0)\vee\nicefrac{1}{n})\wedge n)*\kappa^{\nicefrac{1}{n}}.\]
For every $n\in\N$ let $\psi_n\colon[0,\infty)\rightarrow[0,1]$ be a smooth function that satisfies
\[\psi_n = 0\;\;\textrm{on}\;\;[0,\nicefrac{1}{2n}]\cup[2n,\infty)\;\;\textrm{and}\;\;\psi_n = 1\;\;\textrm{on}\;\;[\nicefrac{1}{n},n].\]
A small adaptation of the proofs of Theorem~\ref{thm_unique}, Proposition~\ref{L1-continuity}, and Theorem~\ref{equiv} in this simplified setting prove that, for every $n\in\N$, there exists a unique solution let $\rho_n\in\C([0,T];L^1(\TT^d))$ of the equation
\[\partial_t\rho_n  = \Delta\Phi(\rho_n)-\nabla\cdot(\Phi^\frac{1}{2}(\rho_n)\psi_n(\rho_n)g_n)\;\;\textrm{in}\;\;\TT^d\times(0,T)\;\;\textrm{with}\;\;\rho_n(\cdot,0)=\rho_{0,n},\]
in the sense of Definition~\ref{classical_weak} with control $\psi_n(\rho_n)g_n$.  We will first show that the $\rho_n$ correspond to smooth fluctuations in the sense of Definition~\ref{smooth_fluctuation}.  The comparison principle, the definition of $\psi_n$, and the definition of $\rho_{0,n}$ prove that
\[\nicefrac{1}{2n}\leq \rho_n\leq 2n\;\;\textrm{on}\;\;\TT^d\times[0,T].\]
Since we have that $\psi_n$, $g_n$, and $\rho_{0,n}$ are smooth and bounded, that $\Phi\in\C^3_{\textrm{loc}}((0,\infty))$ with $\Phi'$ strictly positive on $(0,\infty)$, and that $\rho_n$ is bounded and bounded away from zero, it follows from interior Schauder estimates (see, for example, Lady\u{z}henskaya, Solonnikov, and Ural'ceva \cite{Lad1967}) that $\rho_n\in\C^{3,2}(\TT^d\times[0,T])$.  We view the $\rho_n$ as satisfying the equation
\[\partial_t\rho_n  = \Delta\Phi(\rho_n)-\nabla\cdot(\Phi^\frac{1}{2}(\rho_n)\tilde{g}_n)\;\;\textrm{in}\;\;\TT^d\times(0,T)\;\;\textrm{with}\;\;\rho_n(\cdot,0)=\rho_{0,n},\]
with the control $\tilde{g}_n = \psi_n(\rho_n)g_n$.  We observe that the positivity and boundedness of $\rho_n$ and positivity of $\Phi$ on $(0,\infty)$ prove that elements of $H^1_{\Phi(\rho_n)}$ can be identified with a unique element of the Sobolev space $L^2([0,T];H^1(\TT^d))$ that has zero spatial mean on almost every time slice.  It then follows from the proof of Proposition~\ref{prop_rate_equivalent} that for every $n\in\N$ there exists $H_n\in L^2([0,T];H^1(\TT^d))$ such that
\[\partial_t\rho_n  = \Delta\Phi(\rho_n)-\nabla\cdot(\Phi(\rho_n)\nabla H_n)\;\;\textrm{in}\;\;\TT^d\times(0,T)\;\;\textrm{with}\;\;\rho_n(\cdot,0)=\rho_{0,n}.\]
Hence, we have for every $n\in\N$ that $H_n$ is a solution to the uniformly elliptic equation
\[-\nabla\cdot (\Phi(\rho_n)\nabla H_n) = \partial_t\rho_n-\Delta\Phi(\rho_n)\;\;\textrm{in}\;\;\TT^d\times(0,T),\]
from which we have from $\nicefrac{1}{2n}\leq \rho_n\leq 2n$, the local $\C^3$-regularity and positivity of $\Phi$ on $(0,\infty)$, the $\C^{3,2}$-regularity of $\rho_n$, and interior elliptic regularity estimates (see, for example, \cite{Lad1967}) that $H_n\in \C^{3,1}(\TT^d\times[0,T])$.  This completes the proof that $\rho_n\in\mathcal{S}$ for every $n\in\N$.

\textit{Step 2: The strong convergence of the smooth approximations.}  It remains to prove that, along a subsequence as $n\rightarrow\infty$,
\[\rho_n\rightarrow\rho\;\;\textrm{strongly in}\;\;L^1([0,T];L^1(\TT^d))\;\;\textrm{and}\;\; I(\rho_n)\rightarrow I(\rho).\]
A small adaptation of Proposition~\ref{gen_exist} in this simplifed setting proves that, for every $n\in\N$, for almost every $t\in[0,T]$, we have that $\norm{\rho_n(\cdot,t)}_{L^1(\TT^d)}=\norm{\rho_{0,n}}_{L^1(\TT^d)}$, and the estimates proven in Proposition~\ref{frac_est} and the definitions of $\rho_{0,n}$, $g_n$, and $\psi_n$ prove that there exists $c\in(0,\infty)$ independent of $n\in\N$ such that
\begin{equation}\label{lsc_40}\int_{\TT^d}\Psi_\Phi(\rho_n(x,t))+\int_0^t\int_{\TT^d}\abs{\nabla\Phi^\frac{1}{2}(\rho_n(x,s))}^2 \leq c\left(\int_{\TT^d}\Psi_\Phi(\rho(x,0))+\norm{g}^2_{L^2(\TT^d\times[0,T];\R^d)}\right).\end{equation}
A repetition of the argument leading to \eqref{ge_8} and the definitions of $\psi_n$ and $g_n$ prove that, for $c\in(0,\infty)$ independent of $n\in\N$,
\begin{equation}\label{lsc_41}  \norm{\partial_t\rho_n}_{L^1([0,T];H^{-(\nicefrac{d}{2}+2)}(\TT^d))} \leq c\left(\norm{\Phi^\frac{1}{2}(\rho_n)}^2_{L^2([0,T];H^1(\TT^d))}+\norm{g}^2_{L^2([0,T];L^2(\TT^d))}\right),\end{equation}
which in view of \eqref{lsc_40} is uniformly bounded in $n\in\N$.

The uniform $L^1(\TT^d)$-boundedness of the $\rho_n$ and estimates \eqref{lsc_40}, \eqref{lsc_41}, and Assumption~\ref{as_interpolate} prove that the $\rho_n$ are relatively compact in $L^1([0,T];L^1(\TT^d))$.  Therefore, after passing to a subsequence still denoted $n\rightarrow\infty$, there exists $\tilde{\rho}\in L^1([0,T];L^1(\TT^d))$ such that
\begin{equation}\label{lsc_42}\rho_n\rightarrow\tilde{\rho}\;\;\textrm{strongly in}\;\;L^1([0,T];L^1(\TT^d))\;\;\textrm{and almost surely on}\;\;\TT^d\times[0,T],\end{equation}
and using Lemma~\ref{as_pre} such that
\begin{equation}\label{lsc_420}\Phi^\frac{1}{2}(\rho_n)\rightarrow\Phi^\frac{1}{2}(\tilde{\rho})\;\;\textrm{weakly in}\;\;L^2([0,T];H^1(\TT^d))\;\;\textrm{and strongly in}\;\;L^2([0,T];L^2(\TT^d)).\end{equation}
Furthermore, it follows from the definition of the $\rho_{0,n}$ that
\begin{equation}\label{lsc_31}\rho_n(\cdot,0)\rightarrow\rho(\cdot,0)\;\;\textrm{strongly in}\;\;L^1(\TT^d).\end{equation}
Since $g_n$ is a convolution of $g$, we have that $g_n\rightarrow g$ strongly in $L^2(\TT^d\times[0,T];\R^d)$ as $n\rightarrow\infty$.  It then follows from the almost sure convergence of \eqref{lsc_42}, the definition and boundedness of $\psi_n$, \eqref{lsc_420}, and H\"older's inequality that, along the subsequence as $n\rightarrow\infty$,
\begin{equation}\label{lsc_43} \Phi^\frac{1}{2}(\rho_n)\psi_n(\rho_n)g_n\mathbf{1}_{\{\tilde{\rho}>0\}} \rightarrow \Phi^\frac{1}{2}(\tilde{\rho})g\mathbf{1}_{\{\tilde{\rho}>0\}}\;\;\textrm{strongly in}\;\;L^1(\TT^d\times[0,T];\R^d). \end{equation}
It follows similarly from the almost sure convergence of \eqref{lsc_42}, the continuity of $\Phi$ and $\Phi(0)=0$, the fact that $g_n\rightarrow g$ strongly as $n\rightarrow\infty$, \eqref{lsc_420}, and the boundedness and definition of $\psi_n$ that, along the subsequence as $n\rightarrow\infty$,
\begin{equation}\label{lsc_430} \Phi^\frac{1}{2}(\rho_n)\psi_n(\rho_n)g_n\mathbf{1}_{\{\tilde{\rho}=0\}}\rightarrow 0\;\;\textrm{strongly in}\;\;L^1(\TT^d\times[0,T];\R^d).\end{equation}
In combination, it follows from \eqref{lsc_43}, \eqref{lsc_430}, and $\Phi(0)=0$ that, along the subsequence as $n\rightarrow\infty$,
\begin{equation}\label{lsc_4300} \Phi^\frac{1}{2}(\rho_n)\psi_n(\rho_n)g_n\rightarrow \Phi^\frac{1}{2}(\tilde{\rho})g\;\;\textrm{strongly in}\;\;L^1(\TT^d\times[0,T];\R^d).\end{equation}
Since \eqref{lsc_420} and the weak lower semicontinuity of the Sobolev norm prove that $\tilde{\rho}$ satisfies estimate \eqref{lsc_40}, it follows from Definition~\ref{classical_weak}, \eqref{lsc_42}, \eqref{lsc_31}, and \eqref{lsc_4300} that $\tilde{\rho}$ is a solution of the equation
\[\partial_t\tilde{\rho} = \Delta\Phi(\tilde{\rho})-\nabla\cdot(\Phi^\frac{1}{2}(\tilde{\rho})g)\;\;\textrm{in}\;\;\TT^d\times(0,T)\;\;\textrm{with}\;\;\tilde{\rho}(\cdot,0)=\rho(\cdot,0),\]
in the sense of Definition~\ref{classical_weak}.  We therefore conclude using Theorem~\ref{thm_unique} and Theorem~\ref{equiv} that $\tilde{\rho}=\rho$ and that the $\rho_n$ converge strongly to $\rho$ along the subsequence $n\rightarrow\infty$.

\textit{Step 3: Convergence of the rate functions.}  It remains only to show that, along the subsequence $n\rightarrow\infty$, we have that
\[I(\rho_n)\rightarrow I(\rho).\]
By Lemma~\ref{rf_min} let $h_n\in L^2(\TT^d\times[0,T];\R^d)$ be the unique function satisfying, using the definitions of the rate function, $\psi_n$, and $g_n$,
\begin{equation}\label{lsc_45}I(\rho_n)=\norm{h_n}^2_{L^2}\leq \frac{1}{2}\norm{\psi_n(\rho_n)g_n}^2_{L^2}\leq \frac{1}{2}\norm{g_n}^2_{L^2}\;\;\textrm{with}\;\;\lim_{n\rightarrow\infty}\norm{g_n}^2_{L^2}=\norm{g}^2_{L^2}.\end{equation}
After passing to a subsequence $n\rightarrow\infty$, for some $h\in L^2([0,T];L^2(\TT^d))$,
\[h_n\rightharpoonup h\;\;\textrm{weakly in}\;\;L^2([0,T];L^2(\TT^d;\R^d)),\]
from which it follows from \eqref{lsc_45} that $\norm{h}_{L^2}\leq\norm{g}_{L^2}$.  Since Proposition~\ref{weak_strong} proves that $\rho$ solves
\[\partial_t\rho = \Delta \Phi(\rho)-\nabla\cdot(\Phi^\frac{1}{2}(\rho)h)\;\;\textrm{in}\;\;\TT^d\times(0,\infty),\]
we have by definition of the rate function that $\norm{h}_{L^2}\geq \norm{g}_{L^2}$ and therefore that $\norm{h}_{L^2}= \norm{g}_{L^2}$.  Lemma~\ref{rf_min} proves that that $h=g$, which implies that
\[\norm{g}_{L^2}=\norm{h}_{L^2}\leq\liminf_{n\rightarrow\infty}\norm{h_n}_{L^2}\leq\limsup_{n\rightarrow\infty}\norm{h_n}_{L^2}\leq\limsup_{n\rightarrow\infty}\norm{g_n}_{L^2}= \norm{g}_{L^2},\]
and therefore, as $n\rightarrow\infty$,
\[\lim_{n\rightarrow\infty}I(\rho_n)=\lim_{n\rightarrow\infty}\frac{1}{2}\norm{h_n}^2_{L^2} = \frac{1}{2}\norm{g}^2_{L^2}=I(\rho),\]
which completes the proof that $I=I^{\textrm{lo}}$.  The final claim is then a consequence of Proposition~\ref{prop_rate_equivalent}.\end{proof}

\section*{Acknowledgements}

The first author acknowledges financial support from the EPSRC through the EPSRC Early Career Fellowship EP/V027824/1.  This work was funded by the Deutsche Forschungsgemeinschaft (DFG, German Research Foundation) -- SFB 1283/2 2021 -- 317210226.  The authors thank Claudio Landim for pointing out the works \cite{La.Ts2018,Qu.Re.Va1999}.

\bibliography{WhiteNoise}
\bibliographystyle{plain}

\end{document}